\documentclass[11pt]{amsart}
\usepackage{amsmath,amssymb,amsthm,enumitem,url,tikz-cd}
\usepackage[hmargin=20mm,vmargin=30mm]{geometry}

\setlist[enumerate]{leftmargin=*}
\setlist[itemize]{labelindent=\parindent, leftmargin=*}

\numberwithin{equation}{section}

\theoremstyle{plain}
\newtheorem{thm}{Theorem}[section]
\newtheorem{lem}[thm]{Lemma}
\newtheorem{prop}[thm]{Proposition}

\theoremstyle{definition}

\theoremstyle{remark}
\newtheorem{rem}[thm]{Remark}

\newcommand\Hom{\operatorname{Hom}}
\newcommand\Ind{\operatorname{Ind}}
\newcommand\JH{\operatorname{JH}}
\renewcommand\Re{\operatorname{Re}}
\newcommand\Span{\operatorname{Span}}

\newcommand\cusp{\mathrm{cusp}}
\newcommand\disc{\mathrm{disc}}
\newcommand\gen{\mathrm{gen}}
\newcommand\GL{\mathrm{GL}}
\newcommand\Mp{\mathrm{Mp}}
\renewcommand\O{\mathrm{O}}
\newcommand\res{\mathrm{res}}
\newcommand\SL{\mathrm{SL}}
\newcommand\SO{\mathrm{SO}}
\newcommand\Sp{\mathrm{Sp}}
\newcommand\st{\mathrm{st}}
\newcommand\St{\mathrm{St}}
\newcommand\Sym{\mathrm{Sym}}
\newcommand\U{\mathrm{U}}

\newcommand\A{\mathbb{A}}
\newcommand\C{\mathbb{C}}
\newcommand\F{\mathbb{F}}

\newcommand\R{\mathbb{R}}

\newcommand\Z{\mathbb{Z}}

\newcommand\g{\mathfrak{g}}
\newcommand\q{\mathfrak{q}}

\newcommand\VV{\mathcal{V}}

\newcommand\longhookrightarrow{\lhook\joinrel\longrightarrow}
\newcommand\longtwoheadrightarrow{\relbar\joinrel\twoheadrightarrow}

\DeclareMathOperator*{\Res}{Res}

\title{The automorphic discrete spectrum of $\Mp_4$}
\author{Wee Teck Gan}
\address{Department of Mathematics, National University of Singapore, 10 Lower Kent Ridge Road, Singapore 119076}
\email{matgwt@nus.edu.sg}
\author{Atsushi Ichino}
\address{Department of Mathematics, Kyoto University, Kitashirakawa Oiwake-cho, Sakyo-ku, Kyoto 606-8502, Japan}
\email{ichino@math.kyoto-u.ac.jp}

\begin{document}

\maketitle

\begin{abstract}
We prove the multiplicity formula for the automorphic discrete spectrum of the metaplectic group $\Mp_4$ of rank $2$.
\end{abstract}

\tableofcontents

\section{Introduction}

In our previous paper \cite{gi-mp}, we studied the automorphic discrete spectrum $L^2_\disc(\Mp_{2n})$ of the metaplectic group $\Mp_{2n}$, which is a nonlinear double cover of the symplectic group $\Sp_{2n}$ of rank $n$.
In particular, we proved a decomposition
\[
 L^2_{\disc}(\Mp_{2n}) = \bigoplus_\phi L^2_\phi(\Mp_{2n}), 
\]
where $\phi$ runs over elliptic $A$-parameters for $\Mp_{2n}$ and $L^2_\phi(\Mp_{2n})$ is the near equivalence class determined by $\phi$ (which depends on an auxiliary choice of an additive character).
In addition, when $\phi$ is tempered, we proved a further decomposition
\begin{equation}
\label{eq:main-intro}
 L^2_\phi(\Mp_{2n}) \cong \bigoplus_\pi m_\pi \pi, 
\end{equation}
where $\pi$ runs over representations in the global $A$-packet associated to $\phi$ and $m_\pi$ is the nonnegative integer predicted by the analog of Arthur's conjecture \cite[\S 5.6]{g-iccm}.
When $n=1$, such a result was first established by Waldspurger \cite{w1,w2} as a representation-theoretic reformulation of the Shimura correspondence.

The purpose of this paper is to prove \eqref{eq:main-intro} when $n=2$ and $\phi$ is nontempered (see Theorem \ref{t:main}), which completes the analysis of $L^2_\disc(\Mp_4)$.
We also tabulate the representations in the local $A$-packets for $\Mp_4$ explicitly (see Appendix \ref{a:A-packets}), which may be useful for arithmetic applications. 
The case $n=2$ is especially interesting because, similarly to the case $n=1$, it involves the two groups $\SO_5$ and $\Sp_4$ of the same type $B_2 = C_2$.
In particular, we hope that our representation-theoretic formulation will elucidate Ibukiyama's conjecture \cite{ibukiyama} on integral and half-integral weight Siegel modular forms of genus $2$.

\subsection*{Acknowledgments}

The first-named author is partially supported by an MOE Tier one grant R-146-000-228-114.
The second-named author is partially supported by JSPS KAKENHI Grant Number 19H01781.

\subsection*{Notation}

If $F$ is a local field of characteristic zero, we fix a nontrivial additive character $\psi$ of $F$.
For $a \in F^\times$, we define another nontrivial additive character $\psi_a$ of $F$ by $\psi_a(x) = \psi(ax)$.
We also define a quadratic character $\chi_a$ of $F^\times$ by $\chi_a(x) = (a,x)_F$, where $(\cdot,\cdot)_F$ is the quadratic Hilbert symbol of $F$.
Let $W_F$ be the Weil group of $F$ and put 
\[
 L_F =
 \begin{cases}
  W_F \times \SL_2(\C) & \text{if $F$ is nonarchimedean;} \\
  W_F & \text{if $F$ is archimedean.}
 \end{cases}
\]
For any linear algebraic group $G$ over $F$, we identify $G$ with its group of $F$-rational points.

If $F$ is a number field with ad\`ele ring $\A$, we fix a nontrivial additive character $\psi$ of $\A/F$.
For $a \in F^\times$, we define another nontrivial additive character $\psi_a$ of $\A/F$ by $\psi_a(x) = \psi(ax)$.
We also define a quadratic automorphic character $\chi_a$ of $\A^\times$ by $\chi_a(x) = \prod_v (a,x_v)_{F_v}$, where $v$ runs over places of $F$.
We denote by $\zeta^S(s)$ the partial zeta function of $F$, where $S$ is a finite set of places of $F$.
Put 
\[
 \mu_{2,\A} = \Big\{ (\epsilon_v) \in \prod_v \{ \pm 1 \} \, \Big| \, \text{$\epsilon_v = 1$ for almost all $v$} \Big\}.
\]

For any representation $\pi$, we denote by $\pi^\vee$ its contragredient representation.
For any abelian locally compact group $S$, we denote by $\widehat{S}$ the group of continuous characters of $S$.
For any positive integer $d$, we denote by $S_d$ the unique $d$-dimensional irreducible representation of $\SL_2(\C)$.
We write
\[
 1^d = \underbrace{1,1,\dots,1}_d. 
\]
For any $a \in \frac{1}{2} \Z$ with $a>0$, we denote by $\mathcal{D}_a$ the $2$-dimensional irreducible representation of $W_\R$ induced from the character $z \mapsto (z/\bar{z})^a$ of $W_\C = \C^\times$.

\section{The multiplicity formula}

In this section, we state the multiplicity formula for the automorphic discrete spectrum of $\Mp_4$, which is the main result of this paper.

\subsection{Elliptic $A$-parameters}

Let $F$ be a number field with ad\`ele ring $\A$ and fix a nontrivial additive character $\psi$ of $\A/F$.
Recall from \cite{gi-mp} that an elliptic $A$-parameter for $\Mp_4$ is a formal unordered finite direct sum
\[
 \phi = \bigoplus_i \phi_i \boxtimes S_{d_i}, 
\]
where
\begin{itemize}
\item $\phi_i$ is an irreducible self-dual cuspidal automorphic representation of $\GL_{n_i}(\A)$;
\item $S_{d_i}$ is the unique $d_i$-dimensional irreducible representation of $\SL_2(\C)$;
\item if $d_i$ is odd, then $\phi_i$ is symplectic, i.e.~the exterior square $L$-function $L(s, \phi_i, \wedge^2)$ has a pole at $s=1$;
\item if $d_i$ is even, then $\phi_i$ is orthogonal, i.e.~the symmetric square $L$-function $L(s, \phi_i, \Sym^2)$ has a pole at $s=1$;
\item if $(\phi_i, d_i) = (\phi_j, d_j)$, then $i=j$;
\item $\sum_i n_i d_i  = 4$.
\end{itemize}
Let $S_\phi$ be the global component group of $\phi$, which is defined formally as a free $\Z/2\Z$-module
\[
 S_\phi = \bigoplus_i (\Z/2\Z) a_i
\]
with a basis $\{a_i\}$, where $a_i$ corresponds to $\phi_i \boxtimes S_{d_i}$.
Recall also that Arthur \cite[(1.5.6)]{arthur} associated to $\phi$ (which can be regarded as an elliptic $A$-parameter for $\SO_5$) a character $\epsilon_\phi$ of $S_\phi$, which plays an important role in the multiplicity formula for the automorphic discrete spectrum of $\SO_5$.
We define another character $\tilde{\epsilon}_\phi$ of $S_\phi$, which plays a similar role for $\Mp_4$, by
\[
 \tilde{\epsilon}_\phi(a_i) = \epsilon_\phi(a_i) \times 
 \begin{cases}
  \epsilon(\frac{1}{2}, \phi_i) & \text{if $\phi_i$ is symplectic;} \\
  1 & \text{if $\phi_i$ is orthogonal,} \\
 \end{cases}
\]
where $\epsilon(\frac{1}{2}, \phi_i) = \pm 1$ is the root number of $\phi_i$.

We can enumerate the elliptic $A$-parameters $\phi$ for $\Mp_4$ with associated character $\tilde{\epsilon}_\phi$ as follows.
\begin{itemize}
\item We say that $\phi$ is \emph{tempered} if $d_i=1$ for all $i$.
In this case, we have $\tilde{\epsilon}_\phi(a_i) = \epsilon(\frac{1}{2}, \phi_i)$ for all $i$.
\item We say that $\phi$ is of \emph{Saito--Kurokawa type} if 
\[
 \phi = (\rho \boxtimes S_1) \oplus (\chi \boxtimes S_2)
\]
for some irreducible cuspidal automorphic representation $\rho$ of $\GL_2(\A)$ with trivial central character and some quadratic automorphic character $\chi$ of $\A^\times$.
In this case, we have $S_\phi \cong (\Z/2\Z)^2$ and 
\[
 \tilde{\epsilon}_\phi(a_1) = \epsilon(\tfrac{1}{2}, \rho) \cdot \epsilon(\tfrac{1}{2}, \rho \times \chi), \qquad
 \tilde{\epsilon}_\phi(a_2) = \epsilon(\tfrac{1}{2}, \rho \times \chi), 
\]
where $a_1$ and $a_2$ correspond to $\rho \boxtimes S_1$ and $\chi \boxtimes S_2$, respectively.
\item We say that $\phi$ is of \emph{Howe--Piatetski-Shapiro type} if 
\[
 \phi = (\chi_1 \boxtimes S_2) \oplus (\chi_2 \boxtimes S_2)
\]
for some distinct quadratic automorphic characters $\chi_1$ and $\chi_2$ of $\A^\times$.
In this case, we have $S_\phi \cong (\Z/2\Z)^2$ and $\tilde{\epsilon}_\phi = 1$.
\item We say that $\phi$ is of \emph{Soudry type} if 
\[
 \phi = \rho \boxtimes S_2
\]
for some irreducible dihedral cuspidal automorphic representation $\rho$ of $\GL_2(\A)$ with nontrivial quadratic central character.
In this case, we have $S_\phi \cong \Z/2\Z$ and $\tilde{\epsilon}_\phi = 1$.
\item We say that $\phi$ is \emph{principal} if 
\[
 \phi = \chi \boxtimes S_4
\]
for some quadratic automorphic character $\chi$ of $\A^\times$.
In this case, we have $S_\phi \cong \Z/2\Z$ and $\tilde{\epsilon}_\phi = 1$.
\end{itemize}

\subsection{Local $A$-packets}

Let $\phi = \bigoplus_i \phi_i \boxtimes S_{d_i}$ be an elliptic $A$-parameter for $\Mp_4$.
For each place $v$ of $F$, the localization $\phi_v = \bigoplus_i \phi_{i,v} \boxtimes S_{d_i}$ of $\phi$ at $v$ gives rise to a local $A$-parameter
\[
 \phi_v : L_{F_v} \times \SL_2(\C) \longrightarrow \Sp_4(\C) 
\]
via the local Langlands correspondence \cite{langlands,harris-taylor,henniart,scholze}.
We define the associated $L$-parameter $\varphi_{\phi_v} : L_{F_v} \rightarrow \Sp_4(\C)$ by 
\[
 \varphi_{\phi_v}(w) = \phi_v \! 
 \left( w,
 \begin{pmatrix}
  |w|_v^{\frac{1}{2}} & \\
  & |w|_v^{-\frac{1}{2}}
 \end{pmatrix}
 \right).
\]
We denote by $S_{\phi_v}$ and $S_{\varphi_{\phi_v}}$ the component groups of the centralizers of $\phi_v$ and $\varphi_{\phi_v}$ in $\Sp_4(\C)$, respectively. 
Then we have a canonical map $S_\phi \rightarrow S_{\phi_v}$ and a natural surjection $S_{\phi_v} \rightarrow S_{\varphi_{\phi_v}}$.

Let $\Mp_4(F_v)$ be the metaplectic double cover of $\Sp_4(F_v)$.
We will assign to $\phi_v$ a finite set (which depends on the nontrivial additive character $\psi_v$ of $F_v$)
\[
 \Pi_{\phi_v,\psi_v}(\Mp_4) = \{ \pi_{\eta_v} \, | \, \eta_v \in \widehat{S}_{\phi_v} \}
\]
of semisimple genuine representations of $\Mp_4(F_v)$ of finite length indexed by characters of $S_{\phi_v}$.
If $\phi$ is tempered, then $\Pi_{\phi_v,\psi_v}(\Mp_4)$ is defined as the $L$-packet associated to $\phi_v$ (relative to $\psi_v$) defined via the local Shimura correspondence \cite{ab1,ab2,gs}.
If $\phi$ is nontempered, then $\Pi_{\phi_v,\psi_v}(\Mp_4)$ will be defined in \S\S \ref{ss:local-princ}, \ref{ss:local-sk}, \ref{ss:local-howe-ps}, \ref{ss:local-soudry} below.
Moreover, we will show that the following properties hold:
\begin{itemize}
\item $\Pi_{\phi_v,\psi_v}(\Mp_4)$ is multiplicity-free;
\item $\Pi_{\phi_v,\psi_v}(\Mp_4)$ contains the $L$-packet $\Pi_{\varphi_{\phi_v}, \psi_v}(\Mp_4)$ associated to $\varphi_{\phi_v}$ (relative to $\psi_v$);
\item the diagram
\[
\begin{tikzcd}
 \Pi_{\phi_v, \psi_v}(\Mp_4) \arrow[r] & \widehat{S}_{\phi_v} \\
 \Pi_{\varphi_{\phi_v}, \psi_v}(\Mp_4) \arrow[r] \arrow[u,hook] & \widehat{S}_{\varphi_{\phi_v}} \arrow[u,hook]
\end{tikzcd}
\]
commutes.
\end{itemize}

\subsection{Statement of the main theorem}

Let $\Mp_4(\A)$ be the metaplectic double cover of $\Sp_4(\A)$.
We denote by 
\[
 L^2_{\disc}(\Mp_4)
\]
the genuine part of the discrete spectrum of the unitary representation $L^2(\Sp_4(F) \backslash \Mp_4(\A))$ of $\Mp_4(\A)$, where we regard $\Sp_4(F)$ as a subgroup of $\Mp_4(\A)$ via the canonical splitting.
In our previous paper \cite{gi-mp}, we proved a decomposition
\[
 L^2_{\disc}(\Mp_4) = \bigoplus_{\phi} L^2_{\phi,\psi}(\Mp_4),
\]
where $\phi$ runs over elliptic $A$-parameters for $\Mp_4$ and $L^2_{\phi,\psi}(\Mp_4)$ is the near equivalence class of irreducible representations $\pi$ in $L^2_{\disc}(\Mp_4)$ such that the $L$-parameter of $\pi_v$ (relative to $\psi_v$) is $\varphi_{\phi_v}$ for almost all places $v$ of $F$.
Thus, to understand the spectral decomposition of $L^2_{\disc}(\Mp_4)$, it remains to describe the multiplicity of any irreducible representation of $\Mp_4(\A)$ in $L^2_{\phi,\psi}(\Mp_4)$.

Consider the compact group $S_{\phi, \A} = \prod_v S_{\phi_v}$ and the group $\widehat{S}_{\phi,\A} = \bigoplus_v \widehat{S}_{\phi_v}$ of continuous characters of $S_{\phi, \A}$.
For any $\eta = \bigotimes_v \eta_v \in \widehat{S}_{\phi, \A}$, we may form a semisimple genuine representation
\[
 \pi_\eta = \bigotimes_v \pi_{\eta_v}
\]
of $\Mp_4(\A)$.
Let $\Delta^* \eta = \eta \circ \Delta \in \widehat{S}_\phi$ be the pullback of $\eta$ under the diagonal map $\Delta : S_\phi \rightarrow S_{\phi, \A}$.

We now state our main result:

\begin{thm}
\label{t:main}
For any elliptic $A$-parameter $\phi$ for $\Mp_4$, we have
\[
 L^2_{\phi,\psi}(\Mp_4) \cong \bigoplus_{\eta \in \widehat{S}_{\phi, \A}} m_\eta \pi_\eta,
\]
where
\[
 m_\eta = 
 \begin{cases}
  1 & \text{if $\Delta^* \eta = \tilde{\epsilon}_\phi$;} \\
  0 & \text{otherwise.}
 \end{cases}
\]
In particular, $L^2_{\disc}(\Mp_4)$ is multiplicity-free.
\end{thm}

This theorem was already proved in our previous paper \cite{gi-mp} when $\phi$ is tempered and follows from Propositions \ref{p:princ}, \ref{p:sk}, \ref{p:howe-ps}, \ref{p:soudry} below when $\phi$ is nontempered.

\section{Metaplectic and orthogonal groups}

In this section, we introduce notation for the groups and representations which appear in this paper.
Let $F$ be either a local field of characteristic zero or a number field with ad\`ele ring $\A$.

\subsection{Metaplectic groups}
\label{ss:mp}

Let $W_n$ be a $2n$-dimensional vector space over $F$ equipped with a nondegenerate antisymmetric bilinear form $\langle \cdot, \cdot \rangle_{W_n}$.
We call such $W_n$ a symplectic space over $F$ and denote by $\Sp(W_n)$ the associated symplectic group.
Choose a basis $w_1,\dots,w_n,w_1^*,\dots,w_n^*$ of $W_n$ such that
\[
 \langle w_i, w_j \rangle_{W_n} = \langle w_i^*, w_j^* \rangle_{W_n} = 0, \qquad \langle w_i, w_j^* \rangle_{W_n} = \delta_{ij}.
\]
Given a sequence $\mathbf{k} = (k_1, \dots, k_m)$ of positive integers such that $k_1 + \dots + k_m \le n$, we write 
\[
 W_n = X_1 \oplus \cdots \oplus X_m \oplus X_1^* \oplus \cdots \oplus X_m^* \oplus W_{n_0}
\]
with 
\begin{align*}
 X_i & = \Span(w_{k'_{i-1} + 1}, \dots, w_{k'_i}), \\ 
 X_i^* & = \Span(w_{k'_{i-1} + 1}^*, \dots, w_{k'_i}^*), \\
 W_{n_0} & = \Span(w_{k'_m+1}, \dots, w_n, w_{k'_m+1}^*, \dots, w_n^*),
\end{align*}
where $k'_i = k_1 + \dots + k_i$ and $n_0 = n - k_1 - \dots - k_m$.
Let $P_{\mathbf{k}} = M_{\mathbf{k}} N_{\mathbf{k}}$ be the parabolic subgroup of $\Sp(W_n)$ stabilizing the flag
\[
 X_1 \subset X_1 \oplus X_2 \subset \cdots \subset X_1 \oplus \cdots \oplus X_m, 
\]
where 
\[
 M_{\mathbf{k}} \cong \GL(X_1) \times \dots \times \GL(X_m) \times \Sp(W_{n_0}) 
\]
is the Levi component of $P_{\mathbf{k}}$ stabilizing the flag
\[
 X_1^* \subset X_1^* \oplus X_2^* \subset \cdots \subset X_1^* \oplus \cdots \oplus X_m^* 
\]
and $N_{\mathbf{k}}$ is the unipotent radical of $P_{\mathbf{k}}$.
We also write $P_k = P_{(k)}$ for $1 \le k \le n$ and $B = P_{(1^n)}$, which are a maximal parabolic subgroup and a Borel subgroup of $\Sp(W_n)$, respectively.

Suppose that $F$ is local.
We denote by $\Mp(W_n)$ the metaplectic double cover of $\Sp(W_n)$:
\[
 1 \longrightarrow \{ \pm 1 \} \longrightarrow \Mp(W_n) \longrightarrow \Sp(W_n) \longrightarrow 1.
\]
This cover splits over $N_{\mathbf{k}}$ uniquely.
Let $\widetilde{P}_{\mathbf{k}}$ and $\widetilde{M}_{\mathbf{k}}$ be the preimages of $P_{\mathbf{k}}$ and $M_{\mathbf{k}}$ in $\Mp(W_n)$, respectively, so that $\widetilde{P}_{\mathbf{k}} = \widetilde{M}_{\mathbf{k}} N_{\mathbf{k}}$ and 
\[
 \widetilde{M}_{\mathbf{k}} \cong \widetilde{\GL}(X_1) \times_{\{ \pm 1 \} } \dots \times_{\{ \pm 1 \} } \widetilde{\GL}(X_m) \times_{\{ \pm 1 \} }  \Mp(W_{n_0}).
\]
Here $\widetilde{\GL}(X_i)$ is the double cover of $\GL(X_i)$ given in \cite[\S 2.5]{gi1}.
Suppose that $F$ is global.
We denote by $\Mp(W_n)(\A)$ the metaplectic double cover of $\Sp(W_n)(\A)$:
\[
 1 \longrightarrow \{ \pm 1 \} \longrightarrow \Mp(W_n)(\A) \longrightarrow \Sp(W_n)(\A) \longrightarrow 1.
\]
This cover splits over $\Sp(W_n)(F)$ uniquely.

\subsection{Orthogonal groups}
\label{ss:so}

Let $V_n$ be a $(2n+1)$-dimensional vector space over $F$ equipped with a nondegenerate symmetric bilinear form $\langle \cdot, \cdot \rangle_{V_n}$.
We call such $V_n$ a quadratic space over $F$ and denote by $\O(V_n)$ the associated orthogonal group.
Let $\SO(V_n)$ be the identity component of $\O(V_n)$, so that 
\[
 \O(V_n) = \SO(V_n) \times \{ \pm 1\}.
\]
Let $V_{\mathrm{an}}$ be an anisotropic kernel of $V_n$ and $r$ the Witt index of $V_n$.
Choose a basis $v_1,\dots,v_r,v_1^*,\dots,v_r^*$ of the orthogonal complement of $V_{\mathrm{an}}$ in $V_n$ such that
\[
 \langle v_i, v_j \rangle_{V_n} = \langle v_i^*, v_j^* \rangle_{V_n} = 0, \qquad \langle v_i, v_j^* \rangle_{V_n} = \delta_{ij}.
\]
Given a sequence $\mathbf{k} = (k_1, \dots, k_m)$ of positive integers such that $k_1 + \dots + k_m \le r$, we write 
\[
 V_n = Y_1 \oplus \cdots \oplus Y_m \oplus Y_1^* \oplus \cdots \oplus Y_m^* \oplus V_{n_0}
\]
with 
\begin{align*}
 Y_i & = \Span(v_{k'_{i-1} + 1}, \dots, v_{k'_i}), \\ 
 Y_i^* & = \Span(v_{k'_{i-1} + 1}^*, \dots, v_{k'_i}^*), \\
 V_{n_0} & = \Span(v_{k'_m+1}, \dots, v_r, v_{k'_m+1}^*, \dots, v_r^*) \oplus V_{\mathrm{an}},
\end{align*}
where $k'_i = k_1 + \dots + k_i$ and $n_0 = n - k_1 - \dots - k_m$.
Let $Q_{\mathbf{k}} = L_{\mathbf{k}} U_{\mathbf{k}}$ be the parabolic subgroup of $\SO(V_n)$ stabilizing the flag
\[
 Y_1 \subset Y_1 \oplus Y_2 \subset \cdots \subset Y_1 \oplus \cdots \oplus Y_m, 
\]
where 
\[
 L_{\mathbf{k}} \cong \GL(Y_1) \times \dots \times \GL(Y_m) \times \SO(V_{n_0}) 
\]
is the Levi component of $Q_{\mathbf{k}}$ stabilizing the flag
\[
 Y_1^* \subset Y_1^* \oplus Y_2^* \subset \cdots \subset Y_1^* \oplus \cdots \oplus Y_m^* 
\]
and $U_{\mathbf{k}}$ is the unipotent radical of $Q_{\mathbf{k}}$.
We also write $Q_k = Q_{(k)}$ for $1 \le k \le r$, which is a maximal parabolic subgroup of $\SO(V_n)$.

We denote by $V_n^+$ the unique (up to isometry) $(2n+1)$-dimensional quadratic space over $F$ with Witt index $n$ and trivial discriminant.
Note that $\SO(V_n^+)$ is split over $F$.
If $F$ is local, then the isometry classes of $(2n+1)$-dimensional quadratic spaces over $F$ with trivial discriminant are classified as follows.
\begin{itemize}
\item
When $n=0$ or $F = \C$, there is a unique such quadratic space $V_n^+$.
\item 
When $n>0$ and $F$ is nonarchimedean, or $n=1$ and $F=\R$, there are precisely two such quadratic spaces $V_n^+$ and $V_n^-$.
Here $V_n^-$ is the $(2n+1)$-dimensional quadratic space over $F$ with Witt index $n-1$ and trivial discriminant.
\item 
When $F = \R$, there are precisely $n+1$ such quadratic spaces $V_{p,q}$, where $p$ and $q$ are nonnegative integers such that $p+q=2n+1$ and $q \equiv n \bmod 2$.
Here $V_{p,q}$ is the quadratic space over $\R$ of signature $(p,q)$.
Note that $V_n^+ = V_{n+1,n}$ and $V_1^- = V_{0,3}$.
\end{itemize}

\subsection{Parabolic induction}

Suppose that $F$ is local and fix a nontrivial additive character $\psi$ of $F$.
Let $\widetilde{P}_{\mathbf{k}} = \widetilde{M}_{\mathbf{k}} N_{\mathbf{k}}$ be the parabolic subgroup of $\Mp(W_n)$ as in \S \ref{ss:mp}.
Then any irreducible genuine representation of $\widetilde{M}_{\mathbf{k}}$ is of the form 
\[
 (\tau_1 \otimes \chi_\psi) \boxtimes \dots \boxtimes (\tau_m \otimes \chi_\psi) \boxtimes \pi,
\]
where $\tau_i$ is an irreducible representation of $\GL(X_i)$, $\chi_\psi$ is the genuine character of $\widetilde{\GL}(X_i)$ defined in terms of the Weil index associated to $\psi$ (see \cite[\S 2.6]{gi1}), and $\pi$ is an irreducible genuine representation of $\Mp(W_{n_0})$.
We write 
\[
 I_{P_{\mathbf{k}}, \psi}(\tau_1, \dots, \tau_m, \pi) = 
 \Ind^{\Mp(W_n)}_{\widetilde{P}_{\mathbf{k}}}((\tau_1 \otimes \chi_\psi) \boxtimes \dots \boxtimes (\tau_m \otimes \chi_\psi) \boxtimes \pi)
\]
for the associated parabolically induced representation.
Note that $I_{P_{\mathbf{k}}, \psi}(\tau_1, \dots, \tau_m, \pi)^\vee = I_{P_{\mathbf{k}}, \psi^{-1}}(\tau_1^\vee, \dots, \tau_m^\vee, \pi^\vee)$ and
\[
 I_{P_{\mathbf{k}}, \psi_a}(\tau_1, \dots, \tau_m, \pi) = I_{P_{\mathbf{k}}, \psi}(\tau_1 \otimes (\chi_a \circ \det), \dots, \tau_m \otimes (\chi_a \circ \det), \pi)
\]
for $a \in F^\times$.
If $I_{P_{\mathbf{k}}, \psi}(\tau_1, \dots, \tau_m, \pi)$ is a standard module, then we denote by $J_{P_{\mathbf{k}}, \psi}(\tau_1, \dots, \tau_m, \pi)$ its unique irreducible quotient.

Similarly, for the parabolic subgroup $Q_{\mathbf{k}}$ of $\SO(V_n)$ as in \S \ref{ss:so}, we write
\[
 I_{Q_{\mathbf{k}}}(\tau_1, \dots, \tau_m, \sigma) = 
 \Ind^{\SO(V_n)}_{Q_{\mathbf{k}}}(\tau_1 \boxtimes \dots \boxtimes \tau_m \boxtimes \sigma)
\]
for the parabolically induced representation associated to irreducible representations $\tau_i$ and $\sigma$ of $\GL(Y_i)$ and $\SO(V_{n_0})$, respectively.
If $I_{Q_{\mathbf{k}}}(\tau_1, \dots, \tau_m, \sigma)$ is a standard module, then we denote by $J_{Q_{\mathbf{k}}}(\tau_1, \dots, \tau_m, \sigma)$ its unique irreducible quotient.

\section{Local theta lifts}

In this section, we introduce local theta lifts and recall some of their basic properties. 

\subsection{Notation}

Let $F$ be a local field of characteristic zero and fix a nontrivial additive character $\psi$ of $F$.
Let $W_n$ be a $2n$-dimensional symplectic space over $F$ and $V_m$ a $(2m+1)$-dimensional quadratic space over $F$ with trivial discriminant.
We denote by 
\[
 \omega_{W_n,V_m,\psi}
\]
the Weil representation of $\Mp(W_n) \times \O(V_m)$ with respect to $\psi$.
Note that $\omega_{W_n,V_m,\psi}$ depends only on the $(F^\times)^2$-orbit of $\psi$.

For any irreducible genuine representation $\pi$ of $\Mp(W_n)$, the maximal $\pi$-isotypic quotient of $\omega_{W_n,V_m,\psi}$ is of the form
\[
 \pi \boxtimes \Theta_{W_n,V_m,\psi}(\pi)
\]
for some smooth representation $\Theta_{W_n,V_m,\psi}(\pi)$ of $\O(V_m)$ of finite length.
We denote by 
\[
 \theta_{W_n,V_m,\psi}(\pi)
\]
the maximal semisimple quotient of $\Theta_{W_n,V_m,\psi}(\pi)$.
Then, by the Howe duality \cite{howe,w3,gt2}, $\theta_{W_n,V_m,\psi}(\pi)$ is either zero or irreducible. 
We also regard $\Theta_{W_n,V_m,\psi}(\pi)$ and $\theta_{W_n,V_m,\psi}(\pi)$ as representations of $\SO(V_m)$ via restriction.
Note that if $\theta_{W_n,V_m,\psi}(\pi)$ is nonzero, then it remains irreducible as a representation of $\SO(V_m)$.

Similarly, for any irreducible representation $\sigma$ of $\O(V_m)$, we may define a smooth representation $\Theta_{W_n,V_m,\psi}(\sigma)$ of $\Mp(W_n)$ of finite length and its maximal semisimple quotient $\theta_{W_n,V_m,\psi}(\sigma)$, which is either zero or irreducible. 
Let $\sigma_0$ be an irreducible representation of $\SO(V_m)$.
For $\epsilon = \pm 1$, let $\sigma^\epsilon$ be the extension of $\sigma_0$ to $\O(V_m)$ such that $-1 \in \O(V_m)$ acts as the scalar $\epsilon$, which we call the $\epsilon$-extension of $\sigma_0$.
If $m \ge n$, then by the conservation relation \cite{sz}, there is at most one $\epsilon$ such that $\theta_{W_n,V_m,\psi}(\sigma^\epsilon)$ is nonzero.
We write $\theta_{W_n,V_m,\psi}(\sigma_0) = \theta_{W_n,V_m,\psi}(\sigma^\epsilon)$ if such $\epsilon$ exists, and interpret $\theta_{W_n,V_m,\psi}(\sigma_0)$ as zero otherwise.

\subsection{Elementary Weil representations}
\label{ss:e-weil}

Consider the Weil representation $\omega_{W_n,V_0^+,\psi}$ of $\Mp(W_n) \times \O(V_0^+)$ with respect to $\psi$.
We may decompose it as
\[
 \omega_{W_n, V_0^+, \psi} = (\omega_{W_n, \psi}^+ \boxtimes 1) \oplus (\omega_{W_n, \psi}^- \boxtimes \det), 
\]
where $\omega_{W_n, \psi}^+ = \Theta_{W_n,V_0^+,\psi}(1)$ and $\omega_{W_n, \psi}^- = \Theta_{W_n,V_0^+,\psi}(\det)$ are the big theta lifts of the trivial and nontrivial characters of $\O(V_0^+) = \{ \pm 1\}$, respectively.
Then both $\omega_{W_n, \psi}^+$ and $\omega_{W_n, \psi}^-$ are irreducible, and we call them the even and odd elementary Weil representations of $\Mp(W_n)$ with respect to $\psi$, respectively.

\begin{lem}
\label{l:elem-weil}
We have
\begin{align*}
 \omega_{W_n, \psi}^+ & = J_{B, \psi}(|\cdot|^{n-\frac{1}{2}}, |\cdot|^{n-\frac{3}{2}}, \dots, |\cdot|^{\frac{1}{2}}), \\ 
 \omega_{W_n, \psi}^- & = J_{P, \psi}(|\cdot|^{n-\frac{1}{2}}, |\cdot|^{n-\frac{3}{2}}, \dots, |\cdot|^{\frac{3}{2}}, \omega_{W_1, \psi}^-),
\end{align*}
where $P = P_{(1^{n-1})}$.
\end{lem}

\begin{proof}
We use the Schr\"odinger model of the Weil representation $\omega_{W_n, V_0^+, \psi}$ on the space of Schwartz functions on $F^n$.
Then $\omega_{W_n, \psi}^+$ and $\omega_{W_n, \psi}^-$ are realized on the subspaces of even and odd functions, respectively.
We may define nonzero $\Mp(W_n)$-equivariant maps
\[
 \lambda^+ : \omega_{W_n, \psi}^+ \longrightarrow I_{P_n, \psi}(|\det|^{-\frac{n}{2}}), \qquad
 \lambda^- : \omega_{W_n, \psi}^- \longrightarrow I_{P_{n-1}, \psi}(|\det|^{-\frac{n+1}{2}}, \omega_{W_1, \psi}^-)
\]
by 
\[
 \lambda^+(f)(g) = [\omega^+_{W_n,\psi}(g)f](0), \qquad
 \lambda^-(f)(g) = [\omega^-_{W_n,\psi}(g)f](\underbrace{0,0,\dots,0}_{n-1}, \, \cdot \,).
\]
Since $\omega_{W_n, \psi}^\pm$ is irreducible and $\omega_{W_1, \psi}^-$ is tempered, the images of $\lambda^+$ and $\lambda^-$ must be the unique irreducible subrepresentations of 
\[
 I_{P_n, \psi}(|\det|^{-\frac{n}{2}}) \subset I_{B, \psi}(|\cdot|^{-n+\frac{1}{2}}, |\cdot|^{-n+\frac{3}{2}}, \dots, |\cdot|^{-\frac{1}{2}})
\]
and
\[
 I_{P_{n-1}, \psi}(|\det|^{-\frac{n+1}{2}}, \omega_{W_1, \psi}^-) \subset I_{P, \psi}(|\cdot|^{-n+\frac{1}{2}}, |\cdot|^{-n+\frac{3}{2}}, \dots, |\cdot|^{-\frac{3}{2}}, \omega_{W_1, \psi}^-), 
\]
respectively.
The lemma now follows from this and the fact that $(\omega_{W_n, \psi}^\pm)^\vee = \omega_{W_n, \psi^{-1}}^\pm$.
\end{proof}

\section{Global theta lifts}

In this section, we introduce global theta lifts and recall some of their basic properties. 

\subsection{Notation}

Let $F$ be a number field and fix a nontrivial additive character $\psi$ of $\A/F$.
Let $W_n$ be a $2n$-dimensional symplectic space over $F$ and $V_m$ a $(2m+1)$-dimensional quadratic space over $F$ with trivial discriminant.
We denote by
\[
 \omega_{W_n,V_m,\psi} = \bigotimes_v \omega_{W_{n,v},V_{m,v},\psi_v}
\]
the Weil representation of $\Mp(W_n)(\A) \times \O(V_m)(\A)$ with respect to $\psi$, which is equipped with a natural equivariant map $\phi \mapsto \theta(\phi)$ from $\omega_{W_n,V_m,\psi}$ to the space of left $\Sp(W_n)(F) \times \O(V_m)(F)$-invariant smooth functions on $\Mp(W_n)(\A) \times \O(V_m)(\A)$ of moderate growth.

For any irreducible genuine cuspidal automorphic representation $\pi$ of $\Mp(W_n)(\A)$, we define an automorphic representation $\Theta_{W_n,V_m,\psi}(\pi)$ of $\O(V_m)(\A)$ as the space spanned by all automorphic forms of the form
\[
 \theta(f,\varphi)(h) = \int_{\Sp(W_n)(F) \backslash \Mp(W_n)(\A)} \theta(f)(g,h) \overline{\varphi(g)} \, dg
\]
for $f \in \omega_{W_n,V_m,\psi}$ and $\varphi \in \pi$.
If $\Theta_{W_n,V_m,\psi}(\pi)$ is nonzero and is contained in the space of square-integrable automorphic forms on $\O(V_m)(\A)$, then by \cite[Corollary 7.1.3]{kr}, $\Theta_{W_n,V_m,\psi}(\pi)$ is irreducible and is isomorphic to $\bigotimes_v \theta_{W_{n,v},V_{m,v},\psi_v}(\pi_v)$.

Similarly, for any irreducible cuspidal automorphic representation $\sigma$ of $\O(V_m)(\A)$, we may define an automorphic representation $\Theta_{W_n,V_m,\psi}(\sigma)$ of $\Mp(W_n)(\A)$.

\subsection{Tower property}

Given an irreducible cuspidal automorphic representation $\sigma$ of $\O(V_m)(\A)$, we consider a family of the global theta lifts $\Theta_{W_n,V_m,\psi}(\sigma)$ to $\Mp(W_n)(\A)$ when $W_n$ varies.
Let $n_0$ be the smallest nonnegative integer such that $\Theta_{W_{n_0},V_m,\psi}(\sigma)$ is nonzero.
Then the tower property \cite{rallis} says that
\begin{itemize}
\item such $n_0$ exists;
\item $\Theta_{W_{n_0},V_m,\psi}(\sigma)$ is contained in $\mathcal{A}_{\cusp}(\Mp(W_n))$;
\item $\Theta_{W_n,V_m,\psi}(\sigma)$ is nonzero but is not contained in $\mathcal{A}_{\cusp}(\Mp(W_n))$ for all $n > n_0$,
\end{itemize}
where $\mathcal{A}_{\cusp}(\Mp(W_n))$ denotes the space of genuine cusp forms on $\Mp(W_n)(\A)$.
Moreover, we have the following refinement by M{\oe}glin \cite{moeglin-theta} and Jiang--Soudry \cite[Theorem 3.6]{js} (note that there is a typo in \cite[Theorem 3.6]{js}: $n+a+1$ should be $2n+a+1$).

\begin{prop}
\label{p:tower}
Assume that $n>n_0$ and $n+n_0>2m$.
Then $\Theta_{W_n,V_m,\psi}(\sigma)$ is spanned by residues of Eisenstein series
\[
 \Res_{s=s_0} E(s,\Phi),
\]
where $s_0 = \frac{1}{2}(n+n_0-2m)$ and $\Phi$ runs over sections of 
\[
 \Ind^{\Mp(W_n)(\A)}_{\widetilde{P}_{n-n_0}(\A)}(\chi_\psi |\det|^s \boxtimes \Theta_{W_{n_0},V_m,\psi}(\sigma)).
\]
In particular, $\Theta_{W_n,V_m,\psi}(\sigma)$ is orthogonal to $\mathcal{A}_{\cusp}(\Mp(W_n))$.
\end{prop}

Similarly, given an irreducible genuine cuspidal automorphic representation $\pi$ of $\Mp(W_n)(\A)$, an analogous result holds for a family of the global theta lifts $\Theta_{W_n,V_m,\psi}(\pi)$ to $\O(V_m)(\A)$ when $V_m$ varies in a fixed Witt tower.

\subsection{Nonvanishing}

We now discuss the nonvanishing of global theta lifts.
For this, it suffices to compute the Petersson inner products of the global theta lifts.
This is achieved by the Rallis inner product formula \cite{kr,gqt,yamana}, which expresses the inner products in terms of special values of automorphic $L$-functions and which implies the following.

\begin{prop}
\label{p:nonvanish1}
Let $\sigma$ be an irreducible cuspidal automorphic representation of $\O(V_m)(\A)$ and $L(s, \sigma)$ its standard $L$-function.
Assume that $m \le n \le 2m$ and that $\Theta_{W_{n-1}, V_m, \psi}(\sigma)$ is zero.
Then $\Theta_{W_n, V_m, \psi}(\sigma)$ is nonzero if and only if
\begin{itemize}
 \item $\Theta_{W_{n,v}, V_{m,v}, \psi_v}(\sigma_v)$ is nonzero for all $v$; and
 \item $L(s,\sigma)$ is holomorphic and nonzero at $s=n-m+\frac{1}{2}$.
\end{itemize}   
\end{prop}

An analogous result also holds for global theta lifts in the other direction, but in this paper, we only need the following weaker result, which is a consequence of the regularized Siegel--Weil formula \cite{kr,ichino,wu} (see also the argument in the proof of \cite[Theorem 7.2.5]{kr}).

\begin{prop}
\label{p:nonvanish2} 
Let $\pi$ be an irreducible genuine cuspidal automorphic representation of $\Mp(W_n)(\A)$ and $L_\psi^S(s, \pi)$ its partial standard $L$-function relative to $\psi$, where $S$ is a sufficiently large finite set of places of $F$.
\begin{enumerate}
\item
Assume that $m < n$ and that $L_\psi^S(s, \pi)$ has a pole at $s=n-m+\frac{1}{2}$.
Then there exits a $(2m+1)$-dimensional quadratic space $V_m$ over $F$ with trivial discriminant such that $\Theta_{W_n,V_m,\psi}(\pi)$ is nonzero.
\item
Assume that $m=n$ and that $L_\psi^S(s, \pi)$ is holomorphic and nonzero at $s=\frac{1}{2}$.
Then there exits a $(2n+1)$-dimensional quadratic space $V_n$ over $F$ with trivial discriminant such that $\Theta_{W_n,V_n,\psi}(\pi)$ is nonzero.
\end{enumerate}
\end{prop}

\section{The residual spectrum of $\Mp_4$}
\label{s:residual}

Recall the decomposition
\[
 L^2_\disc(\Mp_4) = L^2_{\cusp}(\Mp_4) \oplus L^2_{\res}(\Mp_4)
\]
of the discrete spectrum into the cuspidal and residual spectra, and the further decomposition 
\[
 L^2_{\res}(\Mp_4) = L^2_{P_1}(\Mp_4) \oplus L^2_{P_2}(\Mp_4) \oplus L^2_B(\Mp_4)
\]
according to cuspidal supports.
In this section, we review the result of Gao \cite{gao} which describes the structure of $L^2_{\res}(\Mp_4)$.

\subsection{Notation}

Let $F$ be a number field with ad\`ele ring $\A$ and fix a nontrivial additive character $\psi$ of $\A/F$.
As in \S \ref{ss:mp}, let $P_{\mathbf{k}}$ be the parabolic subgroup of $\Sp_{2n}$ associated to a sequence $\mathbf{k} = (k_1, \dots, k_m)$ of positive integers such that $k_1 + \dots + k_m \le n$, so that its Levi component is isomorphic to $\GL_{k_1} \times \dots \times \GL_{k_m} \times \Sp_{2n_0}$ with $n_0 = n - k_1 - \dots - k_m$.
For any irreducible representation $\tau_i$ of $\GL_{k_i}(\A)$ and any irreducible genuine representation $\pi$ of $\Mp_{2n_0}(\A)$ such that $I_{P_{\mathbf{k}}, \psi_v}(\tau_{1,v}, \dots, \tau_{m,v}, \pi_v)$ is a standard module for all $v$, we set
\[
 J_{P_{\mathbf{k}}, \psi}(\tau_1, \dots, \tau_m, \pi) = 
 \bigotimes_v J_{P_{\mathbf{k}}, \psi_v}(\tau_{1,v}, \dots, \tau_{m,v}, \pi_v).
\]

\subsection{Structure of $L^2_{P_1}(\Mp_4)$}

Recall that $P_1 = P_{(1)}$ is the maximal parabolic subgroup of $\Sp_4$ with Levi component $\GL_1 \times \Sp_2$.
Recall also that, if $\pi$ is an irreducible genuine cuspidal automorphic representation of $\Mp_2(\A)$, then the associated global $A$-parameter $\phi$ is of the form
\begin{itemize}
\item $\phi = \rho \boxtimes S_1$ for some irreducible cuspidal automorphic representation $\rho$ of $\GL_2(\A)$ with trivial central character; or 
\item $\phi = \chi \boxtimes S_2$ for some quadratic automorphic character $\chi = \chi_a$ of $\A^\times$ with $a \in F^\times$, in which case there is a nonempty finite set $S(\pi)$ of places of $F$ of even cardinality such that
\[
 \pi_v \cong
 \begin{cases}
  J_{B,\psi_v}(\chi_v|\cdot|_v^{\frac{1}{2}}) & \text{if $v \notin S(\pi)$;} \\
  \omega_{W_{1,v}, \psi_{a,v}}^- & \text{if $v \in S(\pi)$.}
 \end{cases}
\]
\end{itemize}
By \cite[Theorem 4.5]{gao}, we have
\[
 L^2_{P_1}(\Mp_4) = L^2_{P_1}(\Mp_4)_{\mathrm{pr}} \oplus L^2_{P_1}(\Mp_4)_{\mathrm{SK}} \oplus L^2_{P_1}(\Mp_4)_{\mathrm{HPS}}
\]
with 
\begin{align}
 L^2_{P_1}(\Mp_4)_{\mathrm{pr}} & \cong  \bigoplus_{\chi} \bigoplus_{\pi} J_{P_1,\psi}(\chi|\cdot|^{\frac{3}{2}}, \pi), \label{P1-pr} \\
 L^2_{P_1}(\Mp_4)_{\mathrm{SK}} & \cong \bigoplus_{\chi, \rho} \bigoplus_{\pi} J_{P_1,\psi}(\chi|\cdot|^{\frac{1}{2}}, \pi), \label{P1-SK} \\
 L^2_{P_1}(\Mp_4)_{\mathrm{HPS}} & \cong \bigoplus_{\chi_1, \chi_2} \bigoplus_{\pi} J_{P_1,\psi}(\chi_1|\cdot|^{\frac{1}{2}}, \pi), \label{P1-HPS}
\end{align}
where 
\begin{itemize}
\item in \eqref{P1-pr}, $\chi$ runs over quadratic automorphic characters of $\A^\times$, and $\pi$ runs over irreducible genuine cuspidal automorphic representations of $\Mp_2(\A)$ with $A$-parameter $\chi \boxtimes S_2$;
\item in \eqref{P1-SK}, $\chi, \rho$ run over pairs of a quadratic automorphic character of $\A^\times$ and an irreducible cuspidal automorphic representation of $\GL_2(\A)$ with trivial central character such that $L(\frac{1}{2}, \rho \times \chi) \ne 0$, and $\pi$ runs over irreducible genuine cuspidal automorphic representations of $\Mp_2(\A)$ with $A$-parameter $\rho \boxtimes S_1$;
\item in \eqref{P1-HPS}, $\chi_1, \chi_2$ run over ordered pairs of distinct quadratic automorphic characters of $\A^\times$, and $\pi$ runs over irreducible genuine cuspidal automorphic representations of $\Mp_2(\A)$ with $A$-parameter $\chi_2 \boxtimes S_2$ such that $\chi_{1,v} \ne \chi_{2,v}$ for all $v \in S(\pi)$.
\end{itemize}
Note that 
\begin{itemize}
\item $J_{P_1,\psi}(\chi|\cdot|^{\frac{3}{2}}, \pi)$ in \eqref{P1-pr} belongs to the near equivalence class determined by the principal $A$-parameter $\chi \boxtimes S_4$;
\item $J_{P_1,\psi}(\chi|\cdot|^{\frac{1}{2}}, \pi)$ in \eqref{P1-SK} belongs to the near equivalence class determined by the $A$-parameter $(\rho \boxtimes S_1) \oplus (\chi \boxtimes S_2)$ of Saito--Kurokawa type;
\item $J_{P_1,\psi}(\chi_1|\cdot|^{\frac{1}{2}}, \pi)$ in \eqref{P1-HPS} belongs to the near equivalence class determined by the $A$-parameter $(\chi_1 \boxtimes S_2) \oplus (\chi_2 \boxtimes S_2)$ of Howe--Piatetski-Shapiro type.
\end{itemize}

\subsection{Structure of $L^2_{P_2}(\Mp_4)$}

Recall that $P_2 = P_{(2)}$ is the maximal parabolic subgroup of $\Sp_4$ with Levi component $\GL_2$.
By \cite[Theorem 3.4]{gao}, we have
\[
 L^2_{P_2}(\Mp_4) \cong \bigoplus_\rho J_{P_2, \psi}(\rho \otimes |\det|^{\frac{1}{2}}), 
\]
where $\rho$ runs over irreducible dihedral cuspidal automorphic representations of $\GL_2(\A)$ with nontrivial quadratic central characters.
Note that $J_{P_2, \psi}(\rho \otimes |\det|^{\frac{1}{2}})$ belongs to the near equivalence class determined by the $A$-parameter $\rho \boxtimes S_2$ of Soudry type.

\subsection{Structure of $L^2_B(\Mp_4)$}

Recall that $B = P_{(1,1)}$ is the Borel subgroup of $\Sp_4$ with Levi component $\GL_1 \times \GL_1$.
By \cite[Theorem 5.10]{gao}, we have
\[
 L^2_B(\Mp_4) = L^2_B(\Mp_4)_{\mathrm{pr}} \oplus L^2_B(\Mp_4)_{\mathrm{HPS}}
\]
with 
\begin{align}
 L^2_B(\Mp_4)_{\mathrm{pr}} & \cong \bigoplus_{\chi} J_{B,\psi}(\chi|\cdot|^{\frac{3}{2}}, \chi|\cdot|^{\frac{1}{2}}), \label{B-pr} \\
 L^2_B(\Mp_4)_{\mathrm{HPS}} & \cong \bigoplus_{\chi_1,\chi_2} J_{B,\psi}(\chi_1|\cdot|^{\frac{1}{2}}, \chi_2|\cdot|^{\frac{1}{2}}), \label{B-HPS}
\end{align}
where
\begin{itemize}
\item in \eqref{B-pr}, $\chi$ runs over quadratic automorphic characters of $\A^\times$;
\item in \eqref{B-HPS}, $\chi_1, \chi_2$ run over unordered pairs of distinct quadratic automorphic characters of $\A^\times$.
\end{itemize}
Note that
\begin{itemize}
\item $J_{B,\psi}(\chi|\cdot|^{\frac{3}{2}}, \chi|\cdot|^{\frac{1}{2}})$ in \eqref{B-pr} belongs to the near equivalence class determined by the principal $A$-parameter $\chi \boxtimes S_4$;
\item $J_{B,\psi}(\chi_1|\cdot|^{\frac{1}{2}}, \chi_2|\cdot|^{\frac{1}{2}})$ in \eqref{B-HPS} belongs to the near equivalence class determined by the $A$-parameter $(\chi_1 \boxtimes S_2) \oplus (\chi_2 \boxtimes S_2)$ of Howe--Piatetski-Shapiro type.
\end{itemize}

\section{Principal $A$-packets}

In this section, we construct the $A$-packet associated to a principal $A$-parameter
\[
 \phi = \chi_a \boxtimes S_4
\]
with $a \in F^\times$.
This is the most degenerate nontempered $A$-packet and is given by elementary Weil representations.

\subsection{Local $A$-packets}
\label{ss:local-princ}

Suppose that $F$ is local and $\phi = \chi_a \boxtimes S_4$ is a local $A$-parameter with $a \in F^\times$.
Then we have $S_{\phi} \cong \Z/2\Z$, so that we may identify $\widehat{S}_\phi$ with $\mu_2$.
We define the $A$-packet $\Pi_{\phi,\psi}(\Mp_4)$ by 
\[
 \Pi_{\phi,\psi}(\Mp_4) = \{ \pi^\epsilon \, | \, \epsilon \in \mu_2 \}, \qquad
 \pi^\epsilon = \omega_{W_2, \psi_a}^\epsilon.
\]
More explicitly, we have
\[
 \pi^+ = J_{B, \psi}(\chi_a |\cdot|^{\frac{3}{2}}, \chi_a|\cdot|^{\frac{1}{2}}), \qquad
 \pi^- = J_{P_1, \psi}(\chi_a |\cdot|^{\frac{3}{2}}, \omega_{W_1,\psi_a}^-)
\]
by Lemma \ref{l:elem-weil}.
Hence $\Pi_{\phi,\psi}(\Mp_4)$ is multiplicity-free and we have $\Pi_{\varphi_\phi,\psi}(\Mp_4) = \{ \pi^+ \}$ as required, where $\varphi_\phi$ is the $L$-parameter associated to $\phi$ and $\Pi_{\varphi_\phi,\psi}(\Mp_4)$ is its $L$-packet.

\subsection{Structure of $L^2_{\phi,\psi}(\Mp_4)$}

Suppose that $F$ is global.
For $\epsilon \in \mu_{2,\A}$, put
\[
 \pi^\epsilon = \bigotimes_v \pi_v^{\epsilon_v},
\]
where $\pi_v^{\epsilon_v}$ is the representation in the local $A$-packet $\Pi_{\phi_v,\psi_v}(\Mp_4)$ defined above.

\begin{prop}
\label{p:princ}
We have 
\[
 L^2_{\phi,\psi}(\Mp_4) \cong \bigoplus_{\epsilon} \pi^\epsilon, 
\]
where $\epsilon$ runs over elements in $\mu_{2,\A}$ such that $\prod_v \epsilon_v = 1$.
\end{prop}

\subsection{Proof of Proposition \ref{p:princ}}

For $\epsilon \in \mu_{2,\A}$ such that $\prod_v \epsilon_v = 1$, we may realize the elementary Weil representation $\pi^\epsilon$ on the space $\VV^\epsilon$ of the associated theta functions on $\Mp_4(\A)$.
Indeed, by the tower property and Proposition \ref{p:tower}, $\VV^\epsilon$ is nonzero and is contained in $L^2_{\res}(\Mp_4)$.
Moreover, by the description of $L^2_{\res}(\Mp_4)$ in \S \ref{s:residual} (see in particular \eqref{P1-pr} and \eqref{B-pr}), we have
\begin{equation}
\label{eq:res-pr}
 L^2_{\phi,\psi}(\Mp_4) \cap L^2_{\res}(\Mp_4) = \bigoplus_\epsilon \VV^\epsilon, 
\end{equation}
where $\epsilon$ runs over elements in $\mu_{2,\A}$ such that $\prod_v \epsilon_v = 1$.

It remains to show that the orthogonal complement of $\bigoplus_\epsilon \VV^\epsilon$ in $L^2_{\phi,\psi}(\Mp_4)$ is zero.
Suppose on the contrary that there exists an irreducible genuine automorphic representation $\pi$ of $\Mp_4(\A)$ occurring in this orthogonal complement.
Then $\pi$ is cuspidal by \eqref{eq:res-pr} and we have
\[
 L^S_{\psi_a}(s, \pi) = \zeta^S(s+\tfrac{3}{2}) \cdot \zeta^S(s+\tfrac{1}{2}) \cdot \zeta^S(s-\tfrac{1}{2}) \cdot \zeta^S(s-\tfrac{3}{2}), 
\]
where $S$ is a sufficiently large finite set of places of $F$.
Since $L^S_{\psi_a}(s, \pi)$ has a pole at $s=\frac{5}{2}$, it follows from Proposition \ref{p:nonvanish2} that the global theta lift $\Theta_{W_2,V_0^+,\psi_a}(\pi)$ to $\O(V_0^+)(\A)$ is nonzero.
Hence, by the adjunction formula, $\pi$ is not orthogonal to $\VV^\epsilon$ for some $\epsilon \in \mu_{2,\A}$ such that $\prod_v \epsilon_v = 1$, which is a contradiction.
This completes the proof of Proposition \ref{p:princ}.

\section{$A$-packets of Saito--Kurokawa type}
\label{s:sk}

In this section, we construct the $A$-packet associated to an $A$-parameter of Saito--Kurokawa type
\[
 \phi = (\rho \boxtimes S_1) \oplus (\chi_a \boxtimes S_2)
\]
with an irreducible cuspidal automorphic representation $\rho$ of $\GL_2(\A)$ with trivial central character and $a \in F^\times$.
Since the associated character $\tilde{\epsilon}_\phi$ is nontrivial, this is the most interesting nontempered $A$-packet from the viewpoint of the multiplicity formula.

\subsection{Local $A$-packets}
\label{ss:local-sk}

Suppose that $F$ is local and $\phi = (\rho \boxtimes S_1) \oplus (\chi_a \boxtimes S_2)$ is a local $A$-parameter with a $2$-dimensional symplectic almost tempered representation $\rho$ of $L_F$ and $a \in F^\times$.
Here we say that a representation $\rho$ of $L_F$ is \emph{almost tempered} if, for any irreducible summand $\rho_0$ of $\rho$, the image of $W_F$ under $\rho_0 |\cdot|^{-s_0}$ is bounded for some $s_0 \in \R$ with $|s_0| < \frac{1}{2}$.
Then we have
\[
 S_\phi \cong
 \begin{cases}
  (\Z/2\Z)^2 & \text{if $\rho$ is irreducible;} \\
  (\Z/2\Z)^2 / (\Z/2\Z \oplus \{ 0 \}) & \text{if $\rho$ is reducible,}
 \end{cases}
\]
so that we may identify $\widehat{S}_\phi$ with
\[
\begin{cases}
 \mu_2 \times \mu_2 & \text{if $\rho$ is irreducible;} \\
 \{ 1 \} \times \mu_2 & \text{if $\rho$ is reducible.}
\end{cases}
\]

We construct the associated $A$-packet by using theta lifts from $\O(V_1)$, where $V_1$ is a $3$-dimensional quadratic space over $F$ with trivial discriminant.
As explained in \S \ref{ss:so}, up to isometry, there are precisely two such quadratic spaces $V_1^+$ and $V_1^-$ when $F$ is nonarchimedean or $F = \R$, and there is a unique such quadratic space $V_1^+$ when $F = \C$.
For $\epsilon \in \mu_2$, we may identify $\SO(V_1^\epsilon)$ with
\[
 (B^\epsilon)^\times / F^\times
\]
for some quaternion algebra $B^\epsilon$ over $F$.
Let $\sigma_0^\epsilon$ be the irreducible representation of $(B^\epsilon)^\times$ with $L$-parameter $\rho \otimes \chi_a$.
Here $\sigma_0^-$ does not exist if $\rho$ is reducible, in which case we interpret $\sigma_0^-$ as zero. 
Since $\rho \otimes \chi_a$ is symplectic, we may regard $\sigma_0^\epsilon$ as a representation of $\SO(V_1^\epsilon)$.
For $\epsilon, \epsilon' \in \mu_2$, let $\sigma^{\epsilon,\epsilon'}$ be the $\epsilon'$-extension of $\sigma_0^\epsilon$ to $\O(V_1^\epsilon)$.
Let $\epsilon_1, \epsilon_2 \in \mu_2$ be such that 
\begin{equation}
\label{eq:e1e2-SK}
 \epsilon = \epsilon_1 \cdot \epsilon(\tfrac{1}{2},\rho) \cdot \epsilon(\tfrac{1}{2},\rho \times \chi_a) \cdot \chi_a(-1), \qquad
 \epsilon' = \epsilon_1 \cdot \epsilon_2 \cdot \epsilon(\tfrac{1}{2},\rho) \cdot \chi_a(-1).
\end{equation}
Note that $\epsilon = \epsilon_1$ if $\rho$ is reducible.

\begin{lem}
\begin{enumerate}
\item
\label{item:SK-i}
The theta lift $\theta_{W_1,V_1^{\epsilon}, \psi_a}(\sigma^{\epsilon, \epsilon'})$ to $\Mp_2(F)$ is nonzero if and only if one of the following holds:
\begin{itemize}
\item $\rho$ is irreducible and $\epsilon_2 = 1$;
\item $\rho$ is reducible and $\epsilon_1 = \epsilon_2 = 1$.
\end{itemize}
\item
\label{item:SK-ii}
The theta lift $\theta_{W_2,V_1^{\epsilon}, \psi_a}(\sigma^{\epsilon, \epsilon'})$ to $\Mp_4(F)$ is nonzero if and only if one of the following holds:
\begin{itemize}
\item $\rho$ is irreducible and
\begin{itemize}
\item if $F$ is nonarchimedean and $\rho = \chi_a \boxtimes S_2$, or $F = \R$ and $\rho = \mathcal{D}_{\frac{1}{2}}$, then
\[
 (\epsilon_1,\epsilon_2) \ne
 \begin{cases}
  (1,-1) & \text{if $\chi_a \ne 1$;} \\
  (-1,-1) & \text{if $\chi_a = 1$;}
 \end{cases}
\]
\item otherwise, $\epsilon_1, \epsilon_2$ are arbitrary;
\end{itemize}
\item $\rho$ is reducible and $\epsilon_1 = 1$.
\end{itemize}
\end{enumerate}
\end{lem}

\begin{proof}
The assertions \eqref{item:SK-i} and \eqref{item:SK-ii} follow from \eqref{eq:dichotomy} and Lemma \ref{l:o3mp4-nonzero} below, respectively. 
\end{proof} 

We now define the $A$-packet $\Pi_{\phi,\psi}(\Mp_4)$ by
\[
 \Pi_{\phi,\psi}(\Mp_4) = \{ \pi^{\epsilon_1,\epsilon_2} \, | \, \epsilon_1,\epsilon_2 \in \mu_2 \}, \qquad 
 \pi^{\epsilon_1,\epsilon_2} = \theta_{W_2,V_1^{\epsilon}, \psi_a}(\sigma^{\epsilon, \epsilon'}),
\]
where $\epsilon, \epsilon' \in \mu_2$ are as in \eqref{eq:e1e2-SK}.
Then, by Lemma \ref{l:o3mp4-mult-free} below, $\Pi_{\phi,\psi}(\Mp_4)$ is multiplicity-free.
Moreover, since 
\[
 \pi^{\epsilon_1,+} = J_{P_1, \psi}(\chi_a |\cdot|^{\frac{1}{2}}, \theta_{W_1,V_1^\epsilon, \psi_a}(\sigma^{\epsilon,\epsilon'}))
\]
by Lemma \ref{l:ind-princ-1} below, we have $\Pi_{\varphi_\phi,\psi}(\Mp_4) = \{ \pi^{+,+}, \pi^{-,+} \}$ if $\rho$ is irreducible and $\Pi_{\varphi_\phi,\psi}(\Mp_4) = \{ \pi^{+,+} \}$ if $\rho$ is reducible as required, where $\varphi_\phi$ is the $L$-parameter associated to $\phi$ and $\Pi_{\varphi_\phi,\psi}(\Mp_4)$ is its $L$-packet.
In Appendix \ref{a:A-packets} below, we will describe $\pi^{\epsilon_1,\epsilon_2}$ explicitly.

\subsection{Structure of $L^2_{\phi,\psi}(\Mp_4)$}

Suppose that $F$ is global.
For $\epsilon_1, \epsilon_2 \in \mu_{2,\A}$, put
\[
 \pi^{\epsilon_1,\epsilon_2} = \bigotimes_v \pi_v^{\epsilon_{1,v},\epsilon_{2,v}},
\]
where $\pi_v^{\epsilon_{1,v},\epsilon_{2,v}}$ is the representation in the local $A$-packet $\Pi_{\phi_v,\psi_v}(\Mp_4)$ defined above.

\begin{prop}
\label{p:sk}
We have 
\[
 L^2_{\phi,\psi}(\Mp_4) \cong \bigoplus_{\epsilon_1,\epsilon_2} \pi^{\epsilon_1,\epsilon_2},
\]
where $\epsilon_1,\epsilon_2$ run over elements in $\mu_{2,\A}$ such that 
\begin{equation}
\label{eq:epsilon-sk}
 \prod_v \epsilon_{1,v} = \epsilon(\tfrac{1}{2}, \rho) \cdot \epsilon(\tfrac{1}{2}, \rho \times \chi_a), \qquad
 \prod_v \epsilon_{2,v} = \epsilon(\tfrac{1}{2}, \rho \times \chi_a).
\end{equation}
\end{prop}

\subsection{Proof of Proposition \ref{p:sk}}

For $\epsilon_1, \epsilon_2 \in \mu_{2,\A}$ satisfying \eqref{eq:epsilon-sk}, we define $\epsilon, \epsilon' \in \mu_{2,\A}$ by
\[
 \epsilon_v = \epsilon_{1,v} \cdot \epsilon(\tfrac{1}{2},\rho_v) \cdot \epsilon(\tfrac{1}{2},\rho_v \times \chi_{a,v}) \cdot \chi_{a,v}(-1), \qquad
 \epsilon'_v = \epsilon_{1,v} \cdot \epsilon_{2,v} \cdot \epsilon(\tfrac{1}{2},\rho_v) \cdot \chi_{a,v}(-1), 
\]
so that $\prod_v \epsilon_v = \prod_v \epsilon_v' = 1$.
Put 
\[
 \sigma^{\epsilon, \epsilon'} = \bigotimes_v \sigma_v^{\epsilon_v, \epsilon'_v}, 
\]
where $\sigma_v^{\epsilon_v, \epsilon'_v}$ is the representation of $\O(V_{1,v}^{\epsilon_v})$ defined above with $L$-parameter $\rho_v \otimes \chi_{a,v}$.
If $\sigma^{\epsilon, \epsilon'}$ is nonzero, then $\sigma^{\epsilon, \epsilon'}$ is an irreducible cuspidal automorphic representation of $\O(V_1^{\epsilon})(\A)$, where $V_1^{\epsilon}$ is the $3$-dimensional quadratic space over $F$ with trivial discriminant such that $V_1^\epsilon \otimes_F F_v = V_{1,v}^{\epsilon_v}$ for all $v$.

We may realize $\pi^{\epsilon_1, \epsilon_2}$ on the global theta lift
\[
 \VV^{\epsilon_1,\epsilon_2} = \Theta_{W_2,V_1^\epsilon,\psi_a}(\sigma^{\epsilon,\epsilon'}).
\]
Indeed, by the tower property and Propositions \ref{p:tower}, \ref{p:nonvanish1}, we have
\begin{itemize}
\item if the abstract representation $\pi^{\epsilon_1,\epsilon_2}$ is nonzero, then the global theta lift $\VV^{\epsilon_1,\epsilon_2}$ is nonzero;
\item if $\epsilon_{2,v} = 1$ for all $v$ and $L(\frac{1}{2}, \rho \times \chi_a)$ is nonzero, then $\VV^{\epsilon_1,\epsilon_2}$ is contained in $L^2_{\res}(\Mp_4)$; 
\item otherwise, $\VV^{\epsilon_1,\epsilon_2}$ is contained in $L^2_{\cusp}(\Mp_4)$.
\end{itemize}
Moreover, by the description of $L^2_{\res}(\Mp_4)$ in \S \ref{s:residual} (see in particular \eqref{P1-SK}), we have
\begin{equation}
\label{eq:res-SK}
 L^2_{\phi,\psi}(\Mp_4) \cap L^2_{\res}(\Mp_4) \subset \bigoplus_{\epsilon_1,\epsilon_2} \VV^{\epsilon_1,\epsilon_2},
\end{equation}
where $\epsilon_1,\epsilon_2$ run over elements in $\mu_{2,\A}$ satisfying \eqref{eq:epsilon-sk}.

It remains to show that the orthogonal complement of $\bigoplus_{\epsilon_1,\epsilon_2} \VV^{\epsilon_1,\epsilon_2}$ in $L^2_{\phi,\psi}(\Mp_4)$ is zero.
Suppose on the contrary that there exists an irreducible genuine automorphic representation $\pi$ of $\Mp_4(\A)$ occurring in this orthogonal complement.
Then $\pi$ is cuspidal by \eqref{eq:res-SK} and we have
\[
 L^S_{\psi_a}(s,\pi) = L^S(s,\rho \times \chi_a) \cdot \zeta^S(s+\tfrac{1}{2}) \cdot \zeta^S(s-\tfrac{1}{2}),
\]
where $S$ is a sufficiently large finite set of places of $F$.
Since $L^S_{\psi_a}(s,\pi)$ has a pole at $s = \frac{3}{2}$, it follows from Proposition \ref{p:nonvanish2} that the global theta lift $\Theta_{W_2,V_1^\epsilon, \psi_a}(\pi)$ to $\O(V_1^\epsilon)(\A)$ is nonzero for some $\epsilon \in \mu_{2,\A}$ such that $\prod_v \epsilon_v = 1$.
Moreover, since $\theta_{W_{2,v},V_{1,v}^{\epsilon_v}, \psi_{a,v}}(\pi_v) \cong \sigma_v^{+,+}$ for almost all $v$, we have $\Theta_{W_2,V_1^\epsilon, \psi_a}(\pi) = \sigma^{\epsilon,\epsilon'}$ for some $\epsilon' \in \mu_{2,\A}$ such that $\prod_v \epsilon_v' = 1$ by the strong multiplicity one theorem.
Hence, by the adjunction formula, $\pi$ is not orthogonal to $\VV^{\epsilon_1,\epsilon_2}$ for some $\epsilon_1,\epsilon_2 \in \mu_{2,\A}$ satisfying \eqref{eq:epsilon-sk}, which is a contradiction.
This completes the proof of Proposition \ref{p:sk}.

\section{$A$-packets of Howe--Piatetski-Shapiro type}
\label{s:howe-ps}

In this section, we construct the $A$-packet associated to an $A$-parameter of Howe--Piatetski-Shapiro type
\[
 \phi = (\chi_a \boxtimes S_2) \oplus (\chi_b \boxtimes S_2)
\]
with $a, b \in F^\times$ such that $\chi_a \ne \chi_b$.

\subsection{Local $A$-packets}
\label{ss:local-howe-ps}

Suppose that $F$ is local and $\phi = (\chi_a \boxtimes S_2) \oplus (\chi_b \boxtimes S_2)$ is a local $A$-parameter with $a,b \in F^\times$.
Then we have
\[
 S_\phi \cong
 \begin{cases}
  (\Z/2\Z)^2 & \text{if $\chi_a \ne \chi_b$;} \\
  (\Z/2\Z)^2 / \Delta \Z/2\Z & \text{if $\chi_a = \chi_b$,}
 \end{cases}
\]
so that we may identify $\widehat{S}_\phi$ with
\[
\begin{cases}
 \mu_2 \times \mu_2 & \text{if $\chi_a \ne \chi_b$;} \\
 \Delta \mu_2 & \text{if $\chi_a = \chi_b$.}
\end{cases}
\]

We construct the associated $A$-packet by using theta lifts from $\O(V_1)$ as in \S \ref{ss:local-sk}, but with the $L$-parameter $\rho \boxtimes S_1$ replaced by the $A$-parameter $\chi_b \boxtimes S_2$.
For $\epsilon \in \mu_2$, let $\sigma_0^\epsilon$ be the irreducible representation of $(B^\epsilon)^\times$ with $A$-parameter $\chi_{ab} \boxtimes S_2$, i.e.
\[
 \sigma_0^\epsilon = \chi_{ab} \circ \mathrm{N}_{B^\epsilon}, 
\]
where $\mathrm{N}_{B^\epsilon}$ is the reduced norm on $B^\epsilon$.
Here $\sigma_0^-$ does not exist if $F = \C$, in which case we interpret $\sigma_0^-$ as zero.
We may regard $\sigma_0^\epsilon$ as a representation of $\SO(V_1^\epsilon)$.
For $\epsilon, \epsilon' \in \mu_2$, let $\sigma^{\epsilon,\epsilon'}$ be the $\epsilon'$-extension of $\sigma_0^\epsilon$ to $\O(V_1^\epsilon)$.
Let $\epsilon_1, \epsilon_2 \in \mu_2$ be such that 
\begin{equation}
\label{eq:e1e2-HPS}
 \epsilon = \epsilon_2, \qquad
 \epsilon' = \epsilon_1 \cdot \epsilon_2 \cdot \chi_{ab}(-1).
\end{equation}

\begin{lem}

\begin{enumerate}
\item
\label{item:HPS-i}
The theta lift $\theta_{W_1,V_1^{\epsilon}, \psi_a}(\sigma^{\epsilon, \epsilon'})$ to $\Mp_2(F)$ is nonzero if and only if one of the following holds:
\begin{itemize}
\item $F \ne \C$, $\chi_a \ne \chi_b$, and $\epsilon_1 = 1$;
\item $F \ne \C$, $\chi_a = \chi_b$, and $\epsilon_1 = \epsilon_2$;
\item $F = \C$ (so that $\chi_a = \chi_b$) and $\epsilon_1 = \epsilon_2 = 1$.
\end{itemize}
\item
\label{item:HPS-ii}
The theta lift $\theta_{W_2,V_1^{\epsilon}, \psi_a}(\sigma^{\epsilon, \epsilon'})$ to $\Mp_4(F)$ is nonzero if and only if one of the following holds:
\begin{itemize}
\item $F$ is nonarchimedean, $\chi_a \ne \chi_b$, and $\epsilon_1, \epsilon_2$ are arbitrary;
\item $F = \R$, $\chi_a \ne \chi_b$, and $(\epsilon_1, \epsilon_2) \ne (-1,-1)$;
\item $F \ne \C$, $\chi_a = \chi_b$, and $\epsilon_1 = \epsilon_2$;
\item $F = \C$ (so that $\chi_a = \chi_b$) and $\epsilon_1 = \epsilon_2 = 1$.
\end{itemize}
\end{enumerate}
\end{lem}

\begin{proof}
The assertions \eqref{item:HPS-i} and \eqref{item:HPS-ii} follow from \eqref{eq:dichotomy} and Lemma \ref{l:o3mp4-nonzero} below, respectively.
\end{proof}

We now define the $A$-packet $\Pi_{\phi,\psi}(\Mp_4)$ by
\[
 \Pi_{\phi,\psi}(\Mp_4) = \{ \pi^{\epsilon_1,\epsilon_2} \, | \, \epsilon_1,\epsilon_2 \in \mu_2 \}, \qquad 
 \pi^{\epsilon_1,\epsilon_2} = \theta_{W_2,V_1^{\epsilon}, \psi_a}(\sigma^{\epsilon, \epsilon'}),
\]
where $\epsilon, \epsilon' \in \mu_2$ are as in \eqref{eq:e1e2-HPS}.
Then, by Lemma \ref{l:o3mp4-mult-free} below, $\Pi_{\phi,\psi}(\Mp_4)$ is multiplicity-free.
Moreover, since 
\[
 \pi^{+,+} = J_{B, \psi}(\chi_a |\cdot|^{\frac{1}{2}}, \chi_b |\cdot|^{\frac{1}{2}}) 
\]
by Lemma \ref{l:howe-ps-local+} below, we have $\Pi_{\varphi_\phi, \psi}(\Mp_4) = \{ \pi^{+,+} \}$ as required, where $\varphi_\phi$ is the $L$-parameter associated to $\phi$ and $\Pi_{\varphi_\phi, \psi}(\Mp_4)$ is its $L$-packet.
In Appendix \ref{a:A-packets} below, we will describe $\pi^{\epsilon_1,\epsilon_2}$ explicitly.
We only remark that if $\chi_a \ne \chi_b$, then it follows from Lemmas \ref{l:howe-ps-local+} and \ref{l:howe-ps-local-} below that
\[
 \pi^{+,-} = J_{P_1, \psi}(\chi_a |\cdot|^{\frac{1}{2}}, \omega_{W_1,\psi_b}^-), \qquad
 \pi^{-,+} = J_{P_1, \psi}(\chi_b |\cdot|^{\frac{1}{2}}, \omega_{W_1,\psi_a}^-),
\]
which we will use later.

\subsection{Structure of $L^2_{\phi,\psi}(\Mp_4)$}

Suppose that $F$ is global.
For $\epsilon_1, \epsilon_2 \in \mu_{2,\A}$, put
\[
 \pi^{\epsilon_1,\epsilon_2} = \bigotimes_v \pi_v^{\epsilon_{1,v},\epsilon_{2,v}},
\]
where $\pi_v^{\epsilon_{1,v},\epsilon_{2,v}}$ is the representation in the local $A$-packet $\Pi_{\phi_v, \psi_v}(\Mp_4)$ defined above.

\begin{prop}
\label{p:howe-ps}
We have 
\[
 L^2_{\phi,\psi}(\Mp_4) \cong \bigoplus_{\epsilon_1,\epsilon_2} \pi^{\epsilon_1,\epsilon_2},
\]
where $\epsilon_1,\epsilon_2$ run over elements in $\mu_{2,\A}$ such that $\prod_v \epsilon_{1,v} = \prod_v \epsilon_{2,v} = 1$.
\end{prop}

\subsection{Proof of Proposition \ref{p:howe-ps}}

The proof is similar to that of Proposition \ref{p:sk}.
For $\epsilon_1, \epsilon_2 \in \mu_{2,\A}$ such that $\prod_v \epsilon_{1,v} = \prod_v \epsilon_{2,v} = 1$, we define $\epsilon, \epsilon' \in \mu_{2,\A}$ by 
\[
 \epsilon_v = \epsilon_{2,v}, \qquad
 \epsilon'_v = \epsilon_{1,v} \cdot \epsilon_{2,v} \cdot \chi_{ab,v}(-1), 
\]
so that $\prod_v \epsilon_v = \prod_v \epsilon_v' = 1$.
Put 
\[
 \sigma^{\epsilon, \epsilon'} = \bigotimes_v \sigma_v^{\epsilon_v, \epsilon'_v}, 
\]
where $\sigma_v^{\epsilon_v, \epsilon'_v}$ is the representation of $\O(V_{1,v}^{\epsilon_v})$ defined above with $A$-parameter $\chi_{ab,v} \boxtimes S_2$.
If $\sigma^{\epsilon, \epsilon'}$ is nonzero, then $\sigma^{\epsilon, \epsilon'}$ is a $1$-dimensional automorphic representation of $\O(V_1^{\epsilon})(\A)$, where $V_1^{\epsilon}$ is the $3$-dimensional quadratic space over $F$ with trivial discriminant such that $V_1^\epsilon \otimes_F F_v = V_{1,v}^{\epsilon_v}$ for all $v$.

Suppose that $\epsilon_{2,v} \ne 1$ for some $v$, so that $V_1^\epsilon$ is anisotropic and $\sigma^{\epsilon, \epsilon'}$ is cuspidal.
Then we may realize $\pi^{\epsilon_1, \epsilon_2}$ on the global theta lift
\[
 \VV^{\epsilon_1,\epsilon_2} = \Theta_{W_2,V_1^\epsilon,\psi_a}(\sigma^{\epsilon,\epsilon'}).
\]
Indeed, by the tower property and Propositions \ref{p:tower}, \ref{p:nonvanish1}, we have:
\begin{itemize}
\item if the abstract representation $\pi^{\epsilon_1,\epsilon_2}$ is nonzero, then the global theta lift $\VV^{\epsilon_1,\epsilon_2}$ is nonzero;
\item if 
\[
\epsilon_{1,v} =
\begin{cases}
 1 & \text{if $\chi_{a,v} \ne \chi_{b,v}$;} \\
 \epsilon_{2,v} & \text{if $\chi_{a,v} = \chi_{b,v}$}
\end{cases} 
\]
for all $v$, then $\VV^{\epsilon_1,\epsilon_2}$ is contained in $L^2_{\res}(\Mp_4)$;
\item otherwise, $\VV^{\epsilon_1,\epsilon_2}$ is contained in $L^2_{\cusp}(\Mp_4)$.
\end{itemize}
Moreover, by the description of $L^2_{\res}(\Mp_4)$ in \S \ref{s:residual} (see in particular \eqref{P1-HPS} and \eqref{B-HPS}), we may realize $\pi^{\epsilon_1,\epsilon_2}$ on a subspace $\VV^{\epsilon_1,\epsilon_2}$ of $L^2_{\res}(\Mp_4)$ even if $\epsilon_{2,v} = 1$ for all $v$, and we have
\begin{equation}
\label{eq:res-HPS}
 L^2_{\phi,\psi}(\Mp_4) \cap L^2_{\res}(\Mp_4) \subset \bigoplus_{\epsilon_1,\epsilon_2} \VV^{\epsilon_1,\epsilon_2}, 
\end{equation}
where $\epsilon_1,\epsilon_2$ run over elements in $\mu_{2,\A}$ such that $\prod_v \epsilon_{1,v} = \prod_v \epsilon_{2,v} = 1$.

It remains to show that the orthogonal complement of $\bigoplus_{\epsilon_1,\epsilon_2} \VV^{\epsilon_1,\epsilon_2}$ in $L^2_{\phi,\psi}(\Mp_4)$ is zero.
Suppose on the contrary that there exists an irreducible genuine automorphic representation $\pi$ of $\Mp_4(\A)$ occurring in this orthogonal complement.
Then $\pi$ is cuspidal by \eqref{eq:res-HPS} and we have
\[
 L^S_{\psi_a}(s,\pi) = \zeta^S(s+\tfrac{1}{2}) \cdot \zeta^S(s-\tfrac{1}{2}) \cdot L^S(s+\tfrac{1}{2},\chi_{ab}) \cdot L^S(s-\tfrac{1}{2},\chi_{ab}), 
\]
where $S$ is a sufficiently large finite set of places of $F$.
Since $L^S_{\psi_a}(s,\pi)$ has a pole at $s = \frac{3}{2}$, it follows from Proposition \ref{p:nonvanish2} that the global theta lift $\Theta_{W_2,V_1^\epsilon, \psi_a}(\pi)$ to $\O(V_1^\epsilon)(\A)$ is nonzero for some $\epsilon \in \mu_{2,\A}$ such that $\prod_v \epsilon_v = 1$.
Moreover, since $\theta_{W_{2,v},V_{1,v}^{\epsilon_v}, \psi_{a,v}}(\pi_v) \cong \sigma_v^{+,+}$ for almost all $v$, we have $\Theta_{W_2,V_1^\epsilon, \psi_a}(\pi) = \sigma^{\epsilon,\epsilon'}$ for some $\epsilon' \in \mu_{2,\A}$ such that $\prod_v \epsilon_v' = 1$ by the strong multiplicity one theorem.
Hence, if $\epsilon_v \ne 1$ for some $v$, then by the adjunction formula, $\pi$ is not orthogonal to $\VV^{\epsilon_1,\epsilon_2}$ for some $\epsilon_1,\epsilon_2 \in \mu_{2,\A}$ such that $\prod_v \epsilon_{1,v} = \prod_v \epsilon_{2,v} = 1$ and $\epsilon_{2,v} \ne 1$ for some $v$, which is a contradiction.
This forces $\epsilon_v = 1$ for all $v$, i.e.~$V_1^\epsilon = V_1^+$.
Since $\sigma^{\epsilon,\epsilon'}$ is not cuspidal, the global theta lift $\Theta_{W_2,V_0^+, \psi_a}(\pi)$ to $\O(V_0^+)(\A)$ is nonzero by the tower property, so that 
\[
 \pi_v \cong J_{B, \psi_v}(\chi_{a,v} |\cdot|_v^{\frac{3}{2}}, \chi_{a,v} |\cdot|_v^{\frac{1}{2}})
\]
for almost all $v$.
This contradicts the assumption that $\pi$ occurs in $L^2_{\phi,\psi}(\Mp_4)$ and completes the proof of Proposition \ref{p:howe-ps}.

\section{$A$-packets of Soudry type}
\label{s:soudry}

In this section, we construct the $A$-packet associated to an $A$-parameter of Soudry type
\[
 \phi = \rho \boxtimes S_2
\]
with an irreducible dihedral cuspidal automorphic representation $\rho$ of $\GL_2(\A)$ with nontrivial quadratic central character.
This is the most troublesome nontempered $A$-packet in the sense that unlike the other nontempered $A$-packets, it cannot be constructed by using theta lifts from smaller orthogonal groups.
We will deal with it in the same way as \cite{gi-mp}, where we dealt with the tempered $A$-packets by using theta lifts to much larger orthogonal groups.

\subsection{Local $A$-packets}
\label{ss:local-soudry}

Suppose that $F$ is local and $\phi = \rho \boxtimes S_2$ is a local $A$-parameter with a $2$-dimensional orthogonal tempered representation $\rho$ of $L_F$.
Note that such $\rho$ can be regarded as a representation of $W_F$. 
We construct the associated $A$-packet
\[
 \Pi_{\phi, \psi}(\Mp_4) = \{ \pi_\eta \, | \, \eta \in \widehat{S}_\phi \}
\]
in a nonuniform way depending on whether $\rho$ is irreducible or not.

\subsubsection{The reducible case}

Suppose that $\rho$ is reducible.
Then $\rho$ is of the form either
\[
 \rho = \chi \oplus \chi^{-1}
\]
for some unitary character $\chi$ of $F^\times$ such that $\chi^2 \ne 1$, or
\[
 \rho = \chi_a \oplus \chi_b
\]
for some $a, b \in F^\times$.
In the former case, $S_\phi$ is trivial and we define the unique representation in $\Pi_{\phi, \psi}(\Mp_4)$ as the unique irreducible genuine representation of $\Mp_4(F)$ with $L$-parameter $\varphi_\phi$ associated to $\phi$.
In the latter case, $\phi$ is of Howe--Piatetski-Shapiro type and we have already constructed $\Pi_{\phi, \psi}(\Mp_4)$ in \S \ref{ss:local-howe-ps}.

\subsubsection{The irreducible case}

Suppose that $\rho$ is irreducible.
Then either $F$ is nonarchimedean or $F=\R$.
Also, we have $S_\phi \cong \Z/2\Z$, so that we may identify $\widehat{S}_\phi$ with $\mu_2$.
We define the $A$-packet $\Pi_{\phi, \psi}(\Mp_4)$ by 
\[
 \Pi_{\phi, \psi}(\Mp_4) = \{ \pi^\epsilon \, | \, \epsilon \in \mu_2 \} 
\]
with the representation $\pi^\epsilon$ given as follows.
Put 
\[
 \pi^+ = J_{P_2,\psi}(\tau \otimes |\det|^{\frac{1}{2}}),
\]
where $\tau$ is the irreducible square-integrable representation of $\GL_2(F)$ with $L$-parameter $\rho$.
When $F$ is nonarchimedean, we consider a symplectic representation $\varphi = \rho \boxtimes S_2$ of $L_F = W_F \times \SL_2(\C)$ and put 
\[
 \pi^- = \pi^-_\varphi, 
\]
where $\pi^-_\varphi$ is the irreducible genuine square-integrable representation of $\Mp_4(F)$ with $L$-parameter $\varphi$ (relative to $\psi$) associated to the nontrivial character of $S_\varphi \cong \Z/ 2\Z$.
When $F = \R$, we write $\rho = \mathcal{D}_\kappa$ with some positive integer $\kappa$ and put
\[
 \pi^- = \pi_\Lambda \oplus \pi_\Lambda^\vee,
\]
where $\pi_\Lambda$ is the genuine discrete series representation of $\Mp_4(\R)$ with lowest $\widetilde{\U}(2)$-type $\Lambda = (\kappa+\frac{3}{2}, \kappa+\frac{3}{2})$.
This is the only instance of a reducible representation of $\Mp_4(F)$ in an $A$-packet.

\subsection{Structure of $L^2_{\phi,\psi}(\Mp_4)$}

Suppose that $F$ is global.
For $\eta \in \widehat{S}_{\phi,\A}$, put
\[
 \pi_\eta = \bigotimes_v \pi_{\eta_v},
\]
where $\pi_{\eta_v}$ is the representation in the local $A$-packet $\Pi_{\phi_v, \psi_v}(\Mp_4)$ defined above.

\begin{prop}
\label{p:soudry}
We have 
\[
 L^2_{\phi,\psi}(\Mp_4) \cong \bigoplus_{\eta} \pi_\eta,
\]
where $\eta$ runs over elements in $\widehat{S}_{\phi,\A}$ such that $\Delta^* \eta = 1$.
\end{prop}

\subsection{The multiplicity preservation}

Since we have constructed the local $A$-packet $\Pi_{\phi_v, \psi_v}(\Mp_4)$ in a nonuniform way, it is somewhat tricky to construct the near equivalence class $L^2_{\phi,\psi}(\Mp_4)$.
We follow \cite{gi-mp} and appeal to some result of J.-S.~Li \cite{li97} on theta lifts in the stable range.
Namely, we fix an integer $r>3$ (in fact, we will take $r=4$ later) and consider (abstract) theta lifts from $\Mp_4(\A)$ to $\SO_{2r+5}(\A)$, where $\SO_{2r+5} = \SO(V_{r+2}^+)$ denotes the split odd special orthogonal group of rank $r+2$.
For any irreducible genuine unitary representation $\pi$ of $\Mp_4(\A)$, we may define an irreducible unitary representation $\theta_\psi(\pi)$ of $\SO_{2r+5}(\A)$ by 
\[
 \theta_\psi(\pi) = \bigotimes_v \theta_{W_{2,v}, V_{r+2,v}^+, \psi_v}(\pi_v).
\]
Let $L^2_{\theta(\phi)}(\SO_{2r+5})$ be the near equivalence class in the automorphic discrete spectrum of $\SO_{2r+5}$ determined by the elliptic $A$-parameter
\[
 \theta(\phi) = \phi \oplus (1 \boxtimes S_{2r}).
\]

\begin{prop}
\label{p:mult-preservation}
Writing
\[
 L^2_{\phi,\psi}(\Mp_4) \cong \bigoplus_\pi m_{\phi,\psi}(\pi) \pi,
\]
where $\pi$ runs over irreducible genuine unitary representations of $\Mp_4(\A)$ and $m_{\phi,\psi}(\pi)$ is the multiplicity of $\pi$ in $L^2_{\phi,\psi}(\Mp_4)$, we have
\[
 L^2_{\theta(\phi)}(\SO_{2r+5}) \cong \bigoplus_\pi m_{\phi,\psi}(\pi) \theta_\psi(\pi).
\]
\end{prop}

\begin{proof}
The proof is similar to that of \cite[Corollary 4.2]{gi-mp}.
Put
\begin{align*}
 m(\pi) & = \dim_\C \Hom_{\Mp_4(\A)}(\pi, \mathcal{A}(\Mp_4)), \\
 m_\disc(\pi) & = \dim_\C \Hom_{\Mp_4(\A)}(\pi, \mathcal{A}^2(\Mp_4)),
\end{align*}
where $\mathcal{A}(\Mp_4)$ is the space of genuine automorphic forms on $\Mp_4(\A)$ and $\mathcal{A}^2(\Mp_4)$ is the subspace of square-integrable automorphic forms.
Similarly, for any irreducible unitary representation $\sigma$ of $\SO_{2r+5}(\A)$, we may define its multiplicities $m(\sigma)$ and $m_\disc(\sigma)$.
Then the result of J.-S.~Li \cite{li97} says that 
\begin{equation}
\label{eq:JSLi}
 m_\disc(\pi) \le m_\disc(\theta_\psi(\pi)) \le m(\theta_\psi(\pi)) \le m(\pi). 
\end{equation}
This and the Howe duality imply that there is an embedding
\[
 \bigoplus_\pi m_{\phi,\psi}(\pi) \theta_\psi(\pi) \longhookrightarrow L^2_{\theta(\phi)}(\SO_{2r+5}).
\]
To show that this embedding is an isomorphism, it suffices to prove the following:
\begin{itemize}
\item $m_\disc(\pi) = m_\disc(\theta_\psi(\pi))$ for any irreducible genuine unitary representation $\pi$ of $\Mp_4(\A)$ such that the $L$-parameter of $\pi_v$ (relative to $\psi_v$) is $\varphi_{\phi_v}$ for almost all $v$;
\item any irreducible summand of $L^2_{\theta(\phi)}(\SO_{2r+5})$ is isomorphic to $\theta_\psi(\pi)$ for some irreducible summand $\pi$ of $L^2_{\phi,\psi}(\Mp_4)$.
\end{itemize}
The first assertion follows from \eqref{eq:JSLi} and Lemma \ref{l:soudry-mult} below.
The second assertion follows from the first one as in the proof of \cite[Corollary 4.2]{gi-mp}.
\end{proof}

\begin{lem}
\label{l:soudry-mult}
Let $\pi$ be an irreducible genuine unitary representation of $\Mp_4(\A)$ such that the $L$-parameter of $\pi_v$ (relative to $\psi_v$) is $\varphi_{\phi_v}$ for almost all $v$.
Then we have
\[
 m_\disc(\pi) = m(\pi).
\]
\end{lem}

\begin{proof}
We may assume that $m(\pi)>0$.
For any realization of $\pi$ on a subspace $\VV$ of $\mathcal{A}(\Mp_4)$, we need to show that $\VV$ is contained in $\mathcal{A}^2(\Mp_4)$.
We may further assume that $\VV$ is not contained in $\mathcal{A}_{\cusp}(\Mp_4)$.
Since the weak lift of $\pi$ to $\GL_4(\A)$ is 
\[
 (\rho \otimes |\det|^{\frac{1}{2}}) \boxplus (\rho \otimes |\det|^{-\frac{1}{2}}),
\]
where $\boxplus$ denotes the isobaric sum, it follows from the argument in the proof of \cite[Proposition 4.1]{gi-mp} that the cuspidal support of $\VV$ is on $P_2$ and is of the form
\[
 \rho \otimes |\det|^{\pm \frac{1}{2}}.
\]
In particular, $\pi$ is a subrepresentation of $I_{P_2,\psi}(\rho \otimes |\det|^{\pm \frac{1}{2}})$.
On the other hand, we have $\pi_v \cong J_{P_2,\psi_v}(\rho_v \otimes |\det|_v^{\frac{1}{2}})$ for almost all $v$, which does not occur as a subrepresentation of $I_{P_2,\psi_v}(\rho_v \otimes |\det|_v^{\frac{1}{2}})$.
Hence the cuspidal support of $\VV$ must be $\rho \otimes |\det|^{-\frac{1}{2}}$, so that $\VV$ is contained in $\mathcal{A}^2(\Mp_4)$ by the square-integrability criterion.
This completes the proof.
\end{proof}

\begin{rem}
\label{r:howe-ps}
In fact, Proposition \ref{p:mult-preservation} holds for other elliptic $A$-parameters for $\Mp_4$.
We give some details here in the case of $A$-parameters of Howe--Piatetski-Shapiro type, which we will use later.
Suppose that $\phi = (\chi_a \boxtimes S_2) \oplus (\chi_b \boxtimes S_2)$ with $a, b \in F^\times$ such that $\chi_a \ne \chi_b$.
Let $\pi$ be an irreducible genuine unitary representation of $\Mp_4(\A)$ such that the $L$-parameter of $\pi_v$ (relative to $\psi_v$) is $\varphi_{\phi_v}$ for almost all $v$, so that its weak lift to $\GL_4(\A)$ is
\[
 \chi_a |\cdot|^{\frac{1}{2}} \boxplus \chi_b |\cdot|^{\frac{1}{2}} \boxplus \chi_a |\cdot|^{-\frac{1}{2}} \boxplus \chi_b |\cdot|^{-\frac{1}{2}}.
\]
Then, for any realization of $\pi$ on a subspace $\VV$ of $\mathcal{A}(\Mp_4)$ which is not contained in $\mathcal{A}_{\cusp}(\Mp_4)$, it follows from the same argument that the cuspidal support of $\VV$ is on $B$ and is of the form $\chi_a |\cdot|^{-\frac{1}{2}} \otimes \chi_b |\cdot|^{-\frac{1}{2}}$.
This implies that $m_\disc(\pi) = m(\pi)$, from which we can deduce Proposition \ref{p:mult-preservation} for $\phi$.
\end{rem}

\subsection{Proof of Proposition \ref{p:soudry}}

We first apply Arthur's multiplicity formula to the near equivalence class $L^2_{\theta(\phi)}(\SO_{2r+5})$.
For the elliptic $A$-parameter $\theta(\phi) = \phi \oplus (1 \boxtimes S_{2r})$, we have
\[
 S_{\theta(\phi)} \cong S_\phi \oplus \Z/2\Z,
\]
where the extra factor $\Z/2\Z$ corresponds to $1 \boxtimes S_{2r}$.
Similarly, for each place $v$ of $F$ and the local $A$-parameter $\theta(\phi_v) = \phi_v \oplus (1 \boxtimes S_{2r})$, we have
\[
 S_{\theta(\phi_v)} \cong S_{\phi_v} \oplus \Z/2\Z.
\]
Put 
\[
 \overline{S}_{\theta(\phi_v)} = S_{\theta(\phi_v)} / \langle z_{\theta(\phi_v)} \rangle,
\]
where $z_{\theta(\phi_v)}$ is the image of $-1 \in \Sp_{2r+4}(\C)$ in $S_{\theta(\phi_v)}$.
Then we have a natural isomorphism
\[
 \iota_v : S_{\phi_v} \longhookrightarrow S_{\theta(\phi_v)} \longtwoheadrightarrow \overline{S}_{\theta(\phi_v)}.
\]
Let
\[
 \Pi_{\theta(\phi_v)}(\SO_{2r+5}) = \{ \sigma_{\xi_v} \, | \, \xi_v \in \widehat{\overline{S}}_{\theta(\phi_v)} \}
\]
be the local $A$-packet defined by Arthur \cite{arthur} consisting of some semisimple representations of $\SO_{2r+5}(F_v)$ of finite length.
For $\xi \in \widehat{\overline{S}}_{\theta(\phi), \A}$, put 
\[
 \sigma_\xi = \bigotimes_v \sigma_{\xi_v}.
\]
Then, noting that the character $\epsilon_{\theta(\phi)}$ associated to $\theta(\phi)$ is trivial, we deduce from Arthur's multiplicity formula \cite[Theorem 1.5.2]{arthur} that
\begin{equation}
\label{eq:mult-so}
 L^2_{\theta(\phi)}(\SO_{2r+5}) \cong \bigoplus_\xi \sigma_\xi,  
\end{equation}
where $\xi$ runs over elements in $\widehat{\overline{S}}_{\theta(\phi), \A}$ such that $\Delta^* \xi = 1$.

We now use Proposition \ref{p:mult-preservation} to transfer the structure of $L^2_{\theta(\phi)}(\SO_{2r+5})$ to $L^2_{\phi,\psi}(\Mp_4)$.
For $\eta_v \in \widehat{S}_{\phi_v}$ and $\xi_v = \eta_v \circ \iota_v^{-1} \in \widehat{\overline{S}}_{\theta(\phi_v)}$, write 
\[
 \sigma_{\xi_v} = \bigoplus_i m_{\xi_v,i} \sigma_{\xi_v,i}
\]
with some positive integers $m_{\xi_v,i}$ and some pairwise distinct irreducible representations $\sigma_{\xi_v,i}$ of $\SO_{2r+5}(F_v)$, and put 
\begin{equation}
\label{eq:tilde_pi_eta}
 \tilde{\pi}_{\eta_v} = \bigoplus_i m_{\xi_v,i} \theta_{W_{2,v}, V_{r+2,v}^+,\psi_v}(\sigma_{\xi_v,i}).
\end{equation}
For $\eta \in \widehat{S}_{\phi,\A}$, put
\[
 \tilde{\pi}_\eta = \bigotimes_v \tilde{\pi}_{\eta_v}.
\]
Then, by Proposition \ref{p:mult-preservation} and \eqref{eq:mult-so}, we have
\[
 L^2_{\phi,\psi}(\Mp_4) \cong \bigoplus_{\eta} \tilde{\pi}_\eta,
\]
where $\eta$ runs over elements in $\widehat{S}_{\phi,\A}$ such that $\Delta^* \eta = 1$.
Thus, to finish the proof of Proposition \ref{p:soudry}, we may take $r=4$ and need to prove the following.

\begin{prop}
\label{p:soudry-local}
Assume that $r=4$.
Then we have
\[
 \tilde{\pi}_{\eta_v} = \pi_{\eta_v}
\]
for $\eta_v \in \widehat{S}_{\phi_v}$.
Here $\tilde{\pi}_{\eta_v}$ is defined by \eqref{eq:tilde_pi_eta}, whereas $\pi_{\eta_v}$ is defined in \S \ref{ss:local-soudry}.
\end{prop}

This proposition will be proved in the next section.
In fact, it holds for any $r>3$ when either $F$ is nonarchimedean or $F=\C$.

\begin{rem}
We recall the analog of \eqref{eq:mult-so} here in the case of $A$-parameters of Howe--Piatetski-Shapiro type, which we will use later. 
Suppose that $\phi = (\chi_a \boxtimes S_2) \oplus (\chi_b \boxtimes S_2)$ with $a, b \in F^\times$ such that $\chi_a \ne \chi_b$ and put $\theta(\phi) = \phi \oplus (1 \boxtimes S_{2r})$.
For each place $v$ of $F$, let
\[
 \Pi_{\theta(\phi_v)}(\SO_{2r+5}) = \{ \sigma_v^{\epsilon_{1,v},\epsilon_{2,v}} \, | \, \epsilon_{1,v}, \epsilon_{2,v} \in \mu_2 \}
\]
be the local $A$-packet defined by Arthur \cite{arthur} consisting of some semisimple representations of $\SO_{2r+5}(F_v)$ of finite length, where we identify $\widehat{\overline{S}}_{\theta(\phi_v)}$ with 
\[
\begin{cases}
 \mu_2 \times \mu_2 & \text{if $\chi_{a,v} \ne \chi_{b,v}$;} \\
 \Delta \mu_2 & \text{if $\chi_{a,v} = \chi_{b,v}$,}
\end{cases}
\]
and interpret $\sigma_v^{\epsilon_{1,v},\epsilon_{2,v}}$ as zero if $\chi_{a,v} = \chi_{b,v}$ and $\epsilon_{1,v} \ne \epsilon_{2,v}$.
For $\epsilon_1, \epsilon_2 \in \mu_{2,\A}$, put
\[
 \sigma^{\epsilon_1,\epsilon_2} = \bigotimes_v \sigma_v^{\epsilon_{1,v},\epsilon_{2,v}}.
\]
Then, noting that the character $\epsilon_{\theta(\phi)}$ associated to $\theta(\phi)$ is trivial, we deduce from Arthur's multiplicity formula \cite[Theorem 1.5.2]{arthur} that
\begin{equation}
\label{eq:mult-so2}
 L^2_{\theta(\phi)}(\SO_{2r+5}) \cong \bigoplus_{\epsilon_1,\epsilon_2} \sigma^{\epsilon_1,\epsilon_2}, 
\end{equation}
where $\epsilon_1,\epsilon_2$ run over elements in $\mu_{2,\A}$ such that $\prod_v \epsilon_{1,v} = \prod_v \epsilon_{2,v} = 1$.
\end{rem}

\section{Proof of Proposition \ref{p:soudry-local}}
\label{s:pf-soudry}

Suppose that $F$ is local and consider local $A$-parameters 
\[
 \phi = \rho \boxtimes S_2, \qquad 
 \theta(\phi) = \phi \oplus (1 \boxtimes S_{2r})
\]
with a $2$-dimensional orthogonal tempered representation $\rho$ of $L_F$ and an integer $r>3$.
Let 
\[
 \Pi_{\phi,\psi}(\Mp_4) = \{ \pi_\eta \, | \, \eta \in \widehat{S}_\phi \}, \qquad
 \Pi_{\theta(\phi)}(\SO_{2r+5}) = \{ \sigma_\xi \, | \, \xi \in \widehat{\overline{S}}_{\theta(\phi)} \}
\]
be the associated $A$-packets.
Note that $\pi_\eta$ is irreducible unless $F = \R$, $\rho$ is irreducible, and $\eta$ is nontrivial.
We denote by
\[
 \theta_\psi = \theta_{W_2,V_{r+2}^+,\psi}
\]
the theta lift from $\Mp_4 = \Mp(W_2)$ to $\SO_{2r+5} = \SO(V_{r+2}^+)$.
For any semisimple genuine representation $\pi$ of $\Mp_4(F)$ of finite length, write $\pi = \bigoplus_i m_i \pi_i$ with some positive integers $m_i$ and some pairwise distinct irreducible genuine representations $\pi_i$ of $\Mp_4(F)$, and put 
\[
 \theta_\psi(\pi) = \bigoplus_i m_i \theta_\psi(\pi_i).
\]
Then we need to show that
\[
 \theta_\psi(\pi_\eta) = \sigma_{\xi}
\]
for $\eta \in \widehat{S}_\phi$ and $\xi = \eta \circ \iota^{-1} \in \widehat{\overline{S}}_{\theta(\phi)}$ (under the assumption that $r=4$ when $F=\R$), where $\iota : S_{\phi} \hookrightarrow S_{\theta(\phi)} \twoheadrightarrow \overline{S}_{\theta(\phi)}$ is the natural isomorphism.

We will consider the various cases in turn.

\subsection{The reducible case I}

Suppose that $\rho = \chi \oplus \chi^{-1}$ for some unitary character $\chi$ of $F^\times$ such that $\chi^2 \ne 1$.
We denote by $\pi_\phi$ and $\sigma_{\theta(\phi)}$ the unique representations in $\Pi_{\phi,\psi}(\Mp_4)$ and $\Pi_{\theta(\phi)}(\SO_{2r+5})$, respectively.
By definition, we have
\[
 \pi_\phi = J_{B, \psi}(\chi |\cdot|^{\frac{1}{2}}, \chi^{-1} |\cdot|^{\frac{1}{2}}).
\]
Also, by the inductive property of $A$-packets \cite[\S 4.2]{moeglin11}, \cite[\S 6]{mr-c}, \cite[\S 5]{mr3}, we have
\[
 \sigma_{\theta(\phi)} = I_{Q_2}(\chi \circ \det, 1), 
\]
where the right-hand side is irreducible.
In fact, the right-hand side is a quotient of
\[
 I_Q(\chi |\cdot|^{\frac{1}{2}}, \chi |\cdot|^{-\frac{1}{2}}, |\cdot|^{r-\frac{1}{2}}, |\cdot|^{r-\frac{3}{2}}, \dots, |\cdot|^{\frac{1}{2}})
 \cong
 I_Q(|\cdot|^{r-\frac{1}{2}}, |\cdot|^{r-\frac{3}{2}}, \dots, |\cdot|^{\frac{1}{2}}, \chi |\cdot|^{\frac{1}{2}}, \chi^{-1} |\cdot|^{\frac{1}{2}}), 
\]
where $Q = Q_{(1^{r+2})}$, so that $\sigma_{\theta(\phi)}$ is the unique irreducible representation of $\SO_{2r+5}(F)$ with $L$-parameter $\varphi_{\theta(\phi)}$.
Hence, by Lemma \ref{l:ind-princ-2} below, we have
\[
 \theta_\psi(\pi_\phi) = \sigma_{\theta(\phi)}.
\]

\subsection{The reducible case II}

Suppose that $\rho = \chi_a \oplus \chi_b$ for some $a, b \in F^\times$.
For $\epsilon_1, \epsilon_2 \in \mu_2$, we denote by $\pi^{\epsilon_1,\epsilon_2}$ and $\sigma^{\epsilon_1,\epsilon_2}$ the representations in $\Pi_{\phi,\psi}(\Mp_4)$ and $\Pi_{\theta(\phi)}(\SO_{2r+5})$ associated to $(\epsilon_1,\epsilon_2)$, respectively, where we identify $\widehat{S}_\phi \cong \widehat{\overline{S}}_{\theta(\phi)}$ with a subgroup of $\mu_2 \times \mu_2$.
By definition, we have
\[
 \pi^{+,+} = \pi_\phi, 
\]
where $\pi_\phi$ is the unique irreducible genuine representation of $\Mp_4(F)$ with $L$-parameter $\varphi_\phi$ (relative to $\psi$).
Also, by \cite[Proposition 7.4.1]{arthur}, we have
\begin{equation}
\label{eq:sigma++}
 \sigma^{+,+} \supset \sigma_{\theta(\phi)}, 
\end{equation}
where $\sigma_{\theta(\phi)}$ is the unique irreducible representation of $\SO_{2r+5}(F)$ with $L$-parameter $\varphi_{\theta(\phi)}$.
Moreover, if $F$ is nonarchimedean and $\chi_a \ne \chi_b$, then the $A$-parameter $\theta(\phi)$ is elementary in the sense of M{\oe}glin \cite{moeglin06} and we can describe $\sigma^{\epsilon_1,\epsilon_2}$ explicitly.
For example, if further $\chi_a, \chi_b, 1$ are pairwise distinct, then we have 
\begin{align*}
 \sigma^{+,+} & = J_Q(|\cdot|^{r-\frac{1}{2}}, |\cdot|^{r-\frac{3}{2}}, \dots, |\cdot|^{\frac{1}{2}}, \chi_a |\cdot|^{\frac{1}{2}}, \chi_b |\cdot|^{\frac{1}{2}}), \\
 \sigma^{+,-} & = J_{Q'}(|\cdot|^{r-\frac{1}{2}}, |\cdot|^{r-\frac{3}{2}}, \dots, |\cdot|^{\frac{3}{2}}, \chi_a |\cdot|^{\frac{1}{2}}, \sigma^-_{\varphi_b}), \\
 \sigma^{-,+} & = J_{Q'}(|\cdot|^{r-\frac{1}{2}}, |\cdot|^{r-\frac{3}{2}}, \dots, |\cdot|^{\frac{3}{2}}, \chi_b |\cdot|^{\frac{1}{2}}, \sigma^-_{\varphi_a})
\end{align*}
by her explicit construction of $A$-packets (see also \cite[\S 6]{xu}).
Here $Q = Q_{(1^{r+2})}$, $Q' = Q_{(1^r)}$, and for any $c \in F^\times$ such that $\chi_c \ne 1$, $\sigma^-_{\varphi_c}$ is the irreducible square-integrable representation of $\SO_5(F)$ with $L$-parameter $\varphi_c = (\chi_c \boxtimes S_2) \oplus (1 \boxtimes S_2)$ associated to the nontrivial character of $\overline{S}_{\varphi_c} \cong \Z/2\Z$.

\begin{lem}
\label{l:mult-free}
The $A$-packet $\Pi_{\theta(\phi)}(\SO_{2r+5})$ is multiplicity-free.
\end{lem}

\begin{proof}
In a more general context, the assertion was completely proved by M{\oe}glin \cite{moeglin11} when $F$ is nonarchimedean and by M{\oe}glin--Renard \cite{mr-c} when $F = \C$, and was largely proved by M{\oe}glin--Renard \cite{mr1,mr2} when $F = \R$.
Here we give a proof based on theta lifts.

Fix the following data:
\begin{itemize}
\item a number field $\F$ with ad\`ele ring $\A$ and distinct places $v_0,v_1$ such that $\F_{v_0} = \F_{v_1} = F$;
\item a nontrivial additive character $\varPsi$ of $\A/\F$ such that $\varPsi_{v_0}, \varPsi_{v_1}$ belong to the $(F^\times)^2$-orbit of $\psi$;
\item $\alpha, \beta \in \F^\times$ such that $\chi_{\alpha,v_0} = \chi_{\alpha,v_1} = \chi_a$, $\chi_{\beta,v_0} = \chi_{\beta,v_1} = \chi_b$, and $\chi_\alpha \ne \chi_\beta$.
\end{itemize}
Put 
\[
 \varPhi = (\chi_\alpha \boxtimes S_2) \oplus (\chi_\beta \boxtimes S_2), 
\]
so that $\varPhi_{v_0} = \varPhi_{v_1} = \phi$.
Then, by \eqref{eq:mult-so2}, there is an embedding
\[
 \Big( \bigoplus_{\epsilon_1, \epsilon_2 \in \mu_2} ( \sigma^{\epsilon_1,\epsilon_2} \otimes \sigma^{\epsilon_1,\epsilon_2} ) \Big)
 \otimes \Big( \bigotimes_{v \ne v_0, v_1} \sigma_{\theta(\varPhi_v)} \Big)
 \longhookrightarrow L^2_{\theta(\varPhi)}(\SO_{2r+5}),
\]
where $\theta(\varPhi) = \varPhi \oplus (1 \boxtimes S_{2r})$ and $\sigma_{\theta(\varPhi_v)}$ is the unique irreducible representation of $\SO_{2r+5}(\F_v)$ with $L$-parameter $\varphi_{\theta(\varPhi_v)}$.
Hence, for any irreducible representation $\sigma$ of $\SO_{2r+5}(F)$, we have
\[
 \sum_{\epsilon_1, \epsilon_2 \in \mu_2} m^{\epsilon_1,\epsilon_2}(\sigma)^2 \le m_\disc(\varSigma),
\]
where
\[
 m^{\epsilon_1,\epsilon_2}(\sigma) = \dim_\C \Hom_{\SO_{2r+5}(F)}(\sigma,\sigma^{\epsilon_1,\epsilon_2})
\]
and $\varSigma$ is the irreducible representation of $\SO_{2r+5}(\A)$ such that
\begin{itemize}
\item $\varSigma_{v_0} = \varSigma_{v_1} = \sigma$;
\item $\varSigma_v = \sigma_{\theta(\varPhi_v)}$ for all $v \ne v_0,v_1$.
\end{itemize}
On the other hand, by Proposition \ref{p:howe-ps} and the analog of Proposition \ref{p:mult-preservation} for $A$-parameters of Howe--Piatetski-Shapiro type (see Remark \ref{r:howe-ps}), we have $m_\disc(\varSigma) \le 1$, so that
\[
 \sum_{\epsilon_1, \epsilon_2 \in \mu_2} m^{\epsilon_1,\epsilon_2}(\sigma)^2 \le 1.
\]
This implies the assertion. 
\end{proof}

We denote by $\JH(\Pi_{\theta(\phi)}(\SO_{2r+5}))$ the set of irreducible representations of $\SO_{2r+5}(F)$ occurring in $\sigma^{\epsilon_1,\epsilon_2}$ for some $\epsilon_1,\epsilon_2 \in \mu_2$.

\begin{lem}
\label{l:bij}
The theta lift induces a bijection
\[
 \theta_\psi : \Pi_{\phi,\psi}(\Mp_4) \longrightarrow \JH(\Pi_{\theta(\phi)}(\SO_{2r+5})).
\]
\end{lem}

\begin{proof}
Let $\F, v_0, v_1, \varPhi, \varPsi$ be as in the proof of Lemma \ref{l:mult-free}.
For a given $\pi \in \Pi_{\phi,\psi}(\Mp_4)$, let $\varPi$ be the irreducible genuine representation of $\Mp_4(\A)$ such that 
\begin{itemize}
\item $\varPi_{v_0} = \varPi_{v_1} = \pi$;
\item $\varPi_v = \pi_{\varPhi_v}$ for all $v \ne v_0,v_1$,
\end{itemize}
where $\pi_{\varPhi_v}$ is the unique irreducible genuine representation of $\Mp_4(\F_v)$ with $L$-parameter $\varphi_{\varPhi_v}$ (relative to $\varPsi_v$).
Then $\varPi$ occurs in $L^2_{\varPhi,\varPsi}(\Mp_4)$ by Proposition \ref{p:howe-ps}, so that $\theta_\varPsi(\varPi)$ occurs in $L^2_{\theta(\varPhi)}(\SO_{2r+5})$ by Remark \ref{r:howe-ps}.
In particular, we have
\[
 \theta_\psi(\pi) = \theta_\varPsi(\varPi)_{v_0} \in \JH(\Pi_{\theta(\phi)}(\SO_{2r+5})).
\]
This defines a map $\theta_\psi:\Pi_{\phi,\psi}(\Mp_4) \rightarrow \JH(\Pi_{\theta(\phi)}(\SO_{2r+5}))$, which is injective by the Howe duality.

Conversely, for a given $\sigma \in \JH(\Pi_{\theta(\phi)}(\SO_{2r+5}))$, let $\varSigma$ be the irreducible representation of $\SO_{2r+5}(\A)$ such that
\begin{itemize}
\item $\varSigma_{v_0} = \varSigma_{v_1} = \sigma$;
\item $\varSigma_v = \sigma_{\theta(\varPhi_v)}$ for all $v \ne v_0,v_1$.
\end{itemize}
Then $\varSigma$ occurs in $L^2_{\theta(\varPhi)}(\SO_{2r+5})$ by \eqref{eq:mult-so2}, so that $\theta_\varPsi(\varSigma)$ occurs in $L^2_{\varPhi,\varPsi}(\Mp_4)$ by Remark \ref{r:howe-ps}.
In particular, we have
\[
 \theta_\psi(\sigma) = \theta_\varPsi(\varSigma)_{v_0} \in \Pi_{\phi,\psi}(\Mp_4).
\]
This implies that the map $\theta_\psi: \Pi_{\phi,\psi}(\Mp_4) \rightarrow \JH(\Pi_{\theta(\phi)}(\SO_{2r+5}))$ is surjective.
\end{proof}

\begin{lem}
\label{l:irred}
For any $\epsilon_1, \epsilon_2 \in \mu_2$, $\sigma^{\epsilon_1,\epsilon_2}$ is either zero or irreducible.
\end{lem}

\begin{proof}
Suppose on the contrary that $\sigma^{\epsilon_1,\epsilon_2}$ is nonzero and reducible.
Then, by Lemma \ref{l:mult-free}, we have $\sigma^{\epsilon_1,\epsilon_2} \supset \sigma' \oplus \sigma''$ for some inequivalent irreducible representations $\sigma',\sigma''$ of $\SO_{2r+5}(F)$.
Let $\F, v_0, v_1, \varPhi, \varPsi$ be as in the proof of Lemma \ref{l:mult-free}.
Let $\varSigma$ be the irreducible representation of $\SO_{2r+5}(\A)$ such that 
\begin{itemize}
\item $\varSigma_{v_0} = \sigma'$;
\item $\varSigma_{v_1} = \sigma''$;
\item $\varSigma_v = \sigma_{\theta(\varPhi_v)}$ for all $v \ne v_0,v_1$.
\end{itemize}
Then $\varSigma$ occurs in $L^2_{\theta(\varPhi)}(\SO_{2r+5})$ by \eqref{eq:mult-so2}, so that $\varPi = \theta_\varPsi(\varSigma)$ occurs in $L^2_{\varPhi,\varPsi}(\Mp_4)$ by Remark \ref{r:howe-ps}.
On the other hand, by Lemma \ref{l:bij}, we have
\[
 \varPi_{v_0} = \theta_\psi(\sigma') = \pi^{\epsilon_1',\epsilon_2'}, \qquad
 \varPi_{v_1} = \theta_\psi(\sigma'') = \pi^{\epsilon_1'',\epsilon_2''}
\]
for some $\epsilon_1', \epsilon_2', \epsilon_1'', \epsilon_2'' \in \mu_2$ such that $(\epsilon_1',\epsilon_2') \ne (\epsilon_1'',\epsilon_2'')$.
Also, by Lemma \ref{l:ind-princ-2} below, we have
\[
 \varPi_v = \theta_{\varPsi_v}(\sigma_{\theta(\varPhi_v)}) = \pi_{\varPhi_v}
\]
for all $v \ne v_0,v_1$.
Hence it follows from Proposition \ref{p:howe-ps} that $\varPi$ does not occur in $L^2_{\varPhi,\varPsi}(\Mp_4)$.
This is a contradiction and completes the proof.
\end{proof}

By Lemmas \ref{l:bij} and \ref{l:irred}, the theta lift induces a bijection
\[
 \theta_\psi : \Pi_{\phi,\psi}(\Mp_4) \longrightarrow \Pi_{\theta(\phi)}(\SO_{2r+5}).
\]
We need to show that 
\[
 \theta_\psi(\pi^{\epsilon_1,\epsilon_2}) = \sigma^{\epsilon_1,\epsilon_2} 
\]
for $\epsilon_1,\epsilon_2 \in \mu_2$.

By Lemma \ref{l:ind-princ-2} below, we have
\[
 \theta_\psi(\pi^{+,+}) = \sigma^{+,+},
\]
noting that $\pi^{+,+} = \pi_\phi$ and $\sigma^{+,+} = \sigma_{\theta(\phi)}$ (where the latter follows from \eqref{eq:sigma++} and Lemma \ref{l:irred}).
This implies that $\theta_\psi(\pi^{-,-}) = \sigma^{-,-}$ when $\chi_a = \chi_b$.
Also, we will show in \S \ref{sss:howe-ps-nonarch} below that 
\begin{equation}
\label{eq:howe-ps-nonarch}
 \theta_\psi(\pi^{+,-}) = \sigma^{+,-}, \qquad 
 \theta_\psi(\pi^{-,+}) = \sigma^{-,+}
\end{equation}
when $F$ is nonarchimedean and $\chi_a,\chi_b,1$ are pairwise distinct.
This implies that $\theta_\psi(\pi^{-,-}) = \sigma^{-,-}$ in this case.

Finally, suppose that $F$ is arbitrary and that $\chi_a \ne \chi_b$.
Fix the following data:
\begin{itemize}
\item a number field $\F$ with ad\`ele ring $\A$ and distinct places $v_0,v_1$ such that $\F_{v_0} = F$ and $\F_{v_1}$ is nonarchimedean;
\item a nontrivial additive character $\varPsi$ of $\A/\F$ such that $\varPsi_{v_0}$ belongs to the $(F^\times)^2$-orbit of $\psi$;
\item $\alpha, \beta \in \F^\times$ such that $\chi_{\alpha,v_0} = \chi_a$, $\chi_{\beta,v_0} = \chi_b$, and $\chi_{\alpha,v_1}, \chi_{\beta,v_1}, 1$ are pairwise distinct.
\end{itemize}
Put 
\[
 \varPhi = (\chi_\alpha \boxtimes S_2) \oplus (\chi_\beta \boxtimes S_2), 
\]
so that $\varPhi_{v_0} = \phi$.
For given $\epsilon_1, \epsilon_2 \in \mu_2$, let $\varSigma$ be the irreducible representation of $\SO_{2r+5}(\A)$ such that
\begin{itemize}
\item $\varSigma_{v_0} = \sigma^{\epsilon_1,\epsilon_2}$;
\item $\varSigma_{v_1} = \sigma_{v_1}^{\epsilon_1,\epsilon_2}$;
\item $\varSigma_v = \sigma_{\theta(\varPhi_v)}$ for all $v \ne v_0,v_1$.
\end{itemize}
Then $\varSigma$ occurs in $L^2_{\theta(\varPhi)}(\SO_{2r+5})$ by \eqref{eq:mult-so2}, so that $\varPi = \theta_\varPsi(\varSigma)$ occurs in $L^2_{\varPhi,\varPsi}(\Mp_4)$ by Remark \ref{r:howe-ps}.
On the other hand, we have
\[
 \varPi_{v_0} = \theta_\psi(\sigma^{\epsilon_1,\epsilon_2}) = \pi^{\epsilon_1',\epsilon_2'}
\]
for some $\epsilon_1',\epsilon_2' \in \mu_2$.
Since we have already shown that
\[
 \varPi_{v_1} = \theta_{\varPsi_{v_1}}(\sigma_{v_1}^{\epsilon_1,\epsilon_2}) = \pi_{v_1}^{\epsilon_1,\epsilon_2}
\]
and
\[
 \varPi_v = \theta_{\varPsi_v}(\sigma_{\theta(\varPhi_v)}) = \pi_{\varPhi_v}
\]
for all $v \ne v_0,v_1$, it follows from Proposition \ref{p:howe-ps} that $(\epsilon_1',\epsilon_2') = (\epsilon_1,\epsilon_2)$.

This completes the proof of Proposition \ref{p:soudry-local} when $\rho$ is reducible.

\subsection{The irreducible case}

Suppose that $\rho$ is irreducible.
Then either $F$ is nonarchimedean or $F = \R$.
For $\epsilon \in \mu_2$, we denote by $\pi^\epsilon$ and $\sigma^\epsilon$ the representations in $\Pi_{\phi,\psi}(\Mp_4)$ and $\Pi_{\theta(\phi)}(\SO_{2r+5})$ associated to $\epsilon$, respectively, where we identify $\widehat{S}_\phi \cong \widehat{\overline{S}}_{\theta(\phi)}$ with $\mu_2$.

\subsubsection{The nonarchimedean case}

Suppose that $F$ is nonarchimedean.
By definition, we have
\[
 \pi^+ = J_{P_2,\psi}(\tau \otimes |\det|^{\frac{1}{2}}), \qquad
 \pi^- = \pi^-_\varphi,
\]
where $\tau$ is the irreducible supercuspidal representation of $\GL_2(F)$ with $L$-parameter $\rho$ and $\pi^-_\varphi$ is the irreducible genuine square-integrable representation of $\Mp_4(F)$ with $L$-parameter $\varphi = \rho \boxtimes S_2$ (relative to $\psi$) associated to the nontrivial character of $S_\varphi \cong \Z/2\Z$.
Moreover, the $A$-parameter $\theta(\phi)$ is elementary in the sense of M{\oe}glin \cite{moeglin06} and we have 
\begin{align*}
 \sigma^+ & = J_Q(|\cdot|^{r-\frac{1}{2}}, |\cdot|^{r-\frac{3}{2}}, \dots, |\cdot|^{\frac{1}{2}}, \tau \otimes |\det|^{\frac{1}{2}}), \\
 \sigma^- & = J_{Q'}(|\cdot|^{r-\frac{1}{2}}, |\cdot|^{r-\frac{3}{2}}, \dots, |\cdot|^{\frac{3}{2}}, \sigma^-_{\varphi'})
\end{align*}
by her explicit construction of $A$-packets (see also \cite[\S 6]{xu}).
Here $Q = Q_{(1^r,2)}$, $Q' = Q_{(1^{r-1})}$, and $\sigma^-_{\varphi'}$ is the irreducible square-integrable representation of $\SO_7(F)$ with $L$-parameter $\varphi' = (\rho \boxtimes S_2) \oplus (1 \boxtimes S_2)$ associated to the nontrivial character of $\overline{S}_{\varphi'} \cong \Z/2\Z$.
By Lemma \ref{l:ind-princ-2} below, we have
\[
 \theta_\psi(\pi^+) = \sigma^+.
\]
Also, by \cite[Theorem 1.4]{atobe-gan}, we have
\[
 \theta_\psi(\pi^-) = \sigma^-. 
\]

\subsubsection{The real case}

Suppose that $F = \R$.
When $r=4$, we will show in \S \ref{sss:soudry-real} below that
\begin{equation}
\label{eq:soudry-real}
 \theta_\psi(\pi^\epsilon) = \sigma^\epsilon 
\end{equation}
for $\epsilon \in \mu_2$.

This completes the proof of Proposition \ref{p:soudry-local} when $\rho$ is irreducible.

\appendix

\section{Local theta lifts from $\O_3$ to $\Mp_4$}

In this appendix, we prove some properties of the local theta lift used in \S\S \ref{s:sk} and \ref{s:howe-ps}.
Let $F$ be a local field of characteristic zero.

\subsection{Properties of the theta lift}

We consider the theta lift from $\O(V_1^\epsilon)$ to $\Mp(W_2)$.

\begin{lem} 
\label{l:o3mp4-nonzero}
Let $\sigma$ be an irreducible representation of $\O(V_1^\epsilon)$.
Then we have
\[
 \theta_{W_2,V_1^\epsilon,\psi}(\sigma) \ne 0
 \, \Longleftrightarrow \, \sigma \ne \det,
\]
where $\det$ is the nontrivial character of $\O(V_1^\epsilon)$ trivial on $\SO(V_1^\epsilon)$.
\end{lem}

\begin{proof}
Put
\begin{align*}
 n^+(\sigma) & = \min \{ n \ge 0 \, | \, \theta_{W_n,V_1^\epsilon,\psi}(\sigma) \ne 0 \}, \\
 n^-(\sigma) & = \min \{ n \ge 0 \, | \, \theta_{W_n,V_1^\epsilon,\psi}(\sigma \otimes \det) \ne 0 \}.
\end{align*}
Since $n^-(\sigma) = 0$ if and only if $\sigma = \det$, the assertion follows from the conservation relation 
\[
 n^+(\sigma) + n^-(\sigma) = 3
\]
proved by Sun--Zhu \cite{sz}.
\end{proof}

\begin{lem}
\label{l:o3mp4-mult-free}
Let $\sigma_1$ and $\sigma_2$ be irreducible representations of $\O(V_1^{\epsilon_1})$ and $\O(V_1^{\epsilon_2})$, respectively.
Assume that $\theta_{W_2,V_1^{\epsilon_1},\psi}(\sigma_1)$ and $\theta_{W_2,V_1^{\epsilon_2},\psi}(\sigma_2)$ are nonzero.
Then we have
\[
 \theta_{W_2,V_1^{\epsilon_1},\psi}(\sigma_1) = \theta_{W_2,V_1^{\epsilon_2},\psi}(\sigma_2)
 \, \Longleftrightarrow \, \epsilon_1 = \epsilon_2, \, \sigma_1 = \sigma_2.
\]
\end{lem}

\begin{proof}
By the Howe duality, it suffices to show that if $\theta_{W_2,V_1^{\epsilon_1},\psi}(\sigma_1) = \theta_{W_2,V_1^{\epsilon_2},\psi}(\sigma_2)$, then $\epsilon_1 = \epsilon_2$.
We may assume that $F \ne \C$.
For any $\epsilon \in \mu_2$ and any irreducible genuine representation $\pi$ of $\Mp(W_2)$, put
\[
 r^\epsilon(\pi) = \min \{ r \ge 0 \, | \, \theta_{W_2,V_r,\psi}(\pi) \ne 0, \, \epsilon(V_r) = \epsilon \},
\]
where $\epsilon(V_r)$ denotes the Hasse--Witt invariant of a $(2r+1)$-dimensional quadratic space $V_r$ over $F$ with trivial discriminant.
By the conservation relation 
\[
 r^+(\pi) + r^-(\pi) =5
\]
proved by Sun--Zhu \cite{sz}, we have $r^\epsilon(\pi) \ge 3$ for some $\epsilon \in \mu_2$.
This implies the assertion.
\end{proof}

\begin{lem}
\label{l:ind-princ-1}
Let $\sigma$ be an irreducible representation of $\O(V_1^\epsilon)$.
Assume that either $\sigma$ is tempered, or $\epsilon = 1$ and $\sigma|_{\SO(V_1^\epsilon)} = I_{Q_1}(\chi |\cdot|^s)$ for some unitary character $\chi$ of $F^\times$ and some $s \in \R$ with $|s|<\frac{1}{2}$.
If
\begin{equation}
\label{eq:root}
 \sigma(-1) = \epsilon \cdot \epsilon(\tfrac{1}{2}, \sigma), 
\end{equation}
then we have
\[
 \theta_{W_2,V_1^\epsilon,\psi}(\sigma) = J_{P_1, \psi}(|\cdot|^{\frac{1}{2}}, \theta_{W_1,V_1^\epsilon,\psi}(\sigma)),
\]
where $\epsilon(s, \sigma)$ is the standard $\epsilon$-factor of $\sigma$ relative to $\psi$ (whose value at $s = \frac{1}{2}$ does not depend on $\psi$).
\end{lem}

\begin{proof}
Recall that 
\begin{equation}
\label{eq:dichotomy}
 \theta_{W_1,V_1^\epsilon,\psi}(\sigma) \ne 0 \, \Longleftrightarrow \, 
 \sigma(-1) = \epsilon \cdot \epsilon(\tfrac{1}{2}, \sigma)
\end{equation}
by the result of Waldspurger \cite{w1,w2} (see also \cite[\S 5]{gs}).
If we put $\pi = \theta_{W_1,V_1^\epsilon,\psi}(\sigma)$, then by assumption, either $\pi$ is tempered, or $\pi = I_{B,\psi}(\chi |\cdot|^s)$ for some unitary character $\chi$ of $F^\times$ and some $s \in \R$ with $|s| < \frac{1}{2}$.
When $F$ is nonarchimedean, the assertion follows from \cite[Proposition 3.2]{gt1} (which is stated for $\pi$ tempered but continues to hold for $\pi = I_{B,\psi}(\chi |\cdot|^s)$ as above).
When $F = \R$, the assertion will be proved in \S \ref{sss:ind-princ-1-real} below.
When $F = \C$, the assertion follows from \cite[Theorem 2.8]{ab1}.
\end{proof}

Recall that $\SO(V_1^\epsilon) \cong (B^\epsilon)^\times/F^\times$ for some quaternion algebra $B^\epsilon$ over $F$.
For any quadratic character $\chi$ of $F^\times$, we may regard $\chi \circ \mathrm{N}_{B^\epsilon}$ as a representation of $\SO(V_1^\epsilon)$, where $\mathrm{N}_{B^\epsilon}$ is the reduced norm on $B^\epsilon$.

\begin{lem}
\label{l:howe-ps-local+}
Let $\sigma$ be a $1$-dimensional representation of $\O(V_1^+)$ such that $\sigma|_{\SO(V_1^+)} = \chi \circ \mathrm{N}_{B^+}$ for some quadratic character $\chi$ of $F^\times$.
\begin{enumerate}
\item 
\label{item:howe-ps-1}
If $\sigma(-1) = \chi(-1)$, then we have
\[
 \theta_{W_2, V_1^+, \psi}(\sigma) = J_{B, \psi}(\chi |\cdot|^{\frac{1}{2}}, |\cdot|^{\frac{1}{2}}).
\]
\item
\label{item:howe-ps-2}
If $\sigma(-1) = -\chi(-1)$ and $\chi \ne 1$, then we have
\[
 \theta_{W_2, V_1^+, \psi}(\sigma) = J_{P_1, \psi}(\chi |\cdot|^{\frac{1}{2}}, \omega_{W_1,\psi}^-).
\]
\end{enumerate}
\end{lem}

\begin{proof}
When $F$ is nonarchimedean or $F = \R$, the assertion will be proved in \S\S \ref{sss:howe-ps-local+-nonarch} and \ref{sss:howe-ps-local+-real} below.
When $F = \C$, the assertion follows from \cite[Theorem 2.8]{ab1}.
\end{proof}

\begin{lem}
\label{l:howe-ps-local-}
Let $\sigma$ be a $1$-dimensional representation of $\O(V_1^-)$ such that $\sigma|_{\SO(V_1^-)} = \chi \circ \mathrm{N}_{B^-}$ for some quadratic character $\chi = \chi_a$ of $F^\times$ with $a \in F^\times$.
If $\sigma(-1) = -\chi(-1)$ and $\chi \ne 1$, then we have
\[
 \theta_{W_2, V_1^-, \psi}(\sigma) = J_{P_1, \psi}(|\cdot|^{\frac{1}{2}}, \omega_{W_1,\psi_a}^-).
\]
\end{lem}

\begin{proof}
Since $\epsilon(\tfrac{1}{2}, \sigma) = \chi(-1)$ and $\theta_{W_1,V_1^-,\psi}(\sigma)  = \omega_{W_1,\psi_a}^-$, the assertion follows from Lemma \ref{l:ind-princ-1}.
\end{proof}

\subsection{The nonarchimedean case}

Suppose that $F$ is nonarchimedean.
Let
\[
 W_2 = X_1 \oplus X_1^* \oplus W_1, \qquad
 V_1^+ = Y_1 \oplus Y_1^* \oplus V_0^+
\]
be the decompositions as in \S\S \ref{ss:mp} and \ref{ss:so} with
\begin{align*}
 X_1 & = \Span(w_1), &
 Y_1 & = \Span(v_1), \\
 X_1^* & = \Span(w_1^*), &
 Y_1^* & = \Span(v_1^*). 
\end{align*}
Recall that $P_1$ is the maximal parabolic subgroup of $\Sp(W_2)$ stabilizing $X_1$.
Let $B_1$ be the Borel subgroup of $\O(V_1^+)$ stabilizing $Y_1$.

\subsubsection{Proof of Lemma \ref{l:howe-ps-local+}}
\label{sss:howe-ps-local+-nonarch}

Let $\sigma$ be an irreducible representation of $\O(V_1^+)$.
Assume that $\sigma \ne \det$, so that $\theta_{W_2,V_1^+,\psi}(\sigma)$ is nonzero by Lemma \ref{l:o3mp4-nonzero}.
We will compute $\theta_{W_2,V_1^+,\psi}(\sigma)$ under the further assumption that there is an embedding
\[
 \sigma \longhookrightarrow \Ind^{\O(V_1^+)}_{B_1}(\chi_s \boxtimes \sigma_0),
\]
where $\chi_s = \chi |\cdot|^s$ for some unitary character $\chi$ of $\GL(Y_1) \cong F^\times$ and some $s \in \R$, and $\sigma_0$ is a character of $\O(V_0^+) = \{ \pm 1\}$.
Note that if $\chi = 1$ and $s = -\frac{1}{2}$, then since $\sigma \ne \det$, we have $\sigma(-1) = 1$.
Without loss of generality, we may assume that $s \ne \frac{3}{2}$.
Put $\epsilon_0 = \sigma_0(-1) = \sigma(-1) \cdot \chi(-1)$.

We have
\begin{align*}
 \Theta_{W_2,V_1^+,\psi}(\sigma)^* 
 & = \Hom_{\O(V_1^+)}(\omega_{W_2,V_1^+,\psi}, \sigma) \\
 & \subset \Hom_{\O(V_1^+)}(\omega_{W_2,V_1^+,\psi}, \Ind^{\O(V_1^+)}_{B_1}(\chi_s \boxtimes \sigma_0)) \\
 & = \Hom_{\GL(Y_1) \times \O(V_0^+)}(R_{B_1}(\omega_{W_2,V_1^+,\psi}), \chi_s \boxtimes \sigma_0), 
\end{align*}
where $*$ denotes the linear dual and $R_{B_1}$ denotes the normalized Jacquet functor with respect to $B_1$.
By the result of Kudla \cite{kudla}, $R_{B_1}(\omega_{W_2,V_1^+,\psi})$ has a filtration
\[
 R_{B_1}(\omega_{W_2,V_1^+,\psi}) = R^0 \supset R^1 \supset \{ 0 \}
\]
of $\GL(Y_1) \times \O(V_0^+) \times \Mp(W_2)$-modules such that 
\begin{align*}
 R^0/R^1 & \cong |{\det}_{Y_1}|^{\frac{3}{2}} \boxtimes \omega_{W_2,V_0^+,\psi}, \\ 
 R^1 & \cong \Ind^{\GL(Y_1) \times \O(V_0^+) \times \Mp(W_2)}_{\GL(Y_1) \times \O(V_0^+) \times \widetilde{P}_1}(\mathcal{S}(\operatorname{Isom}(X_1,Y_1)) \boxtimes \omega_{W_1,V_0^+,\psi}).
\end{align*}
Here $\GL(Y_1) \times \widetilde{\GL}(X_1)$ acts on $\mathcal{S}(\operatorname{Isom}(X_1, Y_1))$ by
\[
 [(a, \tilde{b}) \cdot f](g) = \chi_\psi(\tilde{b}) \cdot f(a^{-1} \circ g \circ b)
\]
for $a \in \GL(Y_1)$, $\tilde{b} \in \widetilde{\GL}(X_1)$ with projection $b \in \GL(X_1)$, and $f \in \mathcal{S}(\operatorname{Isom}(X_1, Y_1))$.
Since $s \ne \frac{3}{2}$, the action of $\GL(Y_1)$ shows that
\[
 \Hom_{\GL(Y_1) \times \O(V_0^+)}(R^0/R^1, \chi_s \boxtimes \sigma_0) = \{ 0 \},
\]
so that
\begin{align*}
 \Theta_{W_2,V_1^+,\psi}(\sigma)^*
 & \subset \Hom_{\GL(Y_1) \times \O(V_0^+)}(R^1, \chi_s \boxtimes \sigma_0) \\
 & = I_{P_1, \psi}(\chi_s^{-1}, \Theta_{W_1, V_0^+, \psi}(\sigma_0))^*.
\end{align*}
Thus, noting that $\Theta_{W_1, V_0^+, \psi}(\sigma_0) = \omega_{W_1, \psi}^{\epsilon_0}$, we obtain a surjection
\begin{equation}
\label{eq:surj-o3mp4}
 I_{P_1, \psi}(\chi_s^{-1}, \omega_{W_1, \psi}^{\epsilon_0})
 \longtwoheadrightarrow \Theta_{W_2,V_1^+,\psi}(\sigma).
\end{equation}
In particular, when $\chi^2 = 1$ and $s = -\frac{1}{2}$, this proves Lemma \ref{l:howe-ps-local+} in the nonarchimedean case.

\subsubsection{More properties of the theta lift}

We prove more properties of the local theta lift, which are not used in the proof of the main theorem but will be necessary when we describe local $A$-packets explicitly in Appendix \ref{a:A-packets} below.

\begin{lem}
\label{l:more-theta-nonarch-1}
Let $\sigma$ be an irreducible supercuspidal representation of $\O(V_1^\epsilon)$.
Assume that $\sigma \ne \det$ when $\epsilon = -1$.
If $\sigma(-1) = - \epsilon \cdot \epsilon(\frac{1}{2}, \sigma)$, then we have
\[
 \theta_{W_2, V_1^\epsilon, \psi}(\sigma) = \theta_{W_2, V_2^\epsilon, \psi}(\St^\epsilon(1,\sigma_0)),
\]
where $\St^\epsilon(1,\sigma_0)$ is the irreducible square-integrable representation of $\SO(V_2^\epsilon)$ contained in $I_{Q_1}(|\cdot|^{\frac{1}{2}}, \sigma_0)$ with $\sigma_0 = \sigma|_{\SO(V_1^\epsilon)}$ (see Lemmas \ref{l:nsc-so+q1}\eqref{nsc-so+q1i} and \ref{l:nsc-so-q1}\eqref{nsc-so-q1i} below).
\end{lem}

\begin{proof}
Put $\pi = \theta_{W_2, V_1^{\epsilon}, \psi}(\sigma)$.
By Lemma \ref{l:o3mp4-nonzero} and \eqref{eq:dichotomy}, $\pi$ is nonzero and supercuspidal.
Hence, by the local Shimura correspondence \cite{gs}, $\theta_{W_2, V_2^\epsilon, \psi}(\pi)$ is nonzero and square-integrable.

We now consider the normalized Jacquet module of $\omega_{W_2, V_2^\epsilon, \psi}$ with respect to the maximal parabolic subgroup $Q_1$ of $\SO(V_2^\epsilon)$.
By the result of Kudla \cite{kudla}, there is an $F^\times \times \SO(V_1^\epsilon) \times \Mp(W_2)$-equivariant surjection 
\[
 R_{Q_1}(\omega_{W_2, V_2^\epsilon, \psi}) \longtwoheadrightarrow |\cdot|^{\frac{1}{2}} \boxtimes \omega_{W_2, V_1^\epsilon, \psi}.
\]
Since there is a surjection $\omega_{W_2, V_1^\epsilon, \psi} \twoheadrightarrow \sigma_0 \boxtimes \pi$, this gives rise via Frobenius reciprocity to a nonzero $\SO(V_2^\epsilon) \times \Mp(W_2)$-equivariant map
\[
 \omega_{W_2, V_2^\epsilon, \psi} \longrightarrow I_{Q_1}(|\cdot|^{\frac{1}{2}}, \sigma_0) \boxtimes \pi. 
\]
Hence $\theta_{W_2, V_2^\epsilon, \psi}(\pi)$ is a subquotient of $I_{Q_1}(|\cdot|^{\frac{1}{2}}, \sigma_0)$.
Since $\theta_{W_2, V_2^\epsilon, \psi}(\pi)$ is square-integrable, $\theta_{W_2, V_2^\epsilon, \psi}(\pi)$ must be $\St^\epsilon(1, \sigma_0)$. 
This completes the proof.
\end{proof}

\begin{lem}
\label{l:more-theta-nonarch-2}
Let $\sigma$ be an irreducible square-integrable representation of $\O(V_1^+)$ such that $\sigma|_{\SO(V_1^+)}$ is the unique irreducible subrepresentation of $I_{Q_1}(\chi |\cdot|^{\frac{1}{2}})$ for some quadratic character $\chi$ of $F^\times$.
\begin{enumerate}
\item If $\sigma(-1) = -\chi(-1)$ and $\chi \ne 1$, then we have
\[
 \theta_{W_2, V_1^+, \psi}(\sigma) = \widetilde{\St}_\psi(\chi, \omega_{W_1, \psi}^-), 
\]
where $\widetilde{\St}_\psi(\chi, \omega_{W_1, \psi}^-)$ is the irreducible genuine square-integrable representation of $\Mp(W_2)$ contained in $I_{P_1, \psi}(\chi |\cdot|^{\frac{1}{2}}, \omega_{W_1, \psi}^-)$ (see Lemma \ref{l:nsc-mp-p1}\eqref{nsc-mp-p1i} below).
\item 
If $\sigma(-1) = 1$ and $\chi = 1$, then we have
\[
 \theta_{W_2, V_1^+, \psi}(\sigma) = \pi_{\gen, \psi}(\st), 
\]
where $\pi_{\gen, \psi}(\st)$ is the irreducible genuine tempered representation of $\Mp(W_2)$ contained in $I_{P_1, \psi}(|\cdot|^{\frac{1}{2}}, \omega_{W_1, \psi}^+)$ (see Lemma \ref{l:nsc-mp-p1}\eqref{nsc-mp-p1iii} below). 
\end{enumerate}
\end{lem}

\begin{proof}
Note that 
\[
 \epsilon(\tfrac{1}{2}, \sigma) = 
 \begin{cases}
  \chi(-1) & \text{if $\chi \ne 1$;} \\
  -1 & \text{if $\chi = 1$,}
 \end{cases}
\]
so that $\theta_{W_1, V_1^+, \psi}(\sigma)$ is zero by \eqref{eq:dichotomy}.
The assertion follows from \eqref{eq:surj-o3mp4}.
\end{proof}

\begin{lem}
\label{l:more-theta-nonarch-3}
Let $\sigma$ be an irreducible representation of $\O(V_1^+)$ such that $\sigma|_{\SO(V_1^+)} = I_{Q_1}(\chi |\cdot|^s)$ for some unitary character $\chi$ of $F^\times$ and some $s \in \R$ with $0 \le s <\frac{1}{2}$.
If $\sigma(-1) = - \chi(-1)$, then we have
\[
 \theta_{W_2, V_1^+, \psi}(\sigma) = J_{P_1,\psi}(\chi |\cdot|^s, \omega^-_{W_1,\psi}).
\]
\end{lem}

\begin{proof}
The assertion follows from \eqref{eq:surj-o3mp4}.
\end{proof}

\subsection{The real case}

Suppose that $F = \R$.
We work in the category of $(\g,K)$-modules.

\subsubsection{Notation}
\label{sss:notation-real}

Let $\O(n)$ be the maximal compact subgroup of $\GL_n(\R)$ defined by
\[
 \O(n) = \{ g \in \GL_n(\R) \, | \, {}^t g^{-1} = g \}. 
\]
Put $n_0 = [\frac{n}{2}]$.
As in \cite[p.~18]{adams}, we parametrize the irreducible representations of $\O(n)$ by highest weights
\[
 (a_1,\dots,a_{n_0};\epsilon)
\]
with $a_i \in \Z$ such that $a_1 \ge \dots \ge a_{n_0} \ge 0$ and $\epsilon = \pm 1$.
Note that the two representations associated to $(a_1,\dots,a_{n_0};1)$ and $(a_1,\dots,a_{n_0};-1)$ are isomorphic if and only if $n$ is even and $a_{n_0}> 0$.

Let $V_{p,q}$ be the quadratic space over $\R$ of signature $(p,q)$ with $p+q$ odd.
We realize the orthogonal group $\O(p,q) = \O(V_{p,q})$ as 
\[
 \O(p,q) = \left\{ g \in \GL_{p+q}(\R) \, \left| \, 
 {}^t g
 \begin{pmatrix}
  \mathbf{1}_p & \\
  & -\mathbf{1}_q
 \end{pmatrix}
 g =
 \begin{pmatrix}
  \mathbf{1}_p & \\
  & -\mathbf{1}_q
 \end{pmatrix}
 \right. \right\}.
\]
Let $K \cong \O(p) \times \O(q)$ be the maximal compact subgroup of $\O(p,q)$ defined by
\[
 K = \{ g \in \O(p,q) \, | \, {}^t g^{-1} = g \}.
\]
Put $p_0 = [\frac{p}{2}]$ and $q_0 = [\frac{q}{2}]$.
As above, we parametrize the irreducible representations of $K$ by highest weights
\[
 (a_1,\dots,a_{p_0};\epsilon) \boxtimes (b_1,\dots,b_{q_0};\delta) 
\]
with $a_i, b_j \in \Z$ such that $a_1 \ge \dots \ge a_{p_0} \ge 0$, $b_1 \ge \dots \ge b_{q_0} \ge 0$ and $\epsilon, \delta = \pm 1$.

Let $W_n$ be the $2n$-dimensional symplectic space over $\R$.
We realize the symplectic group $\Sp_{2n}(\R) = \Sp(W_n)$ as 
\[
 \Sp_{2n}(\R) = \left\{ g \in \GL_{2n}(\R) \, \left| \, 
 {}^t g 
 \begin{pmatrix}
  & \mathbf{1}_n \\
  -\mathbf{1}_n &
 \end{pmatrix}
 g =
 \begin{pmatrix}
  & \mathbf{1}_n \\
  -\mathbf{1}_n &
 \end{pmatrix}
 \right. \right\}.
\]
Let $\overline{K'} \cong \U(n)$ be the maximal compact subgroup of $\Sp_{2n}(\R)$ defined by 
\[
 \overline{K'} = \{ g \in \Sp_{2n}(\R) \, | \, {}^t g^{-1} = g \}
\]
and $K'$ the preimage of $\overline{K'}$ in $\Mp_{2n}(\R)$.
Following the normalization in \cite[p.~19]{adams}, we parametrize the irreducible genuine representations of $K'$ by highest weights
\[
 (a_1, \dots, a_n)
\]
with $a_i \in \Z + \frac{1}{2}$ such that $a_1 \ge \dots \ge a_n$.
Note that this parametrization depends on $\psi$.

For any $a \in \frac{1}{2} \Z$ with $a>0$, we denote by $D_a$ the relative discrete series representation of $\GL_2(\R)$ of weight $2a + 1$ with central character trivial on $\R_+^\times$, so that $\mathcal{D}_a$ is the $L$-parameter of $D_a$.
For any $a \in \frac{1}{2} \Z \smallsetminus \Z$, we denote by $\widetilde{D}_{a,\psi}$ the genuine discrete series representation of $\Mp_2(\R)$ of weight (relative to the parametrization depending on $\psi$)
\[
\begin{cases}
 a+1 & \text{if $a>0$;} \\
 a-1 & \text{if $a<0$.}
\end{cases}
\]

\subsubsection{Theta correspondence over $\R$}

We recall some basic properties of the theta correspondence over $\R$ (see \cite{howe,kashiwara-vergne,adams}).

Let $\mathcal{P}$ be the Fock model of the Weil representation $\omega_{W_n, V_{p,q}, \psi}$.
Namely, $\mathcal{P} = \bigoplus_{d=0}^\infty \mathcal{P}_d$ is the space of polynomials in $n(p+q)$ variables and $\mathcal{P}_d$ is the subspace of homogeneous polynomials of degree $d$, which is invariant under the action of $K \times K'$.
For any irreducible representation $\mu$ of $K$ occurring in $\mathcal{P}$, we define $\deg \mu$ as the smallest nonnegative integer $d$ such that the $\mu$-isotypic component of $\mathcal{P}_d$ is nonzero.
If 
\[
 \mu = (a_1,\dots,a_k,0,\dots,0;\epsilon) \boxtimes (b_1,\dots,b_l,0,\dots,0;\delta)
\]
with $a_k, b_l > 0$, then we have
\[
 \deg \mu = \sum_{i=1}^k a_i + \sum_{j=1}^l b_j + k' + l', 
\]
where 
\begin{equation}
\label{eq:k'l'}
 k' = 
 \begin{cases}
  0 & \text{if $\epsilon = 1$;} \\
  p-2k & \text{if $\epsilon = -1$,}
 \end{cases}
 \qquad
 l' = 
 \begin{cases}
  0 & \text{if $\delta = 1$;} \\
  q-2l & \text{if $\delta = -1$.}
 \end{cases}
\end{equation}
In particular, $\deg \mu$ does not depend on $n$. 
Similarly, for any irreducible genuine representation $\mu'$ of $K'$ occurring in $\mathcal{P}$, we may define $\deg \mu'$.
If
\[
 \mu' = (a_1,\dots,a_n) + \frac{p-q}{2} \cdot (1,\dots,1) 
\]
with $a_i \in \Z$, then we have
\[
 \deg \mu' = \sum_{i=1}^n |a_i|.
\]
In particular, $\deg \mu'$ depends only on $p-q$.

Let $\mathcal{H}$ be the space of joint harmonics, which is a $K \times K'$-invariant subspace of $\mathcal{P}$.
Then $\mathcal{H}$ is multiplicity-free as a representation of $K \times K'$ and induces a correspondence between representations of $K$ and $K'$ given as follows.
Let $\mu$ be an irreducible representation of $K$ and $\mu'$ an irreducible genuine representation of $K'$.
Then $\mu$ and $\mu'$ correspond if and only if $\mu$ and $\mu'$ are of the form
\[
 \mu = (a_1,\dots,a_k,0,\dots,0;\epsilon) \boxtimes (b_1,\dots,b_l,0,\dots,0;\delta)
\]
and 
\[
 \mu' = (a_1,\dots,a_k, \underbrace{1, \dots, 1}_{k'}, 0,\dots,0, \underbrace{-1,\dots,-1}_{l'}, -b_l,\dots, -b_1) + \frac{p-q}{2} \cdot (1,\dots,1) 
\]
with $a_k, b_l > 0$ and $k+k'+l+l' \le n$, where $k',l'$ are as in \eqref{eq:k'l'}.
Note that $\mu$ and $\mu'$ determine each other.

Let $\sigma$ be an irreducible representation of $\O(p,q)$ such that the theta lift $\pi = \theta_{W_n,V_{p,q}, \psi}(\sigma)$ to $\Mp_{2n}(\R)$ is nonzero.
Let $\mu$ be a $K$-type of $\sigma$, i.e.~an irreducible representation of $K$ occurring in $\sigma|_K$.
We say that $\mu$ is of \emph{minimal degree} in $\sigma$ if $\deg \mu$ is minimal among all $K$-types of $\sigma$.
In this case, $\mu$ occurs in $\mathcal{H}$.
Let $\mu'$ be the irreducible genuine representation of $K'$ corresponding to $\mu$.
Then $\mu'$ is a $K'$-type of $\pi$ and is of minimal degree in $\pi$.
An analogous result also holds when we switch the roles of $\sigma$ and $\pi$.

We also need the notion of \emph{lowest $K$-types} introduced by Vogan \cite{vogan}.
In particular, we will use the following properties:
\begin{itemize}
\item 
any irreducible tempered representation with real infinitesimal character is uniquely determined by its unique lowest $K$-type (see also \S \ref{sss:temp-mp4-real} below);
\item
any lowest $K$-type of a standard module occurs with multiplicity one;
\item 
the set of lowest $K$-types of a standard module agrees with that of its unique irreducible quotient.
\end{itemize}
These two notions of $K$-types are closely related as follows.
Assume for simplicity that $p+q=2n+1$.
Let $\sigma$ be an irreducible representation of $\O(p,q)$ such that the theta lift $\pi = \theta_{W_n,V_{p,q}, \psi}(\sigma)$ to $\Mp_{2n}(\R)$ is nonzero.
Then, by \cite[Corollary 5.2]{ab2}, we have:
\begin{itemize}
\item if $\mu$ is a lowest $K$-type of $\sigma$, then $\mu$ is of minimal degree in $\sigma$;
\item if $\mu'$ is a lowest $K'$-type of $\pi$, then $\mu'$ is of minimal degree in $\pi$.
\end{itemize}
 
\subsubsection{Proof of Lemma \ref{l:ind-princ-1}}
\label{sss:ind-princ-1-real}

Suppose first that $\sigma$ is a principal series representation of $\O(2,1)$.
More generally, let $\sigma$ be an irreducible representation of $\O(2,1)$ such that there is a surjection
\[
 \Ind^{\O(2,1)}_{B_1}(\chi_s \boxtimes \sigma_0) \longtwoheadrightarrow \sigma,
\]
where $B_1$ is the Borel subgroup of $\O(2,1)$, $\chi_s = \chi |\cdot|^s$ for some unitary character $\chi$ of $\R^\times$ and some $s \in \R$ with $s \ge 0$, and $\sigma_0$ is a character of $\O(1)$.
Put
\[
 \delta_0 = \chi(-1), \qquad
 \epsilon_0 = \sigma_0(-1).
\]
Then \eqref{eq:root} is equivalent to
\[
 \epsilon_0 = 1, 
\]
in which case we have $\theta_{W_1,V_{2,1},\psi}(\sigma) = J_{B,\psi}(\chi_s)$.
Let $\mu_0$ be a lowest $(K \cap T_1)$-type of $\chi_s \boxtimes \sigma_0$, where $T_1$ is the Levi component of $B_1$.
Let $\mu$ be a lowest $K$-type of $\sigma$, so that $\mu$ occurs in $\Ind^{\O(2,1)}_{B_1}(\chi_s \boxtimes \sigma_0)$ with multiplicity one.
As we will explicate below, we assume the following conditions:
\begin{itemize}
\item $\mu_0$ is of minimal degree in $\chi_s \boxtimes \sigma_0$;
\item $\mu$ is of minimal degree in $\Ind^{\O(2,1)}_{B_1}(\chi_s \boxtimes \sigma_0)$;
\item $\deg \mu = \deg \mu_0$;
\item the restriction of $\mu$ to $K \cap T_1$ contains $\mu_0$.
\end{itemize}
Then, by the induction principle of Adams--Barbasch (see \cite[Proposition 3.25]{ab1}, \cite[Theorem 8.7]{ab2}), $\theta_{W_2,V_{2,1}, \psi}(\sigma)$ is a subquotient of
\[
 I_{P_1, \psi}(\chi_s^{-1}, \omega_{W_1, \psi}^{\epsilon_0})
\]
containing $\mu'$, where $\mu'$ is the $K'$-type corresponding to $\mu$.

Assume that $(\delta_0, \epsilon_0) \ne (1,-1)$.
We may take $\mu_0$ and $\mu$ given by
\[
 \mu_0 = (\enspace; \delta_0) \boxtimes (\enspace; \epsilon_0), \qquad 
 \mu = (0; \epsilon_0) \boxtimes (\enspace; \delta_0 \epsilon_0), 
\]
which satisfy the conditions above and 
\[
 \deg \mu = \deg \mu_0 = 
 \begin{cases}
  0 & \text{if $(\delta_0, \epsilon_0) = (1,1)$;} \\
  1 & \text{if $(\delta_0, \epsilon_0) = (-1,1)$;} \\
  2 & \text{if $(\delta_0, \epsilon_0) = (-1,-1)$.}
 \end{cases}
\]
Then we have
\[
 \mu' = 
 \begin{cases}
  (\frac{1}{2}, \frac{1}{2}) & \text{if $(\delta_0, \epsilon_0) = (1,1)$;} \\
  (\frac{1}{2}, -\frac{1}{2}) & \text{if $(\delta_0, \epsilon_0) = (-1,1)$;} \\
  (\frac{3}{2}, \frac{3}{2}) & \text{if $(\delta_0, \epsilon_0) = (-1,-1)$.}
 \end{cases}
\]
If $\epsilon_0 = 1$, then it follows from \cite[Proposition 6.10]{ab2} that $\mu'$ is a lowest $K'$-type of the principal series representation
\[
 I_{B,\psi}(\chi_s^{-1}, |\cdot|^{\frac{1}{2}})
\]
of $\Mp_4(\R)$, which has $I_{P_1, \psi}(\chi_s^{-1}, \omega_{W_1, \psi}^+)$ as a quotient.
Hence, if $\epsilon_0 = 1$ and $0 \le s \le \frac{1}{2}$, then we have
\begin{equation}
\label{eq:ind-princ-1-real}
 \theta_{W_2,V_{2,1},\psi}(\sigma) = J_{B,\psi}(|\cdot|^{\frac{1}{2}}, \chi_s).
\end{equation}

Suppose next that $\sigma$ is a discrete series representation of $\O(p,q)$ with $(p,q) = (2,1)$ or $(0,3)$.
We write the $L$-parameter of $\sigma|_{\SO(p,q)}$ as $\mathcal{D}_{\kappa-\frac{1}{2}}$ with some positive integer $\kappa$.
Then, by \eqref{eq:root}, we have
\[
 \sigma(-1) =
 \begin{cases}
  (-1)^\kappa & \text{if $(p,q) = (2,1)$;} \\
  (-1)^{\kappa-1} & \text{if $(p,q) = (0,3)$.}
 \end{cases}
\]
Let $\mu$ be the lowest $K$-type of $\sigma$ given by
\[
 \mu = 
 \begin{cases}
  (\kappa;1) \boxtimes (\enspace;1) & \text{if $(p,q) = (2,1)$;} \\
  (\kappa-1;1) & \text{if $(p,q) = (0,3)$.}
 \end{cases}
\]
Put $\pi_0 = \theta_{W_1, V_{p,q}, \psi}(\sigma)$, so that $\pi_0 = \widetilde{D}_{\lambda,\psi}$ with 
\[
 \lambda = 
 \begin{cases}
  \kappa-\frac{1}{2} & \text{if $(p,q) = (2,1)$;} \\
  -\kappa+\frac{1}{2} & \text{if $(p,q) = (0,3)$.}
 \end{cases}
\]
Since $\mu$ is of minimal degree in $\sigma$, it follows from the induction principle \cite[Theorem 8.4]{ab2} that $\theta_{W_2,V_{p,q},\psi}(\sigma)$ is a subquotient of
\[
 I_{P_1, \psi}(|\cdot|^{-\frac{1}{2}}, \pi_0)
\]
containing $\mu'$, where $\mu'$ is the $K'$-type corresponding to $\mu$ and is given by 
\[
 \mu' = 
 \begin{cases}
  (\kappa+\frac{1}{2}, \frac{1}{2}) & \text{if $(p,q) = (2,1)$;} \\
  (-\frac{3}{2}, -\kappa-\frac{1}{2}) & \text{if $(p,q) = (0,3)$.}
 \end{cases}
\]
By \cite[Proposition 6.10]{ab2}, $\mu'$ is a lowest $K'$-type of $I_{P_1, \psi}(|\cdot|^{-\frac{1}{2}}, \pi_0)$, so that
\[
 \theta_{W_2,V_{p,q},\psi}(\sigma) = J_{P_1, \psi}(|\cdot|^{\frac{1}{2}}, \pi_0).
\]
This completes the proof of Lemma \ref{l:ind-princ-1} in the real case.

\subsubsection{Proof of Lemma \ref{l:howe-ps-local+}}
\label{sss:howe-ps-local+-real}

We retain the notation of \S \ref{sss:ind-princ-1-real}, so that $\sigma$ is a quotient of $\Ind^{\O(2,1)}_{B_1}(\chi_s \boxtimes \sigma_0)$.
Then \eqref{item:howe-ps-1} follows from \eqref{eq:ind-princ-1-real}.
To prove \eqref{item:howe-ps-2}, we may assume that $(\delta_0, \epsilon_0) = (-1, -1)$.
Since $\omega_{W_1,\psi}^- = \widetilde{D}_{\frac{1}{2},\psi}$, it follows from \cite[Proposition 6.10]{ab2} that $\mu'$ is a lowest $K'$-type of
\[
 I_{P_1, \psi}(\chi_s^{-1}, \omega_{W_1, \psi}^-).
\]
Hence, if $(\delta_0, \epsilon_0) = (-1, -1)$ and $s \ge 0$, then we have
\[
 \theta_{W_2, V_{2,1}, \psi}(\sigma) = J_{P_1, \psi}(\chi_s, \omega_{W_1,\psi}^-).
\]
This completes the proof of Lemma \ref{l:howe-ps-local+} in the real case.

\subsubsection{More properties of the theta lift}

We prove more properties of the local theta lift, which are not used in the proof of the main theorem but will be necessary when we describe local $A$-packets explicitly in Appendix \ref{a:A-packets} below.

\begin{lem}
\label{l:more-theta-real-1}
Let $\sigma$ be a discrete series representation of $\O(p,q)$ with $(p,q) = (2,1)$ or $(0,3)$ and with lowest $K$-type $\mu$ given by 
\[
 \mu = 
\begin{cases}
 (\kappa; 1) \boxtimes (\enspace;-1) & \text{if $(p,q) = (2,1)$;} \\
 (\kappa-1; -1) & \text{if $(p,q) = (0,3)$}
\end{cases} 
\]
for some positive integer $\kappa$.
Assume that $\sigma \ne \det$ (i.e.~$\kappa > 1$) when $(p,q) = (0,3)$.
\begin{enumerate}
\item If $(p,q) = (2,1)$, then $\theta_{W_2, V_{2,1}, \psi}(\sigma)$ is the genuine (limit of) discrete series representation of $\Mp_4(\R)$ with lowest $K'$-type $(\kappa+\frac{1}{2}, -\frac{1}{2})$ (relative to $\psi$).
\item If $(p,q) = (0,3)$, then $\theta_{W_2, V_{0,3}, \psi}(\sigma)$ is the genuine discrete series representation of $\Mp_4(\R)$ with lowest $K'$-type $(-\frac{5}{2}, -\kappa-\frac{1}{2})$ (relative to $\psi$).
\end{enumerate}
\end{lem}

\begin{proof}
Note that
\[
 \sigma(-1) =
 \begin{cases}
  (-1)^{\kappa-1} & \text{if $(p,q) = (2,1)$;} \\
  (-1)^\kappa & \text{if $(p,q) = (0,3)$}
 \end{cases}
\]
and $\epsilon(\frac{1}{2}, \sigma) = (-1)^\kappa$, so that $\theta_{W_1, V_{p,q}, \psi}(\sigma)$ is zero by \eqref{eq:dichotomy}.
Put $\pi = \theta_{W_2, V_{p,q}, \psi}(\sigma)$, which is nonzero by Lemma \ref{l:o3mp4-nonzero}.
Let $\mu'$ be the irreducible genuine representation of $K'$ corresponding to $\mu$:
\[
 \mu' = 
 \begin{cases}
  (\kappa+\frac{1}{2}, -\frac{1}{2}) & \text{if $(p,q) = (2,1)$;} \\
  (-\frac{5}{2}, -\kappa-\frac{1}{2}) & \text{if $(p,q) = (0,3)$.}
 \end{cases}
\]
Since $\mu$ is the unique $K$-type of minimal degree in $\sigma$, $\mu'$ is the unique $K'$-type of minimal degree in $\pi$.
Hence we deduce that $\mu'$ is the unique lowest $K'$-type of $\pi$, noting that $\theta_{W_2, V_{3,2}, \psi}(\pi)$ is nonzero.

Since $\pi$ has real infinitesimal character by \cite{przebinda}, it remains to show that $\pi$ is tempered (see also \S \ref{sss:temp-mp4-real} below).
For any irreducible genuine nontempered representation $\pi'$ of $\Mp_4(\R)$, it follows from \cite[Proposition 6.10]{ab2} that $\pi'$ has lowest $K'$-types given as follows.
\begin{itemize}
\item
Suppose that $\pi' = J_{P_1, \psi}(\chi |\cdot|^s, \widetilde{D}_{a,\psi})$ for some unitary character $\chi$ of $\R^\times$, some $s \in \R$ with $s>0$, and some $a \in \frac{1}{2} \Z \smallsetminus \Z$.
Then $\pi'$ has a unique lowest $K'$-type
\[
\begin{cases}
 (a+1,\frac{1}{2}) & \text{if $a>0$ and $\chi(-1) = 1$;} \\
 (a+1,\frac{3}{2}) & \text{if $a>0$ and $\chi(-1) = -1$;} \\
 (-\frac{3}{2},a-1) & \text{if $a<0$ and $\chi(-1) = 1$;} \\
 (-\frac{1}{2},a-1) & \text{if $a<0$ and $\chi(-1) = -1$.}
\end{cases} 
\]
\item
Suppose that $\pi' = J_{P_2, \psi}(D_a \otimes |\det|^s)$ for some $a \in \frac{1}{2} \Z$ with $a>0$ and some $s \in \C$ with $\Re s > 0$.
Then $\pi'$ has a unique lowest $K'$-type
\[
 (a+\tfrac{1}{2}, -a-\tfrac{1}{2})
\]
if $a \in \Z$ and two lowest $K'$-types
\[
 (a+1,-a), \qquad
 (a,-a-1)
\]
if $a \notin \Z$.
\item 
Suppose that $\pi' = J_{B, \psi}(\chi_1 |\cdot|^{s_1}, \chi_2 |\cdot|^{s_2})$ for some unitary characters $\chi_1, \chi_2$ of $\R^\times$ and some $s_1, s_2 \in \R$ with $s_1 > 0$ and $s_1 \ge s_2 \ge 0$.
Put $\epsilon_1 = \chi_1(-1)$ and $\epsilon_2 = \chi_2(-1)$.
Then $\pi'$ has a unique lowest $K'$-type
\[
\begin{cases}
 (\frac{1}{2}, \frac{1}{2}) & \text{if $(\epsilon_1, \epsilon_2) = (1,1)$;} \\
 (\frac{1}{2}, -\frac{1}{2}) & \text{if $(\epsilon_1, \epsilon_2) = (1,-1), (-1,1)$;} \\
 (-\frac{1}{2}, -\frac{1}{2}) & \text{if $(\epsilon_1, \epsilon_2) = (-1,-1)$.}
\end{cases} 
\]
\end{itemize}
In particular, the set of lowest $K'$-types of $\pi'$ does not agree with that of $\pi$, which is a singleton $\{ \mu' \}$.
Hence $\pi$ is tempered.
This completes the proof.
\end{proof}

\begin{lem}
\label{l:more-theta-real-2}
Let $\sigma$ be an irreducible representation of $\O(2,1)$ such that $\sigma|_{\SO(2,1)} = I_{Q_1}(\chi |\cdot|^s)$ for some unitary character $\chi$ of $\R^\times$ and some $s \in \R$ with $0 \le s <\frac{1}{2}$.
If $\sigma(-1) = -\chi(-1)$, then we have
\[
 \theta_{W_2, V_{2,1}, \psi}(\sigma) = J_{P_1,\psi}(\chi |\cdot|^s, \omega^-_{W_1,\psi}).
\]
\end{lem}

\begin{proof}
By assumption, we have $\sigma = \Ind^{\O(2,1)}_{B_1}(\chi_s \boxtimes \sigma_0)$, where $B_1$ is the Borel subgroup of $\O(2,1)$, $\chi_s = \chi |\cdot|^s$, and $\sigma_0$ is the nontrivial character of $\O(1)$.
Since $\theta_{W_1, V_{1,0}, \psi}(\sigma_0) = \omega_{W_1,\psi}^-$, it follows from the induction principle \cite[Theorem 8.4]{ab2} that $\theta_{W_2, V_{2,1}, \psi}(\sigma)$ is a subquotient of 
\[
 I_{P_1, \psi}(\chi_s^{-1}, \omega_{W_1, \psi}^-).
\]
On the other hand, by \cite[Proposition 2.3]{gi-real}, $I_{P_1, \psi}(\chi_s^{-1}, \omega_{W_1, \psi}^-)$ is irreducible. 
This implies the assertion. 
\end{proof}

\section{Local theta lifts from $\Mp_4$ to $\SO_{2r+5}$ with $r>3$}
\label{a:mp4so2r+5}

In this appendix, we prove some properties of the local theta lift used in \S\S \ref{s:soudry} and \ref{s:pf-soudry}.
Let $F$ be a local field of characteristic zero.

\subsection{Properties of the theta lift}

We consider the theta lift from $\Mp(W_2)$ to $\SO(V_{r+2}^+)$ with $r>3$.

\begin{lem}
\label{l:ind-princ-2}
Let $\rho$ be a $2$-dimensional orthogonal tempered representation of $L_F$.
Put $\phi = \rho \boxtimes S_2$ and $\theta(\phi) = \phi \oplus (1 \boxtimes S_{2r})$.
Let $\pi_\phi$ be the unique irreducible genuine representation of $\Mp(W_2)$ with $L$-parameter $\varphi_\phi$ (relative to $\psi$) and $\sigma_{\theta(\phi)}$ the unique irreducible representation of $\SO(V_{r+2}^+)$ with $L$-parameter $\varphi_{\theta(\phi)}$.
Then we have
\[
 \theta_{W_2,V_{r+2}^+,\psi}(\pi_\phi) = \sigma_{\theta(\phi)}.
\]
\end{lem}

\begin{proof}
When $F$ is nonarchimedean or $F = \R$, the assertion will be proved in \S\S \ref{sss:ind-princ-2-nonarch} and \ref{sss:ind-princ-2-real} below.
When $F = \C$, the assertion follows from \cite[Theorem 2.8]{ab1}.
\end{proof}

We also need the properties \eqref{eq:howe-ps-nonarch} when $F$ is nonarchimedean and \eqref{eq:soudry-real} when $F = \R$, which will be recalled and proved below.

\subsection{The nonarchimedean case}

Suppose that $F$ is nonarchimedean.
Let
\[
 W_2 = X_1 \oplus X_1^* \oplus W_1 = X_2 \oplus X_2^*, \qquad
 V_{r+2}^+ = Y_1 \oplus Y_1^* \oplus V_{r+1}^+ = Y_2 \oplus Y_2^* \oplus V_r^+
\]
be the decompositions as in \S\S \ref{ss:mp} and \ref{ss:so} with
\begin{align*}
 X_1 & = \Span(w_1), &
 Y_1 & = \Span(v_1), \\
 X_1^* & = \Span(w_1^*), &
 Y_1^* & = \Span(v_1^*), \\
 X_2 & = \Span(w_1, w_2), &
 Y_2 & = \Span(v_1, v_2), \\
 X_2^* & = \Span(w_1^*, w_2^*), &
 Y_2^* & = \Span(v_1^*, v_2^*). 
\end{align*}
We write 
\[
 X_2 = X_1 \oplus X_1'
\]
with $X_1' = \Span(w_2)$.
Recall that $P_i$ and $Q_i$ are the maximal parabolic subgroups of $\Sp(W_2)$ and $\SO(V_{r+2}^+)$ stabilizing $X_i$ and $Y_i$, respectively. 
Let $\mathcal{B}$ be the Borel subgroup of $\GL(X_2)$ stabilizing $X_1$.

\subsubsection{Proof of Lemma \ref{l:ind-princ-2}}
\label{sss:ind-princ-2-nonarch}

Let $\tau$ be the irreducible self-dual tempered representation of $\GL(X_2) \cong \GL_2(F)$ with $L$-parameter $\rho$.
Then we have either 
\begin{itemize}
\item $\tau = \Ind^{\GL(X_2)}_{\mathcal{B}}(\chi \boxtimes \chi^{-1})$ for some unitary character $\chi$ of $F^\times$ such that $\chi^2 \ne 1$; or 
\item $\tau = \Ind^{\GL(X_2)}_{\mathcal{B}}(\chi_a \boxtimes \chi_b)$ for some $a, b \in F^\times$; or
\item $\tau$ is supercuspidal.
\end{itemize}
Let $\pi$ be an irreducible genuine representation of $\Mp(W_2)$.
We will compute $\theta_{W_2,V_{r+2}^+,\psi}(\pi)$ under the assumption that there is an embedding
\[
 \pi \longhookrightarrow I_{P_2, \psi}(\tau_s), 
\]
where $\tau_s = \tau \otimes |\det|^s$ for some $s \in \R$ with $s \le 0$.

We have
\begin{align*}
 \Theta_{W_2,V_{r+2}^+,\psi}(\pi)^* 
 & = \Hom_{\Mp(W_2)}(\omega_{W_2,V_{r+2}^+,\psi}, \pi) \\
 & \subset \Hom_{\Mp(W_2)}(\omega_{W_2,V_{r+2}^+,\psi}, I_{P_2,\psi}(\tau_s)) \\
 & = \Hom_{\widetilde{\GL}(X_2)}(R_{P_2}(\omega_{W_2,V_{r+2}^+,\psi}), \tau_s \otimes \chi_\psi),
\end{align*}
where $*$ denotes the linear dual and $R_{P_2}$ denotes the normalized Jacquet functor with respect to $P_2$.
By the result of Kudla \cite{kudla}, $R_{P_2}(\omega_{W_2,V_{r+2}^+,\psi})$ has a filtration
\[
 R_{P_2}(\omega_{W_2,V_{r+2}^+,\psi}) = R^0 \supset R^1 \supset R^2 \supset \{ 0 \} 
\]
of $\widetilde{\GL}(X_2) \times \SO(V_{r+2}^+)$-modules such that
\begin{align*}
 R^0/R^1 & \cong \chi_\psi |{\det}_{X_2}|^{r+1}, \\
 R^1/R^2 & \cong \Ind^{\widetilde{\GL}(X_2) \times \SO(V_{r+2}^+)}_{\widetilde{\mathcal{B}} \times Q_1}(\chi_\psi |{\det}_{X_1}|^{r+\frac{1}{2}} \boxtimes \mathcal{S}(\operatorname{Isom}(Y_1,X_1'))), \\
 R^2 & \cong \Ind^{\widetilde{\GL}(X_2) \times \SO(V_{r+2}^+)}_{\widetilde{\GL}(X_2) \times Q_2}(\mathcal{S}(\operatorname{Isom}(Y_2,X_2))).
\end{align*}
Here $\widetilde{\GL}(X_1') \times \GL(Y_1)$ acts on $\mathcal{S}(\operatorname{Isom}(Y_1, X_1'))$ by
\[
 [(\tilde{a}, b) \cdot f](g) = \chi_\psi(\tilde{a}) \cdot f(a^{-1} \circ g \circ b)
\]
for $\tilde{a} \in \widetilde{\GL}(X_1')$ with projection $a \in \GL(X_1')$, $b \in \GL(Y_1)$, and $f \in \mathcal{S}(\operatorname{Isom}(Y_1, X_1'))$, and $\widetilde{\GL}(X_2) \times \GL(Y_2)$ acts on $\mathcal{S}(\operatorname{Isom}(Y_2, X_2))$ similarly.
Since $s \le 0$, the actions of $\widetilde{\GL}(X_2)$ and $\widetilde{\GL}(X_1)$ show that
\[
 \Hom_{\widetilde{\GL}(X_2)}(R^0/R^1, \tau_s \otimes \chi_\psi) = \{ 0 \}
\]
and
\begin{align*}
 & \Hom_{\widetilde{\GL}(X_2)}(R^1/R^2, \tau_s \otimes \chi_\psi) \\ 
 & = \Hom_{\widetilde{\GL}(X_1) \times \widetilde{\GL}(X_1')}(\Ind^{\widetilde{\GL}(X_1) \times \widetilde{\GL}(X_1') \times \SO(V_{r+2}^+)}_{\widetilde{\GL}(X_1) \times \widetilde{\GL}(X_1') \times Q_1}(\chi_\psi |{\det}_{X_1}|^{r+\frac{1}{2}} \boxtimes \mathcal{S}(\operatorname{Isom}(Y_1,X_1'))), R_{\overline{\mathcal{B}}}(\tau_s \otimes \chi_\psi)) \\ 
 & = \{ 0 \},
\end{align*}
respectively, where $\overline{\mathcal{B}}$ is the Borel subgroup of $\GL(X_2)$ opposite to $\mathcal{B}$, so that
\begin{align*}
 \Theta_{W_2,V_{r+2}^+,\psi}(\pi)^*
 & \subset \Hom_{\widetilde{\GL}(X_2)}(R^2, \tau_s \otimes \chi_\psi) \\
 & = I_{Q_2}(\tau_s^\vee, 1)^*.
\end{align*}
Thus, we obtain a surjection
\[
 I_{Q_2}(\tau_s^\vee, 1)
 \longtwoheadrightarrow \Theta_{W_2,V_{r+2}^+,\psi}(\pi).
\]
When $s = -\frac{1}{2}$, the left-hand side is a quotient of
\begin{equation}
\label{l:ind-princ-2-pf}
 I_{Q_{(2,r)}}(\tau \otimes |\det|^{\frac{1}{2}}, |\det|^{\frac{r}{2}}),
\end{equation}
where $Q_{(i,j)}$ is the parabolic subgroup of $\SO(V_{r+2}^+)$ with Levi component $\GL_i \times \GL_j \times \SO(V_{r-i-j+2}^+)$.
By \cite[Theorem 4.2]{zelevinsky}, the representation
\[
 \Ind^{\GL_{r+2}(F)}_{\mathcal{Q}_{(2,r)}}((\tau \otimes |\det|^{\frac{1}{2}}) \boxtimes |\det|^{\frac{r}{2}})
\]
is irreducible and is isomorphic to
\[
 \Ind^{\GL_{r+2}(F)}_{\mathcal{Q}_{(r,2)}}(|\det|^{\frac{r}{2}} \boxtimes (\tau \otimes |\det|^{\frac{1}{2}})),
\]
where $\mathcal{Q}_{(i,j)}$ is the maximal parabolic subgroup of $\GL_{i+j}$ with Levi component $\GL_i \times \GL_j$.
Hence \eqref{l:ind-princ-2-pf} is isomorphic to 
\[
 I_{Q_{(r,2)}}(|\det|^{\frac{r}{2}}, \tau \otimes |\det|^{\frac{1}{2}}), 
\]
which is a quotient of 
\[
 I_Q(|\cdot|^{r-\frac{1}{2}}, |\cdot|^{r-\frac{3}{2}}, \dots, |\cdot|^{\frac{1}{2}}, \tau \otimes |\det|^{\frac{1}{2}}),
\]
where $Q = Q_{(1^r,2)}$.
This proves Lemma \ref{l:ind-princ-2} in the nonarchimedean case.

\subsubsection{Proof of \eqref{eq:howe-ps-nonarch}}
\label{sss:howe-ps-nonarch}

Let $a, b \in F^\times$ be such that $\chi_a, \chi_b, 1$ are pairwise distinct.
Then \eqref{eq:howe-ps-nonarch} says that
\[
 \theta_{W_2,V_{r+2}^+,\psi}(J_{P_1, \psi}(\chi_a |\cdot|^{\frac{1}{2}}, \omega_{W_1,\psi_b}^-)) = J_Q(|\cdot|^{r-\frac{1}{2}}, |\cdot|^{r-\frac{3}{2}}, \dots, |\cdot|^{\frac{3}{2}}, \chi_a |\cdot|^{\frac{1}{2}}, \sigma^-_{\varphi_b}),
\]
where $Q = Q_{(1^r)}$ and $\sigma^-_{\varphi_b}$ is the irreducible square-integrable representation of $\SO(V_2^+)$ with $L$-parameter $\varphi_b = (\chi_b \boxtimes S_2) \oplus (1 \boxtimes S_2)$ associated to the nontrivial character of $\overline{S}_{\varphi_b} \cong \Z / 2 \Z$.

Put $\pi_0 = \omega_{W_1,\psi_b}^-$ to ease notation.
Let $\pi$ be an irreducible genuine representation of $\Mp(W_2)$.
We will compute $\theta_{W_2,V_{r+2}^+,\psi}(\pi)$ under the assumption that there is an embedding
\[
 \pi \longhookrightarrow I_{P_1, \psi}(\chi_s, \pi_0), 
\]
where $\chi_s = \chi |\cdot|^s$ for some unitary character $\chi$ of $\GL(X_1) \cong F^\times$ and some $s \in \R$ with $s \le 0$.

We have
\begin{align*}
 \Theta_{W_2,V_{r+2}^+,\psi}(\pi)^* 
 & = \Hom_{\Mp(W_2)}(\omega_{W_2,V_{r+2}^+,\psi}, \pi) \\
 & \subset \Hom_{\Mp(W_2)}(\omega_{W_2,V_{r+2}^+,\psi}, I_{P_1, \psi}(\chi_s, \pi_0)) \\
 & = \Hom_{\widetilde{\GL}(X_1) \times \Mp(W_1)}(R_{P_1}(\omega_{W_2,V_{r+2}^+,\psi}), \chi_s \chi_\psi \boxtimes \pi_0), 
\end{align*}
where $*$ denotes the linear dual and $R_{P_1}$ denotes the normalized Jacquet functor with respect to $P_1$.
By the result of Kudla \cite{kudla}, $R_{P_1}(\omega_{W_2,V_{r+2}^+,\psi})$ has a filtration
\[
 R_{P_1}(\omega_{W_2,V_{r+2}^+,\psi}) = R^0 \supset R^1 \supset \{ 0 \} 
\]
of $\widetilde{\GL}(X_1) \times \Mp(W_1) \times \SO(V_{r+2}^+)$-modules such that 
\begin{align*}
 R^0/R^1 & \cong \chi_\psi |{\det}_{X_1}|^{r+\frac{1}{2}} \boxtimes \omega_{W_1,V_{r+2}^+,\psi}, \\
 R^1 & \cong \Ind^{\widetilde{\GL}(X_1) \times \Mp(W_1) \times \SO(V_{r+2}^+)}_{\widetilde{\GL}(X_1) \times \Mp(W_1) \times Q_1}(\mathcal{S}(\operatorname{Isom}(Y_1,X_1)) \boxtimes \omega_{W_1,V_{r+1}^+,\psi}).
\end{align*}
Here $\widetilde{\GL}(X_1) \times \GL(Y_1)$ acts on $\mathcal{S}(\operatorname{Isom}(Y_1, X_1))$ by
\[
 [(\tilde{a}, b) \cdot f](g) = \chi_\psi(\tilde{a}) \cdot f(a^{-1} \circ g \circ b)
\]
for $\tilde{a} \in \widetilde{\GL}(X_1)$ with projection $a \in \GL(X_1)$, $b \in \GL(Y_1)$, and $f \in \mathcal{S}(\operatorname{Isom}(Y_1, X_1))$.
Since $s \le 0$, the action of $\widetilde{\GL}(X_1)$ shows that
\[
 \Hom_{\widetilde{\GL}(X_1) \times \Mp(W_1)}(R^0/R^1, \chi_s \chi_\psi \boxtimes \pi_0) = \{ 0 \}, 
\]
so that 
\begin{align*}
 \Theta_{W_2,V_{r+2}^+,\psi}(\pi)^*
 & \subset \Hom_{\widetilde{\GL}(X_1) \times \Mp(W_1)}(R^1, \chi_s \chi_\psi \boxtimes \pi_0) \\
 & = I_{Q_1}(\chi_s^{-1}, \Theta_{W_1,V_{r+1}^+,\psi}(\pi_0))^*.
\end{align*}
Thus, we obtain a surjection
\begin{equation}
\label{eq:howe-ps-nonarch-pf}
 I_{Q_1}(\chi_s^{-1}, \Theta_{W_1,V_{r+1}^+,\psi}(\pi_0))
 \longtwoheadrightarrow \Theta_{W_2,V_{r+2}^+,\psi}(\pi).
\end{equation}
On the other hand, since $\pi_0$ is supercuspidal, $\Theta_{W_1,V_{r+1}^+,\psi}(\pi_0)$ is irreducible.
Hence, by \cite[Theorem 1.4]{atobe-gan}, we have
\[
 \Theta_{W_1,V_{r+1}^+,\psi}(\pi_0) = J_{Q'}(|\cdot|^{r-\frac{1}{2}}, |\cdot|^{r-\frac{3}{2}}, \dots, |\cdot|^{\frac{3}{2}}, \sigma^-_{\varphi_b}),
\]
where $Q' = Q_{(1^{r-1})}$.
In particular, when $\chi = \chi_a$ and $s = -\frac{1}{2}$, the left-hand side of \eqref{eq:howe-ps-nonarch-pf} is a quotient of
\[
 I_Q(\chi_a |\cdot|^{\frac{1}{2}}, |\cdot|^{r-\frac{1}{2}}, |\cdot|^{r-\frac{3}{2}}, \dots, |\cdot|^{\frac{3}{2}}, \sigma^-_{\varphi_b})
 \cong
 I_Q(|\cdot|^{r-\frac{1}{2}}, |\cdot|^{r-\frac{3}{2}}, \dots, |\cdot|^{\frac{3}{2}}, \chi_a |\cdot|^{\frac{1}{2}}, \sigma^-_{\varphi_b}).
\]
This proves \eqref{eq:howe-ps-nonarch}.

\subsection{The real case}

Suppose that $F = \R$.
We work in the category of $(\g,K)$-modules.
We retain the notation of \S \ref{sss:notation-real}.

\subsubsection{Proof of Lemma \ref{l:ind-princ-2}}
\label{sss:ind-princ-2-real}

Put $(p,q) = (r+3,r+2)$.
Let $\tau$ be the irreducible self-dual tempered representation of $\GL_2(\R)$ with $L$-parameter $\rho$.
Then we have either 
\begin{itemize}
\item $\tau = \Ind^{\GL_2(\R)}_{\mathcal{B}}(\chi_1 \boxtimes \chi_2)$ for some unitary characters $\chi_1, \chi_2$ of $\R^\times$ such that $\chi_1 \chi_2 = 1$ or $\{\chi_1, \chi_2\} = \{ 1, \mathrm{sgn} \}$, where $\mathcal{B}$ is the Borel subgroup of $\GL_2$; or 
\item $\tau = D_\kappa$ for some positive integer $\kappa$.
\end{itemize}
Put 
\[
 \epsilon_1 = \chi_1(-1), \qquad
 \epsilon_2 = \chi_2(-1)
\]
in the former case.
Let $\pi$ be an irreducible genuine representation of $\Mp_4(\R)$ such that there is a surjection
\[
 I_{P_2, \psi}(\tau_s) \longtwoheadrightarrow \pi,
\]
where $\tau_s = \tau \otimes |\det|^s$ for some $s \in \R$ with $s \ge 0$.
Let $\mu'_0$ be a lowest $(\overline{K'} \cap M_2)$-type of $\tau_s$, where $M_2$ is the Levi component of $P_2$.
Let $\mu'$ be a lowest $K'$-type of $\pi$, so that $\mu'$ occurs in $I_{P_2, \psi}(\tau_s)$ with multiplicity one.
As we will explicate below, we assume the following conditions:
\begin{itemize}
\item $\mu'_0$ is of minimal degree in $\tau_s$;
\item $\mu'$ is of minimal degree in $I_{P_2, \psi}(\tau_s)$;
\item $\deg \mu' = \deg \mu'_0$;
\item the restriction of $\mu'$ to $K' \cap \widetilde{M}_2$ contains $\mu'_0 \otimes \chi_\psi$.
\end{itemize}
Then, by the induction principle of Adams--Barbasch (see \cite[Proposition 3.25]{ab1}, \cite[Theorem 8.7]{ab2}), $\theta_{W_2,V_{p,q}, \psi}(\pi)$ is a subquotient of
\[
 I_{Q_2}(\tau_s^\vee, 1)
\]
containing $\mu$, where $\mu$ is the $K$-type corresponding to $\mu'$.

By \cite[Proposition 6.10]{ab2}, we may take $\mu'_0$ and $\mu'$ given by
\[
 \mu'_0 = 
\begin{cases}
 (0;1) & \text{if $\tau = \Ind^{\GL_2(\R)}_{\mathcal{B}}(\chi_1 \boxtimes \chi_2)$ and $(\epsilon_1, \epsilon_2) = (1,1)$;} \\
 (1;1) & \text{if $\tau = \Ind^{\GL_2(\R)}_{\mathcal{B}}(\chi_1 \boxtimes \chi_2)$ and $(\epsilon_1, \epsilon_2) = (1,-1), (-1,1)$;} \\
 (0;-1) & \text{if $\tau = \Ind^{\GL_2(\R)}_{\mathcal{B}}(\chi_1 \boxtimes \chi_2)$ and $(\epsilon_1, \epsilon_2) = (-1,-1)$;} \\
 (2\kappa+1;1) & \text{if $\tau = D_\kappa$}
\end{cases}
\]
and 
\[
 \mu' = 
\begin{cases}
 (\frac{1}{2}, \frac{1}{2}) & \text{if $\tau = \Ind^{\GL_2(\R)}_{\mathcal{B}}(\chi_1 \boxtimes \chi_2)$ and $(\epsilon_1, \epsilon_2) = (1,1)$;} \\
 (\frac{1}{2}, -\frac{1}{2}) & \text{if $\tau = \Ind^{\GL_2(\R)}_{\mathcal{B}}(\chi_1 \boxtimes \chi_2)$ and $(\epsilon_1, \epsilon_2) = (1,-1), (-1,1)$;} \\
 (-\frac{1}{2}, -\frac{1}{2}) & \text{if $\tau = \Ind^{\GL_2(\R)}_{\mathcal{B}}(\chi_1 \boxtimes \chi_2)$ and $(\epsilon_1, \epsilon_2) = (-1,-1)$;} \\
 (\kappa+\frac{1}{2}, -\kappa-\frac{1}{2}) & \text{if $\tau = D_\kappa$,}
\end{cases}
\]
which satisfy the conditions above (since $\theta_{W_2,V_{3,2}, \psi}(\pi)$ is nonzero) and 
\[
 \deg \mu' = \deg \mu'_0 = 
\begin{cases}
 0 & \text{if $\tau = \Ind^{\GL_2(\R)}_{\mathcal{B}}(\chi_1 \boxtimes \chi_2)$ and $(\epsilon_1, \epsilon_2) = (1,1)$;} \\
 1 & \text{if $\tau = \Ind^{\GL_2(\R)}_{\mathcal{B}}(\chi_1 \boxtimes \chi_2)$ and $(\epsilon_1, \epsilon_2) = (1,-1), (-1,1)$;} \\
 2 & \text{if $\tau = \Ind^{\GL_2(\R)}_{\mathcal{B}}(\chi_1 \boxtimes \chi_2)$ and $(\epsilon_1, \epsilon_2) = (-1,-1)$;} \\
 2\kappa+1 & \text{if $\tau = D_\kappa$.}
\end{cases}
\]
Then we have
\[
 \mu = 
\begin{cases}
 (0,\dots,0;1) \boxtimes (0,\dots,0;1) & \text{if $\tau = \Ind^{\GL_2(\R)}_{\mathcal{B}}(\chi_1 \boxtimes \chi_2)$ and $(\epsilon_1, \epsilon_2) = (1,1)$;} \\
 (0,\dots,0;1) \boxtimes (1,0,\dots,0;1) & \text{if $\tau = \Ind^{\GL_2(\R)}_{\mathcal{B}}(\chi_1 \boxtimes \chi_2)$ and $(\epsilon_1, \epsilon_2) = (1,-1), (-1,1)$;} \\
 (0,\dots,0;1) \boxtimes (1,1,0,\dots,0;1) & \text{if $\tau = \Ind^{\GL_2(\R)}_{\mathcal{B}}(\chi_1 \boxtimes \chi_2)$ and $(\epsilon_1, \epsilon_2) = (-1,-1)$;} \\
 (\kappa,0,\dots,0;1) \boxtimes (\kappa+1,0,\dots,0;1) & \text{if $\tau = D_\kappa$.}
\end{cases}
\]
If we put $Q' = Q_{(2,1^r)}$, then it follows from \cite[Proposition 6.10]{ab2} that $\mu$ is a lowest $K$-type of
\[
 I_{Q'}(\tau_s^\vee, |\cdot|^{r-\frac{1}{2}}, |\cdot|^{r-\frac{3}{2}}, \dots, |\cdot|^{\frac{1}{2}}), 
\]
which has $I_{Q_2}(\tau_s^\vee, 1)$ as a quotient.
Hence, if $0 \le s \le \frac{1}{2}$, then we have
\[
 \theta_{W_2,V_{p,q},\psi}(\pi) = J_Q(|\cdot|^{r-\frac{1}{2}}, |\cdot|^{r-\frac{3}{2}}, \dots, |\cdot|^{\frac{1}{2}}, \tau_s), 
\]
where $Q = Q_{(1^r,2)}$.
This completes the proof of Lemma \ref{l:ind-princ-2} in the real case.

\subsubsection{Some $A$-packets}

Let $G$ be the split odd special orthogonal group of rank $r+2$, so that $G = \SO(p,q)$ with $(p,q) = (r+3,r+2)$.
Let $\theta$ be the Cartan involution of $G$ defined by $\theta(g) = {}^t g^{-1}$.
Let $\g$ be the complexified Lie algebra of $G$.

We consider local $A$-parameters
\[
 \phi = \mathcal{D}_{\kappa} \boxtimes S_2, \qquad
 \theta(\phi) = \phi \oplus (1 \boxtimes S_{2r})
\]
with a positive integer $\kappa$.
For $\epsilon \in \mu_2$, let $\sigma^\epsilon$ be the representation in the $A$-packet $\Pi_{\theta(\phi)}(G)$ associated to $\epsilon$.
Then, by \cite{mr1,mr2}, we have
\[
 \sigma^+ = A_{\q_1}(\lambda_1), \qquad
 \sigma^- = A_{\q_0}(\lambda_0) \oplus A_{\q_2}(\lambda_2),
\]
where $\q_i$ is a $\theta$-stable parabolic subalgebra of $\g$ whose normalizer $L_i$ in $G$ satisfies
\[
 L_i \cong \U(i,2-i) \times \SO(p-2i, q-4+2i),
\]
$\lambda_i$ is the $1$-dimensional representation of $L_i$ given by 
\[
 \lambda_i = {\det}^{\kappa-r-1} \boxtimes 1,
\]
and $A_{\q_i}(\lambda_i)$ is the cohomologically induced representation defined by \cite[(5.6)]{knappvogan}.
Note that $\lambda_i$ is in the weakly fair range, so that $A_{\q_i}(\lambda_i)$ is a (possibly zero, possibly reducible) unitary representation of $G$ of finite length.
If $\kappa > r$, then $\lambda_i$ is in the good range, so that $A_{\q_i}(\lambda_i)$ is nonzero and irreducible.
Moreover, by \cite[Theorem 11.216]{knappvogan}, we have
\[
 A_{\q_1}(\lambda_1) \cong J_Q(|\cdot|^{r-\frac{1}{2}}, |\cdot|^{r-\frac{3}{2}}, \dots, |\cdot|^{\frac{1}{2}}, D_\kappa \otimes |\det|^{\frac{1}{2}})
\]
for $\kappa > r$, where $Q = Q_{(1^r,2)}$.

\begin{lem}
\label{l:atlas}
Assume that $r=4$.
Then $A_{\q_i}(\lambda_i)$ is nonzero and irreducible.
Moreover, we have
\[
 A_{\q_1}(\lambda_1) \cong J_Q(|\cdot|^{\frac{7}{2}}, |\cdot|^{\frac{5}{2}}, |\cdot|^{\frac{3}{2}}, |\cdot|^{\frac{1}{2}}, D_\kappa \otimes |\det|^{\frac{1}{2}}),
\]
where $Q = Q_{(1^4,2)}$.
\end{lem}

\begin{proof}
It remains to prove the assertion for $\kappa \le 4$, but this can be checked by the Atlas software \cite{atlas}.
For example, if $i=1$ and $\kappa=2$, then we have:
\begin{verbatim}
atlas> set G=SO(7,6)
Variable G: RealForm
atlas> Aq_reducible(KGB(G,14),[2,0,0,2,0,0],[1,0,0,1,0,0])
Value: 
1*parameter(x=1953,lambda=[11,7,9,1,3,1]/2,nu=[7,1,5,-1,3,1]/2) [44]
atlas> set P=Parabolic:([4],KGB(G,2336))
Variable P: ([int],KGBElt)
atlas> set p=parameter(KGB(Levi(P),0),[0,0,0,0,5,-3]/2,[7,5,3,1,1,1]/2)
Variable p: Param
atlas> finalize(real_induce_standard(p,G))
Value: 
1*parameter(x=1953,lambda=[11,7,9,1,3,1]/2,nu=[7,1,5,-1,3,1]/2) [44]
\end{verbatim}
\end{proof}

\subsubsection{Proof of \eqref{eq:soudry-real}}
\label{sss:soudry-real}

For $\epsilon \in \mu_2$, let $\pi^\epsilon$ be the representation in the $A$-packet $\Pi_{\phi,\psi}(\Mp_4)$ associated to $\epsilon$.
By definition, we have
\[
 \pi^+ = J_{P_2,\psi}(D_\kappa \otimes |\det|^{\frac{1}{2}}), \qquad
 \pi^- = \pi_\Lambda \oplus \pi_\Lambda^\vee,
\]
where $\pi_\Lambda$ is the genuine discrete series representation of $\Mp_4(\R)$ with lowest $K'$-type $\Lambda = (\kappa+\frac{3}{2}, \kappa+\frac{3}{2})$ (relative to the parametrization depending on $\psi$).
Note that $\pi^-$ does not depend on $\psi$.
Then \eqref{eq:soudry-real} says that
\[
 \theta_{W_2, V_{p,q}, \psi}(\pi^+) = A_{\q_1}(\lambda_1), \qquad
 \theta_{W_2, V_{p,q}, \psi}(\pi^-) = A_{\q_0}(\lambda_0) \oplus A_{\q_2}(\lambda_2)
\]
when $(p,q) = (7,6)$.

The first assertion follows from Lemmas \ref{l:ind-princ-2} and \ref{l:atlas}.
To prove the second assertion, we consider the theta lift $\theta_{W_2, V_{p,q},\psi}(\pi_\Lambda)$ when $4 \le p \le q+5$.
By \cite[Theorem 3.3]{ab2}, we have
\[
 \pi_{\Lambda} = \theta_{W_2, V_{5,0}, \psi}(\sigma_\Lambda), 
\]
where $\sigma_\Lambda$ is the irreducible representation of $\O(5,0)$ with highest weight $(\kappa-1,\kappa-1;1)$.
Hence it follows from \cite[Theorem 1.4]{lmt} combined with induction in stages (see \cite[Corollary 11.86]{knappvogan}) that $\theta_{W_2, V_{p,q}, \psi}(\pi_\Lambda)$ is a subquotient of $A_\q(\lambda)$, where $\q$ is a $\theta$-stable parabolic subalgebra of $\mathfrak{so}(p,q)$ whose normalizer $L$ in $\SO(p,q)$ satisfies
\[
 L \cong \U(2, 0) \times \SO(p-4,q) 
\]
and $\lambda$ is the $1$-dimensional representation of $L$ given by
\[
 {\det}^{\kappa-\frac{p+q-3}{2}} \boxtimes 1.
\]
In particular, when $\{p,q\} = \{6,7\}$, we have 
\[
 \theta_{W_2, V_{p,q},\psi}(\pi_\Lambda) = A_\q(\lambda)
\]
by Lemma \ref{l:atlas}.
On the other hand, by \cite[Lemma 1.5]{ab2}, we have 
\[
 \theta_{W_2, V_{p,q}, \psi}(\pi_\Lambda^\vee) = \theta_{W_2, V_{q,p}, \psi}(\pi_\Lambda).
\]
This completes the proof of \eqref{eq:soudry-real}.

\section{Local $A$-packets for $\Mp_4$}
\label{a:A-packets}

In this appendix, we describe local $A$-packets for $\Mp_4$ explicitly.
Let $F$ be a local field of characteristic zero.

\subsection{The nonarchimedean case}

Suppose that $F$ is nonarchimedean.
Let $\st$ be the Steinberg representation of $\GL_2(F)$ and put $\st_\chi = \st \otimes (\chi \circ \det)$ for any quadratic character $\chi$ of $F^\times$.
We may regard $\st_\chi$ as a representation of $\SO(V_1^+)$.
Let $\widetilde{\st}_{\chi, \psi}$ be the irreducible genuine square-integrable representation of $\Mp(W_1)$ contained in $I_{B,\psi}(\chi |\cdot|^{\frac{1}{2}})$.

\subsubsection{Elliptic tempered representations of $\Mp(W_2)$}

We first introduce some notation.
For this, we need to enumerate the elliptic tempered representations of $\Mp(W_2)$ which are not supercuspidal.

The following lemmas summarize the result of Hanzer--Mati\'c \cite{hanzer-matic} on the composition series of parabolically induced representations of $\Mp(W_2)$.
 
\begin{lem}
\label{l:nsc-mp-p1}
Let $\chi$ be a unitary character of $F^\times$.
Let $s \in \R$ with $s \ge 0$.
Let $\pi$ be either an irreducible genuine square-integrable representation of $\Mp(W_1)$ or an even elementary Weil representation of $\Mp(W_1)$.
Then $I_{P_1, \psi}(\chi |\cdot|^s, \pi)$ is reducible if and only if one of the following holds:
\begin{enumerate}
\item
\label{nsc-mp-p1i}
$\chi^2 = 1$ (so that $\chi = \chi_a$ for some $a \in F^\times$), $s = \frac{1}{2}$, and $\pi$ is supercuspidal but $\pi \ne \omega_{W_1,\psi_a}^-$, in which case we have
\[
 0 \longrightarrow \widetilde{\St}_\psi(\chi, \pi) \longrightarrow I_{P_1, \psi}(\chi |\cdot|^{\frac{1}{2}}, \pi) \longrightarrow J_{P_1, \psi}(\chi |\cdot|^{\frac{1}{2}}, \pi) \longrightarrow 0,
\]
where $\widetilde{\St}_\psi(\chi, \pi)$ is an irreducible square-integrable representation;
\item
$\chi^2 = 1$, $s = \frac{1}{2}$, and $\pi = \widetilde{\st}_{\mu, \psi}$ for some quadratic character $\mu$ of $F^\times$, in which case we have
\[ 
 0 \longrightarrow \Pi \longrightarrow I_{P_1, \psi}(\chi |\cdot|^{\frac{1}{2}}, \widetilde{\st}_{\mu, \psi}) \longrightarrow J_{P_1, \psi}(\chi |\cdot|^{\frac{1}{2}}, \widetilde{\st}_{\mu, \psi}) \longrightarrow  0
\]
with 
\[
 \Pi = 
\begin{cases}
  \widetilde{\St}_\psi(\chi, \widetilde{\st}_{\mu, \psi}) & \text{if $\chi \ne \mu$;} \\
  \pi_{\gen, \psi}(\st_\chi) & \text{if $\chi = \mu \ne 1$;} \\
  \pi_{\mathrm{ng}, \psi}(\st_\chi) & \text{if $\chi = \mu = 1$,}
\end{cases}
\]
where $\widetilde{\St}_\psi(\chi, \widetilde{\st}_{\mu, \psi})$ is an irreducible square-integrable representation (which is isomorphic to $\widetilde{\St}_\psi(\mu, \widetilde{\st}_{\chi, \psi})$), and $\pi_{\gen, \psi}(\st_\chi)$ and $\pi_{\mathrm{ng}, \psi}(\st_\chi)$ are the representations as in Lemma \ref{l:nsc-mp-p2}\eqref{nsc-mp-p2i} below;
\item 
\label{nsc-mp-p1iii}
$\chi^2 = 1$, $s = \frac{1}{2}$, and $\pi = \omega_{W_1,\psi_b}^+$ for some $b \in F^\times$, in which case we have
\[
  0 \longrightarrow \Pi \longrightarrow I_{P_1, \psi}(\chi |\cdot|^{\frac{1}{2}}, \omega_{W_1,\psi_b}^+) \longrightarrow J_{B, \psi}(\chi |\cdot|^{\frac{1}{2}}, \chi_b |\cdot|^{\frac{1}{2}}) \longrightarrow 0
\]
with 
\[
 \Pi = 
\begin{cases}
 J_{P_1, \psi}(\chi_b |\cdot|^{\frac{1}{2}}, \widetilde{\st}_{\chi,\psi}) & \text{if $\chi \ne \chi_b$;} \\
 \pi_{\mathrm{ng}, \psi}(\st_\chi) & \text{if $\chi = \chi_b \ne 1$;} \\
 \pi_{\gen, \psi}(\st_\chi) & \text{if $\chi = \chi_b = 1$;}
\end{cases}
\]
\item 
\label{nsc-mp-p1iv}
$\chi^2 = 1$, $s = \frac{3}{2}$, and $\pi = \widetilde{\st}_{\chi, \psi}$, in which case we have
\[ 
 0 \longrightarrow \widetilde{\St}^+_{\chi, \psi} \longrightarrow I_{P_1, \psi}(\chi |\cdot|^{\frac{3}{2}}, \widetilde{\st}_{\chi, \psi}) \longrightarrow J_{P_1, \psi}(\chi |\cdot|^{\frac{3}{2}}, \widetilde{\st}_{\chi, \psi}) \longrightarrow 0,
\]
where $\widetilde{\St}^+_{\chi, \psi}$ is an irreducible square-integrable representation;
\item
$\chi^2 = 1$ (so that $\chi = \chi_a$ for some $a \in F^\times$), $s = \frac{3}{2}$, and $\pi = \omega_{W_1,\psi_a}^-$, in which case we have
\[
 0 \longrightarrow \widetilde{\St}^-_{\chi, \psi} \longrightarrow I_{P_1, \psi}(\chi |\cdot|^{\frac{3}{2}}, \omega_{W_1,\psi_a}^-) \longrightarrow \omega_{W_2,\psi_a}^- \longrightarrow 0, 
\]
where $\widetilde{\St}^-_{\chi, \psi}$ is an irreducible square-integrable representation;
\item
$\chi^2 = 1$ (so that $\chi = \chi_a$ for some $a \in F^\times$), $s = \frac{3}{2}$, and $\pi = \omega_{W_1,\psi_a}^+$, in which case we have
\[
 0 \longrightarrow J_{P_2, \psi}(\st_\chi \otimes |\det|) \longrightarrow I_{P_1, \psi}(\chi |\cdot|^{\frac{3}{2}}, \omega_{W_1, \psi_a}^+) \longrightarrow \omega_{W_2, \psi_a}^+ \longrightarrow 0.
\]
\end{enumerate}
\end{lem}

\begin{lem}
\label{l:nsc-mp-p2}
Let $\tau$ be an irreducible unitary square-integrable representation of $\GL_2(F)$ with central character $\omega_\tau$.
Let $s \in \R$ with $s \ge 0$.
Then $I_{P_2, \psi}(\tau \otimes |\det|^s)$ is reducible if and only if one of the following holds:
\begin{enumerate}
\item 
\label{nsc-mp-p2i}
$\omega_\tau = 1$ and $s=0$, in which case we have
\[
 I_{P_2, \psi}(\tau) = \pi_{\gen, \psi}(\tau) \oplus \pi_{\mathrm{ng}, \psi}(\tau),
\]
where $\pi_{\gen, \psi}(\tau)$ is an irreducible $\psi$-generic tempered representation and $\pi_{\mathrm{ng}, \psi}(\tau)$ is an irreducible non-$\psi$-generic tempered representation;
\item 
$\tau$ is self-dual and supercuspidal, $\omega_\tau \ne 1$, and $s = \frac{1}{2}$, in which case we have
\[
 0 \longrightarrow \widetilde{\St}_\psi(\tau) \longrightarrow I_{P_2, \psi}(\tau \otimes |\det|^{\frac{1}{2}}) \longrightarrow J_{P_2, \psi}(\tau \otimes |\det|^{\frac{1}{2}}) \longrightarrow 0, 
\]
where $\widetilde{\St}_\psi(\tau)$ is an irreducible square-integrable representation;
\item 
$\tau = \st_\chi$ for some quadratic character $\chi$ of $F^\times$ and $s=1$, in which case we have
\[
 0 \longrightarrow \widetilde{\St}^+_{\chi,\psi} \longrightarrow I_{P_2, \psi}(\st_\chi \otimes |\det|) \longrightarrow J_{P_2, \psi}(\st_\chi \otimes |\det|) \longrightarrow 0, 
\]
where $\widetilde{\St}^+_{\chi,\psi}$ is the representation as in Lemma \ref{l:nsc-mp-p1}\eqref{nsc-mp-p1iv}.
\end{enumerate}
\end{lem}

The representations described in the lemmas above exhaust all irreducible genuine elliptic tempered representations of $\Mp(W_2)$ which are not supercuspidal.

\subsubsection{Elliptic tempered representations of $\SO(V_2^\epsilon)$}

We also need to enumerate the elliptic tempered representations of $\SO(V_2^\epsilon)$ which are not supercuspidal.
We write $I_{Q_i}({\cdots}) = I^\epsilon_{Q_i}({\cdots})$ and $J_{Q_i}({\cdots}) = J^\epsilon_{Q_i}({\cdots})$ to indicate that they are representations of $\SO(V_2^\epsilon)$.
Let $\St^\epsilon$ be the Steinberg representation of $\SO(V_2^\epsilon)$ and put $\St^\epsilon_\chi = \St^\epsilon \otimes (\chi \circ \nu)$ for any quadratic character $\chi$ of $F^\times$, where $\nu$ denotes the spinor norm.

The following lemmas summarize the result of Sally--Tadi\'c \cite{sally-tadic} (see also \cite{matic}) on the composition series of parabolically induced representations of $\SO(V_2^+)$.

\begin{lem}
\label{l:nsc-so+q1}
Let $\chi$ be a unitary character of $F^\times$.
Let $s \in \R$ with $s \ge 0$.
Let $\sigma$ be an irreducible square-integrable representation of $\SO(V_1^+)$.
Then $I^+_{Q_1}(\chi |\cdot|^s, \sigma)$ is reducible if and only if one of the following holds:
\begin{enumerate}
\item
\label{nsc-so+q1i}
$\chi^2 = 1$, $s = \frac{1}{2}$, and $\sigma$ is supercuspidal, in which case we have
\[
 0 \longrightarrow \St^+(\chi, \sigma) \longrightarrow I^+_{Q_1}(\chi |\cdot|^{\frac{1}{2}}, \sigma) \longrightarrow J^+_{Q_1}(\chi |\cdot|^{\frac{1}{2}}, \sigma) \longrightarrow 0, 
\]
where $\St^+(\chi, \sigma)$ is an irreducible square-integrable representation;
\item 
$\chi^2 = 1$, $s = \frac{1}{2}$, and $\sigma = \st_\mu$ for some quadratic character $\mu$ of $F^\times$, in which case we have
\[
 0 \longrightarrow \Sigma \longrightarrow I^+_{Q_1}(\chi |\cdot|^{\frac{1}{2}}, \st_\mu) \longrightarrow J^+_{Q_1}(\chi |\cdot|^{\frac{1}{2}}, \st_\mu) \longrightarrow 0
\]
with
\[
 \Sigma = 
\begin{cases}
 \St^+(\chi, \st_\mu) & \text{if $\chi \ne \mu$;} \\
 \sigma_\gen(\st_\chi) & \text{if $\chi = \mu$,}
\end{cases}
\]
where $\St^+(\chi, \st_\mu)$ is an irreducible square-integrable representation (which is isomorphic to $\St^+(\mu, \st_\chi)$) and $\sigma_\gen(\st_\chi)$ is the representation as in Lemma \ref{l:nsc-so+q2}\eqref{nsc-so+q2i} below;
\item
$\chi^2 = 1$, $s = \frac{3}{2}$, and $\sigma = \st_\chi$, in which case we have
\[
 0 \longrightarrow \St^+_\chi \longrightarrow I^+_{Q_1}(\chi |\cdot|^{\frac{3}{2}}, \st_\chi) \longrightarrow J^+_{Q_1}(\chi |\cdot|^{\frac{3}{2}}, \st_\chi) \longrightarrow 0.
\]
\end{enumerate}
\end{lem}

\begin{lem}
\label{l:nsc-so+q2}
Let $\tau$ be an irreducible unitary square-integrable representation of $\GL_2(F)$ with central character $\omega_\tau$.
Let $s \in \R$ with $s \ge 0$.
Then $I^+_{Q_2}(\tau \otimes |\det|^s)$ is reducible if and only if one of the following holds:
\begin{enumerate}
\item 
\label{nsc-so+q2i}
$\omega_\tau = 1$ and $s=0$, in which case we have
\[
 I^+_{Q_2}(\tau) = \sigma_\gen(\tau) \oplus \sigma_{\mathrm{ng}}(\tau),
\]
where $\sigma_\gen(\tau)$ is an irreducible generic tempered representation and $\sigma_{\mathrm{ng}}(\tau)$ is an irreducible nongeneric tempered representation;
\item 
$\tau$ is self-dual and supercuspidal, $\omega_\tau \ne 1$, and $s=\frac{1}{2}$, in which case we have
\[
 0 \longrightarrow \St^+(\tau) \longrightarrow I^+_{Q_2}(\tau \otimes |\det|^{\frac{1}{2}}) \longrightarrow J^+_{Q_2}(\tau \otimes |\det|^{\frac{1}{2}}) \longrightarrow 0, 
\]
where $\St^+(\tau)$ is an irreducible square-integrable representation;
\item 
$\tau = \st_\chi$ for some quadratic character $\chi$ of $F^\times$ and $s=1$, in which case we have
\[
 0 \longrightarrow \St^+_\chi \longrightarrow I^+_{Q_2}(\st_\chi \otimes |\det|) \longrightarrow J^+_{Q_2}(\st_\chi \otimes |\det|) \longrightarrow 0.
\]
\end{enumerate}
\end{lem}

We can easily describe the composition series of parabolically induced representations of $\SO(V_2^-)$ as follows.

\begin{lem}
\label{l:nsc-so-q1}
Let $\chi$ be a unitary character of $F^\times$.
Let $s \in \R$ with $s \ge 0$.
Let $\sigma$ be an irreducible representation of $\SO(V_1^-)$.
Then $I^-_{Q_1}(\chi |\cdot|^s, \sigma)$ is reducible if and only if one of the following holds:
\begin{enumerate}
\item 
\label{nsc-so-q1i}
$\chi^2 = 1$, $s=\frac{1}{2}$, and $\sigma \ne \chi \circ \nu$, in which case we have
\[
 0 \longrightarrow \St^-(\chi, \sigma) \longrightarrow I^-_{Q_1}(\chi |\cdot|^{\frac{1}{2}}, \sigma) \longrightarrow J^-_{Q_1}(\chi|\cdot|^{\frac{1}{2}}, \sigma) \longrightarrow 0, 
\]
where $\St^-(\chi, \sigma)$ is an irreducible square-integrable representation;
\item 
$\chi^2 = 1$, $s=\frac{3}{2}$, and $\sigma = \chi \circ \nu$, in which case we have
\[
 0 \longrightarrow \St^-_\chi \longrightarrow I^-_{Q_1}(\chi |\cdot|^{\frac{3}{2}}, \chi \circ \nu) \longrightarrow J^-_{Q_1}(\chi |\cdot|^{\frac{3}{2}}, \chi \circ \nu) \longrightarrow 0.
\]
\end{enumerate}
\end{lem}

The representations described in the lemmas above exhaust all irreducible elliptic tempered representations of $\SO(V_2^\epsilon)$ which are not supercuspidal.

\subsubsection{Local Shimura correspondence}

The local Shimura correspondence \cite{gs} is a bijection between representations of $\Mp(W_2)$ and $\SO(V_2^\epsilon)$ defined via theta lifts, which restricts to a bijection
\[
 \Pi_{\varphi, \psi}(\Mp(W_2)) \, \longleftrightarrow \,
 \Pi_\varphi(\SO(V_2^+)) \sqcup \Pi_\varphi(\SO(V_2^-))
\]
for any $L$-parameter $\varphi: W_F \times \SL_2(\C) \rightarrow \Sp_4(\C)$.
The following table describes this bijection explicitly when any irreducible summand of $\varphi$ is symplectic.

\begin{center}
{\renewcommand{\arraystretch}{1.5}
\begin{tabular}{lll} \hline
 $\varphi:$ $L$-parameter &
 $\Pi_{\varphi, \psi}(\Mp(W_2))$ & 
 $\Pi_{\varphi}(\SO(V_2^\epsilon))$ \\ \hline \hline
 $\varphi = \varrho \boxtimes S_1$ &
 $\pi^{\epsilon}:$ supercuspidal & 
 $\sigma^{\epsilon}:$ supercuspidal \\ \hline
 $\varphi = \rho \boxtimes S_2$ &
 $\pi^+ = \widetilde{\St}_\psi(\tau)$ &
 $\sigma^+ = \St^+(\tau)$ \\
 & $\pi^-:$ supercuspidal & 
 $\sigma^-:$ supercuspidal \\ \hline
 $\varphi = \chi_a \boxtimes S_4$ &
 $\pi^+ = \widetilde{\St}^+_{\chi_a, \psi}$ &
 $\sigma^+ = \St^+_{\chi_a}$ \\ 
 & $\pi^- = \widetilde{\St}^-_{\chi_a, \psi}$ &
 $\sigma^- = \St^-_{\chi_a}$ \\ \hline
 $\varphi = 1 \boxtimes S_4$ &
 $\pi^+ = \widetilde{\St}^-_{1,\psi}$ &
 $\sigma^+ = \St_1^+$ \\ 
 & $\pi^- = \widetilde{\St}^+_{1,\psi}$ &
 $\sigma^- = \St_1^-$ \\ \hline
 $\varphi = (\rho_1 \boxtimes S_1) \oplus (\rho_2 \boxtimes S_1)$
 & $\pi^{\epsilon_1, \epsilon_2}:$ supercuspidal
 & $\sigma^{\epsilon_1, \epsilon_2}:$ supercuspidal \\ \hline
 $\varphi = (\rho_0 \boxtimes S_1) \oplus (\chi_a \boxtimes S_2)$
 & $\pi^{+,+} = \widetilde{\St}_\psi(\chi_a, \pi_0^+)$
 & $\sigma^{+,+} = \St^+(\chi_a,\sigma_0^+)$ \\
 & $\pi^{+,-} = \theta_{W_2, V_1^{\varepsilon_a}, \psi_a}((\sigma_0^{\varepsilon_a} \otimes \nu_a)^{-\varepsilon_0 \chi_a(-1)})$
 & $\sigma^{+,-} = \theta_{W_1, V_2^-, \psi_a}(\pi^{\varepsilon_a}_0) \otimes \nu_a$ \\
 & $\pi^{-,+} = \widetilde{\St}_\psi(\chi_a, \pi_0^-)$
 & $\sigma^{-,+} = \St^-(\chi_a,\sigma_0^-)$ \\
 & $\pi^{-,-} = \theta_{W_2, V_1^{-\varepsilon_a}, \psi_a}((\sigma_0^{-\varepsilon_a} \otimes \nu_a)^{\varepsilon_0 \chi_a(-1)})$
 & $\sigma^{-,-} = \theta_{W_1, V_2^+, \psi_a}(\pi^{-\varepsilon_a}_0) \otimes \nu_a$ \\ \hline
 $\varphi = (\rho_0 \boxtimes S_1) \oplus (1 \boxtimes S_2)$
 & $\pi^{+,+} = \theta_{W_2, V_1^+, \psi}(\sigma_0^{+, -\varepsilon_0})$
 & $\sigma^{+,+} = \St^+(1,\sigma_0^+)$ \\
 & $\pi^{+,-} = \widetilde{\St}_\psi(1, \pi_0^+)$
 & $\sigma^{+,-} = \theta_{W_1, V_2^-, \psi}(\pi_0^+)$ \\
 & $\pi^{-,+} = \theta_{W_2, V_1^-, \psi}(\sigma_0^{-, \varepsilon_0})$
 & $\sigma^{-,+} = \St^-(1,\sigma_0^-)$ \\
 & $\pi^{-,-} = \widetilde{\St}_\psi(1, \pi_0^-)$
 & $\sigma^{-,-} = \theta_{W_1, V_2^+, \psi}(\pi_0^-)$ \\ \hline
 $\varphi = (\chi_a \boxtimes S_2) \oplus (\chi_b \boxtimes S_2)$
 & $\pi^{+,+} = \widetilde{\St}_\psi(\chi_a, \widetilde{\st}_{\chi_b,\psi})$
 & $\sigma^{+,+} = \St^+(\chi_a, \st_{\chi_b})$ \\
 & $\pi^{+,-} = \widetilde{\St}_\psi(\chi_a, \omega_{W_1,\psi_b}^-)$
 & $\sigma^{+,-} = \St^-(\chi_a, \nu_b)$ \\
 & $\pi^{-,+} = \widetilde{\St}_\psi(\chi_b, \omega_{W_1,\psi_a}^-)$
 & $\sigma^{-,+} = \St^-(\chi_b, \nu_a)$ \\
 & $\pi^{-,-} = \theta_{W_2, V_1^-, \psi_b}(\nu_{ab}^{\chi_{ab}(-1)})$
 & $\sigma^{-,-} = \theta_{W_1, V_2^+, \psi_b}(\omega_{W_1, \psi_a}^-) \otimes \nu_b$ \\ \hline
 $\varphi = (\chi_a \boxtimes S_2) \oplus (1 \boxtimes S_2)$
 & $\pi^{+,+} = \widetilde{\St}_\psi(\chi_a, \omega_{W_1,\psi}^-)$
 & $\sigma^{+,+} = \St^+(\chi_a, \st_1)$ \\
 & $\pi^{+,-} = \widetilde{\St}_\psi(\chi_a, \widetilde{\st}_{1,\psi})$
 & $\sigma^{+,-} = \St^-(\chi_a, \nu_1)$ \\
 & $\pi^{-,+} = \theta_{W_2, V_1^-, \psi}(\nu_a^{\chi_a(-1)})$
 & $\sigma^{-,+} = \St^-(1, \nu_a)$ \\
 & $\pi^{-,-} = \widetilde{\St}_\psi(1, \omega_{W_1,\psi_a}^-)$
 & $\sigma^{-,-} = \theta_{W_1, V_2^+, \psi}(\omega_{W_1, \psi_a}^-)$ \\ \hline
 $\varphi = \varphi_0 \oplus \varphi_0$
 & $\pi^{+,+} = \pi_{\gen, \psi}(\tau_0)$
 & $\sigma^{+,+} = \sigma_\gen(\tau_0)$ \\
 & $\pi^{-,-} = \pi_{\mathrm{ng}, \psi}(\tau_0)$
 & $\sigma^{-,-} = \sigma_{\mathrm{ng}}(\tau_0)$ \\ \hline
\end{tabular}
}
\end{center}

\begin{itemize}[leftmargin=8mm]
\item $a, b \in F^\times$ such that $\chi_a, \chi_b, 1$ are pairwise distinct
\item $\varrho:$ $4$-dimensional irreducible symplectic representation of $W_F$
\item $\rho_0, \rho_1, \rho_2:$ $2$-dimensional irreducible symplectic representation of $W_F$ such that $\rho_1 \ne \rho_2$
\item $\rho:$ $2$-dimensional irreducible orthogonal representation of $W_F$
\item $\varphi_0:$ $2$-dimensional irreducible symplectic representation of $L_F$
\item $\Pi_{\rho_0,\psi}(\Mp(W_1)) = \{ \pi_0^\epsilon \, | \, \epsilon \in \mu_2 \}$
\item $\Pi_{\rho_0}(\SO(V_1^\epsilon)) = \{ \sigma_0^\epsilon \}$
\item $\Pi_\rho(\GL_2(F)) = \{ \tau \}$
\item $\Pi_{\varphi_0}(\GL_2(F)) = \{ \tau_0 \}$
\item $\nu_a = \chi_a \circ \nu$ with $\nu:$ spinor norm
\item $\nu_a^\epsilon:$ $\epsilon$-extension of $\nu_a$
\item $\varepsilon_0 = \epsilon(\frac{1}{2}, \rho_0)$
\item $\varepsilon_a = \epsilon(\frac{1}{2}, \rho_0) \cdot \epsilon(\frac{1}{2}, \rho_0 \times \chi_a) \cdot \chi_a(-1)$
\end{itemize}

\subsubsection{Nontempered $A$-packets}

For any $A$-parameter $\phi : L_F \times \SL_2(\C) \rightarrow \Sp_4(\C)$, we have described most of the representations in the $A$-packet $\Pi_{\phi,\psi}(\Mp(W_2))$ explicitly in the body of this paper.
The following lemmas determine the remaining representations.

\begin{lem}
Suppose that $\phi = (\rho \boxtimes S_1) \oplus (\chi_a \boxtimes S_2)$ with a $2$-dimensional symplectic almost tempered representation $\rho$ of $L_F$ and $a \in F^\times$.
\begin{enumerate}
\item
If either 
\begin{itemize}
\item $\rho = \rho_0 \boxtimes S_1$ for some $2$-dimensional irreducible symplectic representation $\rho_0$ of $W_F$; or 
\item $\rho = \chi_b \boxtimes S_2$ for some $b \in F^\times$ such that $\chi_a \ne \chi_b$,
\end{itemize}
then for any $\epsilon_1 \in \mu_2$, we have 
\[
 \pi^{\epsilon_1, -} = \pi^{\epsilon_1, -}_\varphi,
\]
where $\pi^{\epsilon_1,-}_\varphi$ is the irreducible genuine square-integrable representation of $\Mp(W_2)$ with $L$-parameter $\varphi = \rho \oplus (\chi_a \boxtimes S_2)$ (relative to $\psi$) associated to $(\epsilon_1,-1) \in \widehat{S}_\varphi = \mu_2 \times \mu_2$.
\item
If $\rho = \chi_a \boxtimes S_2$, then for any $\epsilon_1 \in \mu_2$, we have
\[
 \pi^{\epsilon_1, -} =
\begin{cases}
 0 & \text{if $\chi_a \ne 1$ and $\epsilon_1 = 1$;} \\
 \pi_\varphi^{-,-} & \text{if $\chi_a \ne 1$ and $\epsilon_1 = -1$;} \\
 \pi_\varphi^{+,+} & \text{if $\chi_a = 1$ and $\epsilon_1 = 1$;} \\
 0 & \text{if $\chi_a = 1$ and $\epsilon_1 = -1$,}
\end{cases}
\]
where $\pi^{\epsilon_1,\epsilon_1}_\varphi$ is the irreducible genuine tempered  representation of $\Mp(W_2)$ with $L$-parameter $\varphi = (\chi_a \boxtimes S_2) \oplus (\chi_a \boxtimes S_2)$ (relative to $\psi$) associated to $(\epsilon_1,\epsilon_1) \in \widehat{S}_\varphi = \Delta \mu_2$.
\item
If $\rho = (\chi |\cdot|^s \boxtimes S_1) \oplus (\chi^{-1} |\cdot|^{-s} \boxtimes S_1)$ for some unitary character $\chi$ of $F^\times$ and some $s \in \R$ with $0 \le s < \frac{1}{2}$, then we have
\[
 \pi^{+,-} = J_{P_1,\psi}(\chi |\cdot|^s, \omega^-_{W_1,\psi_a}).
\]
\end{enumerate}
\end{lem}

\begin{proof}
For $\epsilon_1 \in \mu_2$, put $\epsilon = \epsilon_1 \cdot \epsilon(\tfrac{1}{2},\rho) \cdot \epsilon(\tfrac{1}{2},\rho \times \chi_a) \cdot \chi_a(-1)$ and $\epsilon' = - \epsilon_1 \cdot \epsilon(\tfrac{1}{2},\rho) \cdot \chi_a(-1)$.
Let $\sigma_0^\epsilon$ be the irreducible representation of $(B^\epsilon)^\times$ with $L$-parameter $\rho \otimes \chi_a$ and $\sigma^{\epsilon,\epsilon'}$ the $\epsilon'$-extension of $\sigma_0^\epsilon$ to $\O(V_1^\epsilon)$.
Since $\pi^{\epsilon_1, -} = \theta_{W_2,V_1^{\epsilon}, \psi_a}(\sigma^{\epsilon, \epsilon'})$, the lemma follows from Lemmas \ref{l:o3mp4-nonzero}, \ref{l:more-theta-nonarch-1}, \ref{l:more-theta-nonarch-2}, and \ref{l:more-theta-nonarch-3}.
\end{proof}

\begin{lem}
Suppose that $\phi = (\chi_a \boxtimes S_2) \oplus (\chi_b \boxtimes S_2)$ with $a,b \in F^\times$.
\begin{enumerate}
\item
If $\chi_a \ne \chi_b$, then we have
\[
 \pi^{-,-} = \pi^{-,-}_\varphi,
\]
where $\pi^{-,-}_\varphi$ is the irreducible genuine square-integrable representation of $\Mp(W_2)$ with $L$-parameter $\varphi = (\chi_a \boxtimes S_2) \oplus (\chi_b \boxtimes S_2)$ (relative to $\psi$) associated to $(-1,-1) \in \widehat{S}_\varphi = \mu_2 \times \mu_2$.
\item 
If $\chi_a = \chi_b$, then we have
\[
 \pi^{-,-} = J_{P_1, \psi}(\chi_a |\cdot|^{\frac{1}{2}}, \widetilde{\st}_{\chi_a, \psi}).
\]
\end{enumerate}
\end{lem}

\begin{proof}
Put $\epsilon' = \chi_{ab}(-1)$.
Let $\sigma^{-, \epsilon'}$ be the $\epsilon'$-extension of $\chi_{ab} \circ \mathrm{N}_{B^-}$ to $\O(V_1^-)$.
Since $\pi^{-,-} = \theta_{W_2, V_1^-, \psi_a}(\sigma^{-, \epsilon'})$, the lemma follows from Lemmas \ref{l:ind-princ-1} and \ref{l:more-theta-nonarch-1}.
\end{proof}

The following table describes the representations in $\Pi_{\phi,\psi}(\Mp(W_2))$ when $\phi$ is nontempered and any irreducible summand of $\phi$ is symplectic.

\begin{center}
{\renewcommand{\arraystretch}{1.5}
\begin{tabular}{ll} \hline
 $\phi:$ $A$-parameter &
 $\Pi_{\phi, \psi}(\Mp(W_2))$ \\ \hline \hline
 $\phi = \chi_a \boxtimes S_4$
 & $\pi^+ = J_{B, \psi}(\chi_a |\cdot|^{\frac{3}{2}}, \chi_a|\cdot|^{\frac{1}{2}})$ \\
 & $\pi^- = J_{P_1, \psi}(\chi_a |\cdot|^{\frac{3}{2}}, \omega_{W_1,\psi_a}^-)$ \\ \hline 
 $\phi = (\rho_0 \boxtimes S_1) \oplus (\chi_a \boxtimes S_2)$ with $\rho_0 \ne \chi_a \boxtimes S_2$
 & $\pi^{\epsilon_1,+} = J_{P_1, \psi}(\chi_a |\cdot|^{\frac{1}{2}}, \pi_0^{\epsilon_1})$ \\
 & $\pi^{\epsilon_1, -} = \pi^{\epsilon_1, -}_\varphi$ \\ \hline 
 $\phi = (\rho_0 \boxtimes S_1) \oplus (\chi_a \boxtimes S_2)$ with $\rho_0 = \chi_a \boxtimes S_2$
 & $\pi^{\epsilon_1,+} = J_{P_1, \psi}(\chi_a |\cdot|^{\frac{1}{2}}, \pi_0^{\epsilon_1})$ \\
 & $\pi^{\epsilon_1, -} =
\begin{cases}
 0 & \text{if $\chi_a \ne 1$ and $\epsilon_1 = 1$} \\
 \pi_\varphi^{-,-} & \text{if $\chi_a \ne 1$ and $\epsilon_1 = -1$} \\
 \pi_\varphi^{+,+} & \text{if $\chi_a = 1$ and $\epsilon_1 = 1$} \\
 0 & \text{if $\chi_a = 1$ and $\epsilon_1 = -1$}
\end{cases}
$ \\ \hline 
 $\phi = (\chi_a \boxtimes S_2) \oplus (\chi_b \boxtimes S_2)$
 & $\pi^{+,+} = J_{B, \psi}(\chi_a |\cdot|^{\frac{1}{2}}, \chi_b |\cdot|^{\frac{1}{2}})$ \\ 
 & $\pi^{+,-} = J_{P_1, \psi}(\chi_a |\cdot|^{\frac{1}{2}}, \omega_{W_1,\psi_b}^-)$ \\ 
 & $\pi^{-,+} = J_{P_1, \psi}(\chi_b |\cdot|^{\frac{1}{2}}, \omega_{W_1,\psi_a}^-)$ \\
 & $\pi^{-,-} = \pi^{-,-}_\varphi$ \\ \hline
 $\phi = (\chi_a \boxtimes S_2) \oplus (\chi_a \boxtimes S_2)$ 
 & $\pi^{+,+} = J_{B, \psi}(\chi_a |\cdot|^{\frac{1}{2}}, \chi_a |\cdot|^{\frac{1}{2}})$ \\ 
 & $\pi^{-,-} = J_{P_1, \psi}(\chi_a |\cdot|^{\frac{1}{2}}, \widetilde{\st}_{\chi_a, \psi})$ \\ \hline
 $\phi = \rho \boxtimes S_2$ & 
 $\pi^+ = J_{P_2,\psi}(\tau \otimes |\det|^{\frac{1}{2}})$
 \\
 & $\pi^- = \pi^-_\varphi$ \\ \hline 
\end{tabular}
}
\end{center}

\begin{itemize}[leftmargin=20mm]
\item $a, b \in F^\times$ such that $\chi_a \ne \chi_b$
\item $\rho_0:$ $2$-dimensional irreducible symplectic representation of $L_F$
\item $\rho:$ $2$-dimensional irreducible orthogonal representation of $L_F$
\item $\varphi = \phi \circ \Delta:$ $4$-dimensional symplectic representation of $L_F$
\item $\Delta : W_F \times \SL_2(\C) \rightarrow W_F \times \SL_2(\C) \times \SL_2(\C) :$ diagonal map
\item $\Pi_{\rho_0, \psi}(\Mp(W_1)) = \{ \pi_0^\epsilon \, | \, \epsilon \in \mu_2 \}$
\item $\Pi_\rho(\GL_2(F)) = \{ \tau \}$
\end{itemize}

\subsection{The real case}

Suppose that $F = \R$.
We retain the notation of \S \ref{sss:notation-real}.

\subsubsection{(Limit of) discrete series representations of $\Mp_4(\R)$}
\label{sss:temp-mp4-real}

Recall from \cite[Theorem 8.1]{vogan00} that there is a bijection
\[
\begin{tikzcd}
 \{ \text{irreducible genuine tempered representations of $\Mp_4(\R)$ with real infinitesimal character} \} \arrow[d] \\
 \{ \text{irreducible genuine representations of $K'$} \} \arrow[u]
\end{tikzcd}
\]
sending $\pi$ in the first set (see \cite[Definition 8.5]{vogan00} for the precise definition) to the unique lowest $K'$-type of $\pi$, where $K'$ is the maximal compact subgroup of $\Mp_4(\R)$.
In particular, any genuine (limit of) discrete series representation of $\Mp_4(\R)$ (which always has real infinitesimal character) is uniquely determined by its lowest $K'$-type.

We now describe the $L$-packet $\Pi_{\varphi,\psi}(\Mp_4(\R))$ explicitly for any $L$-parameter $\varphi : W_\R \rightarrow \Sp_4(\C)$ when any irreducible summand of $\varphi$ is symplectic.
We may write
\[
 \varphi = \mathcal{D}_a \oplus \mathcal{D}_b
\]
with some $a, b \in \frac{1}{2} + \Z$ such that $a \ge b > 0$ and 
\[
 \Pi_{\varphi, \psi}(\Mp_4(\R)) = \{ \pi^{\epsilon_1,\epsilon_2} \, | \, \epsilon_1, \epsilon_2 \in \mu_2 \},
\]
where we interpret $\pi^{\epsilon_1,\epsilon_2}$ as zero if $a = b$ and $\epsilon_1 \ne \epsilon_2$.
If $a > b$, then by \cite[\S 2.2]{luo}, $\pi^{\epsilon_1,\epsilon_2}$ is the genuine discrete series representation of $\Mp_4(\R)$ with lowest $K'$-type $\Lambda^{\epsilon_1,\epsilon_2}$ (relative to $\psi$) given by
\begin{align*}
 \Lambda^{+,+} & = (a+1, -b), \\
 \Lambda^{+,-} & = (a+1,b+2), \\
 \Lambda^{-,+} & = (-b-2,-a-1), \\
 \Lambda^{-,-} & = (b, -a-1).
\end{align*}
If $a=b$, then $\pi^{\epsilon_1,\epsilon_2}$ is the genuine limit of discrete series representation of $\Mp_4(\R)$ with lowest $K'$-type $\Lambda^{\epsilon_1, \epsilon_2}$ (relative to $\psi$) given by
\begin{align*}
 \Lambda^{+,+} & = (a+1, -a), \\
 \Lambda^{-,-} & = (a, -a-1), 
\end{align*}
so that 
\[
 I_{P_2, \psi}(D_a) = \pi^{+,+} \oplus \pi^{-,-}.
\]
The representations described above exhaust all genuine (limit of) discrete series representations of $\Mp_4(\R)$.

\subsubsection{Nontempered $A$-packets}

For any $A$-parameter $\phi : W_\R \times \SL_2(\C) \rightarrow \Sp_4(\C)$, we have described most of the representations in the $A$-packet $\Pi_{\phi,\psi}(\Mp_4(\R))$ explicitly in the body of this paper.
The following lemmas determine the remaining representations.

\begin{lem}
Suppose that $\phi = (\rho \boxtimes S_1) \oplus (\chi_a \boxtimes S_2)$ with a $2$-dimensional symplectic almost tempered representation $\rho$ of $W_\R$ and $a \in \R^\times$.
\begin{enumerate}
\item 
If $\rho = \mathcal{D}_{\kappa - \frac{1}{2}}$ for some positive integer $\kappa$, then for any $\epsilon_1 \in \mu_2$, we have
\[
 \pi^{\epsilon_1, -} =
 \begin{cases}
  \pi_\varphi^{\epsilon_1, \chi_a(-1)} & \text{if $\kappa>1$ or $\epsilon_1 = \chi_a(-1)$;} \\
  0 & \text{if $\kappa=1$ and $\epsilon_1 = -\chi_a(-1)$,}
 \end{cases}
\]
where $\pi_\varphi^{\epsilon_1, \chi_a(-1)}$ is the genuine (limit of) discrete series representation of $\Mp_4(\R)$ with $L$-parameter $\varphi = \mathcal{D}_{\kappa - \frac{1}{2}} \oplus \mathcal{D}_{\frac{1}{2}}$ (relative to $\psi$) associated to $(\epsilon_1, \chi_a(-1)) \in \widehat{S}_\varphi \subset \mu_2 \times \mu_2$.
\item 
If $\rho = \chi |\cdot|^s \oplus \chi^{-1} |\cdot|^{-s}$ for some unitary character $\chi$ of $\R^\times$ and some $s \in \R$ with $0 \le s < \frac{1}{2}$, then we have
\[
 \pi^{+,-} = J_{P_1,\psi}(\chi |\cdot|^s, \omega^-_{W_1,\psi_a}).
\]
\end{enumerate}
\end{lem}

\begin{proof}
For $\epsilon_1 \in \mu_2$, put $\epsilon = \epsilon_1 \cdot \epsilon(\tfrac{1}{2},\rho) \cdot \epsilon(\tfrac{1}{2},\rho \times \chi_a) \cdot \chi_a(-1)$ and $\epsilon' = - \epsilon_1 \cdot \epsilon(\tfrac{1}{2},\rho) \cdot \chi_a(-1)$.
Let $\sigma_0^\epsilon$ be the irreducible representation of $(B^\epsilon)^\times$ with $L$-parameter $\rho \otimes \chi_a$ and $\sigma^{\epsilon,\epsilon'}$ the $\epsilon'$-extension of $\sigma_0^\epsilon$ to $\O(V_1^\epsilon)$.
Since $\pi^{\epsilon_1, -} = \theta_{W_2,V_1^{\epsilon}, \psi_a}(\sigma^{\epsilon, \epsilon'})$, the lemma follows from Lemmas \ref{l:o3mp4-nonzero}, \ref{l:more-theta-real-1}, and \ref{l:more-theta-real-2}.
\end{proof}

\begin{lem}
Suppose that $\phi = (\chi_a \boxtimes S_2) \oplus (\chi_b \boxtimes S_2)$ with $a,b \in \R^\times$.
\begin{enumerate}
\item 
If $\chi_a \ne \chi_b$, then we have
\[
 \pi^{-,-} = 0.
\]
\item
If $\chi_a = \chi_b$, then we have
\[
 \pi^{-,-} = J_{P_1, \psi}(\chi_a |\cdot|^{\frac{1}{2}}, (\omega^-_{W_1,\psi_a})^\vee).
\]
\end{enumerate}
\end{lem}

\begin{proof}
Put $\epsilon' = \chi_{ab}(-1)$.
Let $\sigma^{-, \epsilon'}$ be the $\epsilon'$-extension of $\chi_{ab} \circ \mathrm{N}_{B^-}$ to $\O(V_1^-)$.
Since $\pi^{-,-} = \theta_{W_2, V_1^-, \psi_a}(\sigma^{-, \epsilon'})$, the lemma follows from Lemmas \ref{l:o3mp4-nonzero} and \ref{l:ind-princ-1}.
\end{proof}

The following table describes the representations in $\Pi_{\phi,\psi}(\Mp_4(\R))$ when $\phi$ is nontempered and any irreducible summand of $\phi$ is symplectic.

\begin{center}
{\renewcommand{\arraystretch}{1.5}
\begin{tabular}{ll} \hline
 $\phi:$ $A$-parameter &
 $\Pi_{\phi, \psi}(\Mp_4(\R))$ \\ \hline \hline
 $\phi = \chi_a \boxtimes S_4$
 & $\pi^+ = J_{B, \psi}(\chi_a |\cdot|^{\frac{3}{2}}, \chi_a|\cdot|^{\frac{1}{2}})$ \\
 & $\pi^- = J_{P_1, \psi}(\chi_a |\cdot|^{\frac{3}{2}}, \omega_{W_1,\psi_a}^-)$ \\ \hline 
 $\phi = (\mathcal{D}_{\kappa - \frac{1}{2}} \boxtimes S_1) \oplus (\chi_a \boxtimes S_2)$
 & $\pi^{\epsilon_1,+} = J_{P_1, \psi}(\chi_a |\cdot|^{\frac{1}{2}}, \widetilde{D}_{\lambda,\psi})$ \\
 & $\pi^{\epsilon_1, -} =
 \begin{cases}
  \pi_\varphi^{\epsilon_1, \chi_a(-1)} & \text{if $\kappa>1$ or $\epsilon_1 = \chi_a(-1)$} \\
  0 & \text{if $\kappa=1$ and $\epsilon_1 = -\chi_a(-1)$}
 \end{cases}
 $ \\ \hline 
 $\phi = (\chi_a \boxtimes S_2) \oplus (\chi_b \boxtimes S_2)$
 & $\pi^{+,+} = J_{B, \psi}(\chi_a |\cdot|^{\frac{1}{2}}, \chi_b |\cdot|^{\frac{1}{2}})$ \\ 
 & $\pi^{+,-} = J_{P_1, \psi}(\chi_a |\cdot|^{\frac{1}{2}}, \omega_{W_1,\psi_b}^-)$ \\ 
 & $\pi^{-,+} = J_{P_1, \psi}(\chi_b |\cdot|^{\frac{1}{2}}, \omega_{W_1,\psi_a}^-)$ \\
 & $\pi^{-,-} = 0$ \\ \hline
 $\phi = (\chi_a \boxtimes S_2) \oplus (\chi_a \boxtimes S_2)$ &
 $\pi^{+,+} = J_{B, \psi}(\chi_a |\cdot|^{\frac{1}{2}}, \chi_a |\cdot|^{\frac{1}{2}})$ \\ 
 & $\pi^{-,-} = J_{P_1, \psi}(\chi_a |\cdot|^{\frac{1}{2}}, (\omega^-_{W_1,\psi_a})^\vee)$ \\ \hline 
 $\phi = \mathcal{D}_\kappa \boxtimes S_2$
 & $\pi^+ = J_{P_2,\psi}(D_\kappa \otimes |\det|^{\frac{1}{2}})$ \\
 & $\pi^- = \pi_{\varphi}^{+,-} \oplus \pi_{\varphi}^{-,+}$ \\ \hline
\end{tabular}
}
\end{center}

\begin{itemize}[leftmargin=25mm]
 \item $a, b \in \R^\times$ such that $\chi_a \ne \chi_b$
 \item $\kappa:$ positive integer
 \item $\lambda =
 \begin{cases}
  \kappa-\frac{1}{2} & \text{if $\epsilon_1 = 1$} \\
  -\kappa+\frac{1}{2} & \text{if $\epsilon_1 = -1$}
 \end{cases}
 $
 \item $\varphi = 
 \begin{cases}
 \mathcal{D}_{\kappa - \frac{1}{2}} \oplus \mathcal{D}_{\frac{1}{2}}
 & \text{if $\phi = (\mathcal{D}_{\kappa - \frac{1}{2}} \boxtimes S_1) \oplus (\chi_a \boxtimes S_2)$} \\
 \mathcal{D}_{\kappa + \frac{1}{2}} \oplus \mathcal{D}_{\kappa - \frac{1}{2}}
 & \text{if $\phi = \mathcal{D}_\kappa \boxtimes S_2$}
 \end{cases}
 $
\end{itemize}

\subsection{The complex case}

Suppose that $F = \C$.
For any $A$-parameter $\phi : W_\C \times \SL_2(\C) \rightarrow \Sp_4(\C)$, we have described the representations in the $A$-packet $\Pi_{\phi,\psi}(\Mp_4(\C))$ explicitly in the body of this paper.
The following table describes the representations in $\Pi_{\phi,\psi}(\Mp_4(\C))$ when $\phi$ is nontempered and any irreducible summand of $\phi$ is symplectic.

\begin{center}
{\renewcommand{\arraystretch}{1.5}
\begin{tabular}{ll} \hline
 $\phi:$ $A$-parameter &
 $\Pi_{\phi, \psi}(\Mp_4(\C))$ \\ \hline \hline
 $\phi = 1 \boxtimes S_4$ &
 $\pi^+ = J_{B, \psi}(|\cdot|^{\frac{3}{2}}, |\cdot|^{\frac{1}{2}})$ \\
 & $\pi^- = J_{P_1, \psi}(|\cdot|^{\frac{3}{2}}, \omega_{W_1,\psi}^-)$ \\ \hline
 $\phi = (1 \boxtimes S_2) \oplus (1 \boxtimes S_2)$ &
 $\pi^{+,+} = J_{B, \psi}(|\cdot|^{\frac{1}{2}}, |\cdot|^{\frac{1}{2}})$ \\
 & $\pi^{-,-} = 0$ \\ \hline
\end{tabular}
}
\end{center}

\end{document}